\documentclass[12pt]{amsart} 
\usepackage{amsmath,amssymb,amscd,amsthm}
\usepackage{latexsym}
\usepackage{graphicx}
\usepackage[english]{babel}
\usepackage[latin1]{inputenc}       
\usepackage{pgf,pgfarrows,pgfnodes,pgfautomata,pgfheaps}
\usepackage{colortbl}
\usepackage{setspace}

\usepackage{hyperref}

\newtheorem{thmx}{Theorem}

\newtheorem{conjecture}{Conjecture}
\newtheorem{question}{Question}

\newtheorem{theorem}{Theorem}[section]
\newtheorem{cor}[theorem]{Corollary}
\newtheorem{proposition}[theorem]{Proposition}
\newtheorem{lemma}[theorem]{Lemma}

\newtheorem{remark}[theorem]{Remark}

\def\irr#1{{\rm Irr}(#1)}
\def\irrr#1#2 {\irr {#1 \mid #2}}
\newcommand{\R}{\mathbb R}

\newcommand{\sfe}{{{\mathbb S}^{n-1}}}

\newcommand{\E}{\mathbb E}

\begin{document}

\title[On a conjectural symmetric version of Ehrhard's inequality]{On a conjectural symmetric version of Ehrhard's inequality}
\author[Galyna V. Livshyts]{Galyna V. Livshyts}
\address{School of Mathematics, Georgia Institute of Technology, Atlanta, GA} \email{glivshyts6@math.gatech.edu}

\subjclass[2010]{Primary: 52} 
\date{\today}
\begin{abstract} 

We formulate a plausible conjecture for the optimal Ehrhard-type inequality for convex symmetric sets with respect to the Gaussian measure. Namely, letting $J_{k-1}(s)=\int^s_0 t^{k-1} e^{-\frac{t^2}{2}}dt$ and $c_{k-1}=J_{k-1}(+\infty)$, we conjecture that the function $F:[0,1]\rightarrow\R,$ given by
$$F(a)= \sum_{k=1}^n 1_{a\in E_k}\cdot(\beta_k J_{k-1}^{-1}(c_{k-1} a)+\alpha_k)$$
(with an appropriate choice of a decomposition $[0,1]=\cup_{i} E_i$ and coefficients $\alpha_i, \beta_i$) satisfies, for all symmetric convex sets $K$ and $L,$ and any $\lambda\in[0,1]$,
$$
F\left(\gamma(\lambda K+(1-\lambda)L)\right)\geq \lambda F\left(\gamma(K)\right)+(1-\lambda) F\left(\gamma(L)\right).
$$
We explain that this conjecture is ``the most optimistic possible'', and is equivalent to the fact that for any symmetric convex set $K,$ its \emph{Gaussian concavity power} $p^s(K,\gamma)$ is greater than or equal to $p_s(RB^k_2\times \R^{n-k},\gamma),$ for some $k\in \{1,...,n\}$. We call the sets $RB^k_2\times \R^{n-k}$ \emph{round $k$-cylinders}; they also appear as the conjectured Gaussian isoperimetric minimizers for symmetric sets, see Heilman \cite{Heilman}. 

In this manuscript, we make progress towards this question, and show that for any symmetric convex set $K$ in $\R^n,$
$$p_s(K,\gamma)\geq \sup_{F\in L^2(K,\gamma)\cap Lip(K):\,\int F=1} \left(2T_{\gamma}^F(K)-Var(F)\right)+\frac{1}{n-\E X^2},$$
where $T_{\gamma}^F(K)$ is the $F-$torisional rigidity of $K$ with respect to the Gaussian measure. \emph{Moreover, the equality holds if and only if $K=RB^k_2\times \R^{n-k}$ for some $R>0$ and $k=1,...,n.$} As a consequence, we get
$$p_s(K,\gamma)\geq Q(\E |X|^2, \E \|X\|_K^4, \E \|X\|^2_K, r(K)),$$
where $Q$ is a certain rational function of degree $2$, the expectation is taken with respect to the restriction of the Gaussian measure onto $K$, and $r(K)$ is the in-radius of $K$. The result follows via a combination of some novel estimates, the $L2$ method (previously studied by several authors, notably Kolesnikov, Milman \cite{KM1}, \cite{KM2}, \cite{KolMil}, \cite{KolMilsupernew}, the author \cite{KolLiv}, Hosle \cite{HKL}, Colesanti \cite{Col1}, Marsiglietti \cite{CLM}, Eskenazis, Moschidis \cite{EM}), and the analysis of the Gaussian torsional rigidity. 

As an auxiliary result on the way to the equality case characterization, we characterize the equality cases in the ``convex set version'' of the Brascamp-Lieb inequality, and moreover, obtain a quantitative stability version in the case of the standard Gaussian measure; this may be of independent interest. All the equality case characterizations rely on the careful analysis of the smooth case, the stability versions via trace theory, and local approximation arguments.

In addition, we provide a non-sharp estimate for a function $F$ whose composition with $\gamma(K)$ is concave in the Minkowski sense for all symmetric convex sets.

\end{abstract}
\maketitle

\pagebreak

\begin{spacing}{0.9}
\tableofcontents
\end{spacing}

\section{Introduction}

Consider the standard Gaussian measure $\gamma$ on the $n$-dimensional euclidean space $\R^n$, that is, the measure with density $d\gamma(x)=\frac{1}{\sqrt{2\pi}^n} e^{-\frac{|x|^2}{2}}.$ We shall use notation 
$$\Phi(x)=\frac{1}{\sqrt{2\pi}}\int_{-\infty}^x e^{-\frac{s^2}{2}}ds,$$
for the cdf of the 1-dimensional Gaussian measure. 

One of the beautiful and useful geometric properties of the Gaussian measure is the Ehrhard inequality \cite{Ehr}, \cite{Ehr2}, extended by Borell \cite{bor-ehr} to general sets: for a pair of closed Borel-measurable sets $K$ and $L$ and for $\lambda\in [0,1]$,
\begin{equation}\label{ehrhard}
\Phi^{-1}\left(\gamma(\lambda K+(1-\lambda)L)\right)\geq \lambda\Phi^{-1}\left(\gamma(K)\right)+(1-\lambda) \Phi^{-1}\left(\gamma(L)\right).
\end{equation}
See also Bobkov \cite{bobkov} for a celebrated related inequality, as well as Lata\l{}a \cite{Lat-ehr}, van Handel \cite{vH}, Shenfeld \cite{ShvH} (for the equality cases), Neeman, Paouris \cite{NeePao}, Ivanisvili \cite{Paata}, Paouris, Valettas \cite{PaoVal} for alternative proofs, generalizations, applications, and further discussions.

The Gaussian isoperimetric inequality, proved by Borell \cite{Bor-isop} and Sudakov-Tsirelson \cite{ST}, states that for every $a\in [0,1],$ the half-space $H_a$ with $\gamma(H_a)=a$ has the smallest Gaussian perimeter among all measurable sets in $\R^n$ of Gaussian measure $a$. It is well-known that the Ehrhard inequality implies the Gaussian isoperimetric inequality (see e.g. Lata\l{}a \cite{Lat} or Section 3.3 below). 

The analogue of the isoperimetric problem for the Gaussian measure in the case of symmetric sets was asked by Barthe \cite{barthe} and O'Donnelll \cite{Odon}, conjectured by Morgan and studied by Heilman \cite{Heilman}: it is believed that for every $a\in [0,1]$ there is $k\in \{1,...,n\}$ such that for every Borel-measurable set $K$ with $\gamma(K)=a,$ one has either $\gamma^+(\partial K)\geq \gamma^+(\partial C_k(a))$ or $\gamma^+(\partial K)\geq \gamma^+(\partial C_k(1-a))$, where $C_k(a)=RB^k_2\times \R^{n-k}$ is a \emph{round $k$-cylinder} such that $\gamma(C_k(a))=a$. See Heilman \cite{Heilman}, \cite{Heilman-notes} for more details, and partial progress on this quesiton. Barchiesi and Julin \cite{italians} showed that when $\gamma(K)\in [\alpha,1]$ for a sufficiently large $\alpha>0,$ one has $\gamma^+(\partial K)\geq \gamma^+(\partial C_1(a))$; in other words, the Gaussian perimeter is minimized on symmetric strips whenever the Gaussian measure of a set is sufficiently large.

One might wonder what is an ``optimal'' function $F:[0,1]\rightarrow\R,$ for which the inequality
\begin{equation}\label{ehrhard-general}
F\left(\gamma(\lambda K+(1-\lambda)L)\right)\geq \lambda F\left(\gamma(K)\right)+(1-\lambda) F\left(\gamma(L)\right)
\end{equation}
holds for all symmetric convex sets $K$ and $L$? One should be able to improve upon $F=\Phi^{-1}$ since the Ehrhard inequality is never sharp for symmetric sets \cite{ShvH} unless they coincide. An initial naive guess may be that $F=\varphi^{-1}$ could work, where $\varphi(t)=\frac{1}{\sqrt{2\pi}}\int_{-t}^t e^{-\frac{s^2}{2}} ds,$ -- ``the symmetric version'' of the Gaussian cdf. However, it was pointed to the author by Liran Rotem (after a discussion) a few years ago, that a numerical computation shows that this naive conjecture fails. Later, the author have discussed this with several other experts and learned that others had also done a similar simulation, reaching the same conclusion. In fact, one may notice that $F=\varphi^{-1}$ certainly fails (\ref{ehrhard-general}) on the entire $[0,1]$, because if it didn't, that would imply that the strips are Gaussian isoperimetric minimizers among symmetric convex sets of any measure (more details of this implication shall be discussed in Section 3), and this is not the case. 

A less naive attempt to put forward a conjecture for an $F$ satisfying (\ref{ehrhard-general}) could be 
\begin{equation}\label{bad-func}
F(a)=\sum_{k=1}^n 1_{a\in I_k}\cdot( J_{k-1}^{-1}(c_{k-1} a)+a_k),
\end{equation}
where $I_k$ is the sub-interval of $[0,1]$ where the surface area of the $k$-round cylinder $C_k(a)$ is minimal (over all choices of $k$). Here
$$J_{k-1}(R)=\int_0^R s^{k-1}e^{-\frac{s^2}{2}} ds= \int_0^R g_{k-1}(s)ds,$$
$c_{k-1}=J_{k-1}(+\infty),$ and $a_k$ are chosen in such a way that $F$ is continuous (and then automatically continuously differentiable, due to the properties of the special functions involved). Note the geometric meaning of the function $J_{k-1}^{-1}(c_{k-1} a)$: it is the radius of the $k-$round cylinder of measure $a,$ or in other words, $\gamma\left( J_{k-1}^{-1}(c_{k-1} a) B^k_2\times \R^{n-k}\right)=a$. However, unfortunately, this function does not satisfy (\ref{ehrhard-general}) for symmetric convex sets, as shall be shown in Section 3.

On the bright side, one may explicitly conjecture ``the optimal'' plausible function satisfying (\ref{ehrhard-general}), which is somewhat similar to (\ref{bad-func}). 

\begin{conjecture}\label{theconj}
With the above notation for $c_p, g_p, J_p$, the function
$$F(a)=\int_0^a \exp\left(\int_{C_0}^{t} \min_{k=1,...,n} \left(\frac{c_{k-1} (-k+1+J^{-1}_{k-1}(c_{k-1}s)^2)}{g_{k}\circ J^{-1}_{k-1}(c_{k-1}s)}\right)ds\right) dt$$
satisfies (\ref{ehrhard-general}) for all symmetric convex sets. Here the choice of $C_0\in (0,1)$ is arbitrary.

Equivalently, there exists a collection of disjoint sets $E_k\subset [0,1],$  $k=1,...,n,$ such that $[0,1]=\cup_{k=1}^n E_k,$ and there exist coefficients $\alpha_k, \beta_k\in\R$, such that the function
$$F(a)= \sum_{k=1}^n 1_{a\in E_k}\cdot(\beta_k J_{k-1}^{-1}(c_{k-1} a)+\alpha_k)$$ 
satisfies (\ref{ehrhard-general}) for all symmetric convex sets, and is of class $C^2$ and increasing. 

Moreover, the equality in (\ref{ehrhard-general}) occurs if and only if there exists a $k\in\{1,...,n\}$ such that $K=C_k(a_1)$ and $L=C_k(a_2)$, for some $a_1, a_1\in E_k.$
\end{conjecture}

The fact that the expressions above are equivalent shall be explained in Section 3.

We believe that the above conjecture is indeed plausible; we shall also explain in Section 3 that this conjecture is, in a sense, ``the most optimistic possible'', and, of course, is stronger then the Ehrhard inequality.

\begin{remark}
The sets $E_k$ in the conjecture above are such that for $a\in E_k,$ the function $\left(\log \frac{1}{s_k(a)}\right)'$	is minimal among all $k,$ where we denote $s_k(a)=\gamma^+(\partial C_k(a))$. It is worth emphasizing that the sets $E_k$ are not the same as the intervals $I_k$ (which were defined so that for each $a\in I_k,$ the function $s_k(a)$ is minimal, among all $k$.) 
\end{remark}

We shall now move to discussing how we managed to guess Conjecture \ref{theconj}, and state progress towards it. For a convex set $K$ and a measure $\mu,$ we define 
$$p_s(K,\mu):=\limsup_{\epsilon\rightarrow 0}\left\{p:\,\forall L\in \mathcal{K}_s,\, \mu((1-\epsilon) K+\epsilon L)^{p}\geq (1-\epsilon) \mu(K)^{p}+\epsilon\mu(L)^{p}\right\}.$$
Here $\mathcal{K}_s$ stands for the set of symmetric convex sets. We shall focus on the case when $\mu$ is the standard Gaussian measure $\gamma.$ We make the following

\begin{conjecture}\label{theconj1}
Pick $a\in [0,1]$ and suppose $K$ is a symmetric convex set with $\gamma(K)=a$. Then 
$$p_s(K,\gamma)\geq \min_{k=1,...,n} p_s(C_k(a),\gamma),$$
with equality if and only if $K=C_k(a)$ for some (appropriate) $k=1,...,n$ (recall that we use notation $C_k(a)$ for round $k$-cylinders $R B^k_2\times \R^{n-k}$ of measure $a$.)
\end{conjecture}
In Section 3, we shall see that
\begin{proposition}\label{key-connection}
Conjecture \ref{theconj} and Conjecture \ref{theconj1} are equivalent.	
\end{proposition}



In view of the Proposition \ref{key-connection}, it is of interest to study lower estimates for $p_s(K,\gamma).$ The main result of this paper is the sharp estimate stated below. 

For a convex domain $K$ in $\R^n$, and a Lipschitz function $F: K\rightarrow\R,$ define the $F-$Gaussian torsional rigidity by
$$T^F_{\gamma}(K)=\sup_{v\in W^{1,2}(K,\gamma):\,v|_{\partial K}=0} \frac{\left(\int Fv\right)^2}{\int |\nabla v|^2}.$$
In Section 5 we discuss this object in more detail; it appears to be an important tool, relevant for the questions which we study here.

\begin{thmx}\label{Gauss-main-1}
For any symmetric convex set $K$ in $\R^n,$
$$p_s(K,\gamma)\geq \sup_{F\in L^2(K,\gamma)\cap Lip(K):\,\int F=1} \left(2T_{\gamma}^F(K)-Var(F)\right)+\frac{1}{n-\E X^2}.$$

Moreover, the equality holds if and only if $K=RB^k_2\times \R^{n-k}$ for some $R>0$ and $k=1,...,n.$

Here $\E$ and $Var$ stand for the expectation and the variance with respect to the restriction of the standard Gaussian measure onto $K.$
\end{thmx}

As a corollary, we will deduce (see Section 6 for the precise formulation):

\begin{cor}\label{Gauss-main-intro}
For any symmetric convex set $K$ in $\R^n,$
$$p_s(K,\gamma)\geq Q(\E |X|^2, \E \|X\|_K^4, \E \|X\|^2_K, r(K)),$$
where $Q$ is a certain rational function of degree $2$, the expectation is taken with respect to the restriction of the Gaussian measure onto $K$, and $r(K)$ is the in-radius of $K$.

Moreover, the equality holds if and only if $K=RB^k_2\times \R^{n-k}$ for some $R>0$ and $k=1,...,n.$
\end{cor}

The key feature of the Theorem \ref{Gauss-main-1} is \emph{the equality case characterization.} To further emphasize it, we outline separately the exact value of $p_s(K,\gamma)$ for the case when $K$ is a round $k-$cylinder.

\begin{proposition}[Case of cylinders]\label{cyl}
When $K=RB_2^{k}\times \R^{n-k}$, with $\gamma(K)=a,$ for some $a\in [0,1]$, and for $k=1,...,n,$ we have
$$p_s(K,\gamma)= 1-\frac{c_{k-1} a(k-1-J^{-1}_{k-1}(c_{k-1}a)^2)}{g_{k}\circ J^{-1}_{k-1}(c_{k-1}a)}.$$
\end{proposition}

There is nothing special in searching for $F$ in the form $F(a)=a^{p(a)}$ for the function $F$ from (\ref{ehrhard-general}), of course. In Section 3, we shall see some other equivalent formulations of Conjecture \ref{theconj}. 

Most of this manuscript is dedicated to proving Theorem \ref{Gauss-main-1}. In addition to several novel estimates and ideas, our method is based on the reduction of the question to infinitesimal version, previously explored by Colesanti, Hug, Saorin-Gomez, \cite{Col1}, \cite{Col2}, \cite{Colesanti-Hug-Saorin2}, L, Marsiglietti, \cite{CLM}, \cite{CL}, and on the $L2$ method, studied by Kolesnikov, Milman, \cite{KM1}, \cite{KM2}, \cite{KolMil}, \cite{KolMilsupernew} (these works include an important advancement -- the new proof of the Brunn-Minkowski inequality via the L2-method), L \cite{KolLiv}, \cite{KolLiv1}, Hosle \cite{HKL}, Milman \cite{Milm-new}. Some of the estimates also involve optimizing quadratic inequalities, see a remarkable work of Eskenazis, Moschidis \cite{EM}, or Remark 2.2 from Hosle, Kolesnikov, L \cite{HKL}. Another key component of the proof involves the sharp lower estimate for the \emph{Gaussian torsional rigidity} (more details in Section 5). The equality case characterization is done via careful analysis of the smooth case in several inequalities, several qualitative stability estimates obtained via the trace inequalities for convex sets (see subsection 2.4), and delicate approximation arguments. The stability version for the torsional energy lower bound from Section 6 may be of independent interest. 

In addition, one of the steps on the way to the equality case characterization in Theorem \ref{Gauss-main-1} is the stability in the ``convex set version'' of the Brascamp-Lieb inequality \cite{BrLi}: this celebrated inequality states that for any convex set $K$ in $\R^n$, any locally Lipschitz function $f\in L^2(\R^n)$, and any strictly convex function $V:\R^n\rightarrow\R,$ letting $d\mu(x)=e^{-V(x)}dx,$ we have
$$
\mu(K)\int_K f^2 d\mu-\left(\int_K f d\mu\right)^2
\leq\mu(K)\int_K \langle (\nabla^2 V)^{-1}\nabla f,\nabla f\rangle d\mu.
$$
In Section 4, we will show

\begin{theorem}[Equality case characterization in the ``convex set version of'' the Brascamp-Lieb inequality, and the quantitative stability in the Gaussian case]\label{eqchar-intro}

We state two results:

\begin{enumerate}
	\item Let $\mu$ be a log-concave measure on $\R^n$ with $C^2$ positive density $e^{-V}$, and suppose $\nabla^2 V>0.$ Then for any convex set $K$ and any function $f\in W^{1,2}(K)\cap C^1(K),$ we have, as per Brascamp and Lieb \cite{BrLi}, 
$$\mu(K)\int_K f^2 d\mu-\left(\int_K f d\mu\right)^2\leq\mu(K)\int_K \langle (\nabla^2 V)^{-1}\nabla f,\nabla f\rangle d\mu;$$
Moreover, the equality occurs if one of the two things happen: 
\begin{itemize}
\item $f=C$ for some constant $C\in\R;$ 
\item there exists a rotation $U$ such that\\ 
a) $UK=L\times \R^{n-k}$ for some $k=1,...,n$ and a $k-$dimensional convex set $L;$\\ 
b) $f\circ U=\langle \nabla V,\theta\rangle+C$, for some vector $\theta\in\R^{n-k}$ and some constant $C\in\R.$
\end{itemize} 

\medskip

\item

Moreover, if $\mu=\gamma$ (the standard Gaussian measure), and a convex set $K$ contains $rB^n_2,$ then assuming that for $\epsilon>0,$
$$\gamma(K)\int_K f^2 d\gamma-\left(\int_K f d\gamma\right)^2\geq\gamma(K)\int_K \langle (\nabla^2 V)^{-1}\nabla f,\nabla f\rangle d\gamma-\epsilon,$$
we conclude that there exists a vector $\theta\in\R^n$ (possibly zero), which depends only on $K$ and $f$, such that
\begin{itemize}
\item $\int_{\partial K} \langle \theta, n_x\rangle^2 d \gamma_{\partial K}\leq \frac{(n+1)\epsilon}{r};$
\item $\|f-\langle x,\theta\rangle-\frac{1}{\gamma(K)}\int_{K} f d\gamma\|_{L^1(K,\gamma)}\leq \sqrt{\gamma(K)}\left(\sqrt{n\epsilon}+\sqrt[4]{n\epsilon}\right).$
\end{itemize}

\end{enumerate}

\end{theorem}

A related question was posed by Brandolini, Chiacchio, Henrot, Trombetti \cite{strips-conj}. See Section 5 for the history and important past results in the case of the Gaussian measure. Part (2) of Theorem \ref{eqchar-intro} will be used directly in the proof of Theorem \ref{Gauss-main-1}. The quantitative stability works especially well in the Gaussian case thanks to nicer trace-type estimates for convex sets (see subsection 2.4 and particularly Theorem \ref{GaussGarg}). In Section 4 we will also deduce stability for general log-concave measures, but it is weaker in general, comparing to the case of the standard Gaussian measure. 

In Section 7, we will discuss the equality case characterization in a ``symmetric upgrade'' of the Gaussian Brascamp-Lieb inequality (by a symmetric upgrade we mean something more involved than the classical result of Cordero-Erasquin, Fradelizi and Maurey \cite{CFM}, which itself would not have vast equality cases). In Section 7 we will also discuss some curious upgrades of the Gaussian Brascamp-Lieb inequality for non-symmetric sets, and draw some connections with the Ehrhard inequality.

\medskip

Next, we will show an estimate somewhat weaker than Conjecture \ref{theconj}.

\begin{theorem}\label{Gauss}
The function
$$
F(a)=\int_0^a \exp\left(\int_{C_0}^t\left(\frac{\varphi^{-1}(s)^2}{4e^2n^2s}+\frac{1}{ns-\frac{1}{c_{n-1}}J_{n+1}\circ J^{-1}_{n-1}(c_{n-1}s)}-\frac{1}{s}\right)ds\right) dt
$$
satisfies 
$$
F\left(\gamma(\lambda K+(1-\lambda)L)\right)\geq \lambda F\left(\gamma(K)\right)+(1-\lambda) F\left(\gamma(L)\right),
$$
for all convex {symmetric} sets $K$ and $L$ and every $\lambda\in [0,1]$.
\end{theorem}

As a corollary of Theorem \ref{Gauss} we will show Proposition \ref{isoper} -- a Gaussian version of Minkowski's first inequality: for any pair of symmetric convex sets $K$ and $L$,
$$\gamma(K+tL)'|_{t=0}\geq \left(1-\frac{\E X^2}{n}\right)\gamma(K)^{1-\frac{1}{n-\E X^2}}\gamma(L)^{\frac{1}{n-\E X^2}},$$
where the expected value is taken with respect to the restriction of the Gaussian measure $\gamma$ onto $K.$

\medskip

In Section 2 we discuss preliminaries from PDE, Sobolev space theory, rearrangements, asymptotic analysis, and other relevant topics; notably, in subsection 2.4 we will derive useful trace inequalities for convex sets. In Section 3 we discuss Conjecture \ref{theconj}, its isoperimetric implications and relation to the S-inequality \cite{sconj}, connections to Conjecture \ref{theconj1}, prove Proposition \ref{cyl}, survey results of Kolesnikov and Milman \cite{KM1} regarding the $L2$ estimates, and perform other preparations. In Section 4 we prove Theorem \ref{eqchar-intro} (the equality case characterization of the Brascamp-Lieb inequality and the universal stability in the Gaussian case). In Section 5 we discuss energy minimization for the standard Gaussian measure, explore various bounds and properties of torsional rigidity. In Section 6 we prove Theorem \ref{Gauss-main-1}, Corollary \ref{Gauss-main-intro} as well as Theorem \ref{Gauss} and some of its consequences. In Section 7 we will discuss upgrades on the Brascamp-Lieb inequality, draw some connections with the Ehrhard inequality, and utilize our work from Section 6 in order to characterize the equality cases in the ``enchanted symmetric version'' of the Brascamp-Lieb inequality. 


\textbf{Acknowledgement.} The author is very grateful to Benjamin Jaye for many fruitful conversations and useful references to various PDE results concerning existence and convergence, not to mention his help with a computer simulation and creating the pictures from Section 3, all of which was in-disposable. The author is grateful to Alexander Kolesnikov for teaching her a lot of mathematics. The author also thanks Ramon van Handel for a nice discussion about Ehrhard's inequality, and for comments on the early version of this preprint. Thanks also to Steven Heilman, Francesco Chiacchio, Alexandros Eskenazis, Konstantin Tikhomirov, Marco Barchiesi for their comments and interest in the early version.

The author is supported by the NSF CAREER DMS-1753260. The author worked on this project while being a Research Fellow at the program in Probability, Geometry, and Computation in High Dimensions at the Simons Institute for the Theory of Computing. The paper was completed while the author was in residence at the Hausdorff Institute of Mathematics at the program in The Interplay between High-Dimensional Geometry and Probability.

\section{Preliminaries, and some new trace estimates for convex sets}

\subsection{The Brascamp-Lieb inequality} Given a measure $\mu$ on $\R^n$ and $p>0$, we shall consider $L^p(\mu, K)$, the space of functions $g$ such that $\int_K |g|^p d\mu<\infty.$ Recall that the Brascamp-Lieb inequality says that for any locally Lipschitz function $f\in L^2(\mu,\R^n)$ and any convex function $V:\R^n\rightarrow\R,$ we have
\begin{equation}\label{BrLi}
\int_{\R^n} f^2 d\mu-\left(\int_{\R^n} f d\mu\right)^2\leq\int_{\R^n} \langle (\nabla^2 V)^{-1}\nabla f,\nabla f\rangle d\mu,
\end{equation}
where $d\mu(x)=e^{-V(x)}dx$. Note that the integral on the right hand side makes sense in the almost everywhere sense. The function $e^{-V}$ is called log-concave when $V$ is convex. See Brascamp, Lieb \cite{BrLi}, or e.g. Bobkov, Ledoux \cite{BL1-BrLib}.

Recall that a set $K$ is called convex if together with every pair of points it contains the interval connecting them, and recall that the characteristic function of a convex set is log-concave. As a consequence of (\ref{BrLi}), for any convex body $K$, 
\begin{equation}\label{BrLi-convex}
\mu(K)\int_K f^2 d\mu-\left(\int_K f d\mu\right)^2\leq\mu(K)\int_K \langle (\nabla^2 V)^{-1}\nabla f,\nabla f\rangle d\mu.
\end{equation}
In the case of the standard Gaussian measure $\gamma$, this becomes, for any convex set $K,$
\begin{equation}\label{poinc-sect4}
\gamma(K)\int_K f^2 d\gamma-\left(\int_K f d\gamma\right)^2\leq\gamma(K)\int_K |\nabla f|^2 d\gamma.
\end{equation}
Furthermore, Cordero-Erasquin, Fradelizi and Maurey showed \cite{CFM} that for symmetric convex sets and even $f$,
\begin{equation}\label{poinc-sect4-sym}
\gamma(K)\int_K f^2 d\gamma-\left(\int_K f d\gamma\right)^2\leq\frac{1}{2}\gamma(K)\int_K |\nabla f|^2 d\gamma.
\end{equation}


\medskip

\subsection{Symmetrizations and Ehrhard's principle}

Let us recall Ehrhard's rearrangement \cite{Ehr}, the Gaussian analogue of the radial rearrangement. For a measurable set $K,$ denote by $H_K$ (sometimes also denoted $K^*$) the left half-space of the Gaussian measure equal to $\gamma(K)$. Namely,
$$H_K=\{x\in\R^n:\, x_1\leq \Phi^{-1}(\gamma(K))\},$$
where 
$$\Phi(t)=\frac{1}{\sqrt{2\pi}}\int^t_{-\infty} e^{-\frac{s^2}{2}}ds.$$
Next, for a measurable function $u:\R^n\rightarrow\R,$ consider its rearrangement $u^*$ to be the function whose level sets are half-spaces, and have the same Gaussian measures as the level sets of $u.$ Namely, $u^*(x)=t$ whenever $\Phi(x_1)=\gamma(\{u\leq t\})$. In other words, for all $t\in\R,$
$$\{u^*\leq t\}=H_{\{u\leq t\}}.$$
The Lebesgue analogue of Ehrhard's symmetrization, called Schwartz (radial) rearrangement, is an indispensable tool in PDE, see, for instance Burchard \cite{Almut}, Lieb, Loss \cite{LLbook}, Kesavan \cite{Kesavan}, or Vazquez \cite{Vasquez}; see also Carlen, Kerce \cite{CarKer} for a nice discussion related to Ehrhard's symmetrization; see also Bogachev \cite{bogachev}.

Let us recall Ehrhard's principle, which follows from the Gaussian isoperimetric inequality (for the proof, see, e.g. Carlen, Kerce \cite{CarKer} or Ehrhard \cite{Ehr}.) It is the Gaussian analogue of the Polya-Szeg\"o principle \cite{PolSz}.

\begin{lemma}[Ehrhard's principle]\label{EhrPr}
Let $K$ be a Borel measurable set in $\R^n$ and let $H_K$ be the left half-space of the same Gaussian measure as $K.$ Let $u: K\rightarrow\R$ be a locally Lipschitz function and let $u^*$ be its Ehrhard symmetral. Then for any convex increasing function $\varphi:\R^+\rightarrow\R^+$ we have
$$\int_K \varphi(|\nabla u|)d\gamma\geq \int_{H_K} \varphi(|\nabla u^*|)d\gamma.$$	
\end{lemma}

\begin{remark} The case when $\varphi(t)=|t|$ is particularly straight-forward, so we outline it for the reader's convenience. By the co-area formula,
$$\int_K |\nabla u| d\gamma=\int_{\R} \gamma^+(\partial \{u<t\}) dt.$$
By the Gaussian isoperimetric inequality \cite{ST}, \cite{Bor-isop}, 
$$\gamma^+(\partial \{u<t\}) dt\geq \gamma^+(\partial \{u^*<t\}) dt,$$ 
and thus, $\int_K |\nabla u| d\gamma$ is greater than or equal to $\int_K |\nabla u^*| d\gamma.$ The general statement also follows from this idea, but with a cleverer use of convexity, see \cite{CarKer}.
\end{remark}

\medskip

\subsection{Some background from Sobolev space theory}

Consider an absolutely continuous measure $\mu$ on $\R^n$ with a locally-Lipschitz density $e^{-V(x)}.$ We shall mostly consider the case when $\mu$ is log-concave, and $V$ is a twice-differentiable convex function. We call $V$ a potential of $\mu.$ By $C^p(\Omega)$ denote the space of $p$ times differentiable functions on a domain $\Omega$, whose $p$-th partial derivatives are continuous. By $Lip(K)$ denote the class of locally Lipschitz functions on $K.$ 

Given a convex set $K$ in $\R^n,$ for a point $x\in\partial K$, we denote by $n_x$ the outward unit normal at $x;$ the vector-field $n_x$ is uniquely defined almost everywhere on $\partial K$. We say that $K$ is of class $C^2$ if its boundary is locally twice differentiable; in this case, $n_x$ is well-defined for all $x\in\partial K$. For a $C^2$ convex set, consider the second fundamental form of $K$ to be the matrix $\rm{II}=\frac{dn_x}{dx}$ (with a ``plus'' because the normal is outer) acting on the tangent space at $x$. The Gauss curvature at $x$ is $det(\rm{II})$ and the mean curvature is $tr(\rm{II}).$ Given a measure $\mu$ with potential $V,$ define the $\mu-$associated mean curvature
$$H_{\mu}=tr(\rm{II})-\langle \nabla V,n_x\rangle.$$

We say that $K$ is strictly convex if $\rm{II}$ is non-singular everywhere on $K.$ We say that $K$ is uniformly strictly convex if $\det(\rm{II})>c>0$ for some $c>0.$  

Associated with $\mu,$ consider the operator $L: C^2(\R^n)\rightarrow L^2(\mu,\R^n),$
$$Lu=\Delta u-\langle \nabla u,\nabla V\rangle.$$
The operator $Lu$ satisfies the following integration by parts identity whenever it makes sense (as follows immediately from the classical Divergence theorem):
$$\int_{K} vLu \, d\mu=-\int_{K}\langle \nabla v,\nabla u\rangle d\mu+\int_{\partial K} v\langle \nabla u,n_x\rangle d\mu_{\partial K}.$$
Here by $\mu_{\partial K}$ we mean the measure $e^{-V(x)}dH_{n-1}(x),$ where $H_{n-1}$ is the $(n-1)$-dimensional Hausdorff measure. The term $\langle \nabla u,n_x\rangle$ is called the normal derivative of $u.$

Fix a function $u\in L^1(\mu, K).$ We say that $w_i$ is a weak $i$-the partial derivative of $u,$ and use notation $w_i=\frac{\partial u}{\partial x_i}$, if for every $v\in C^1(K)$ with $v|_{\partial K}=0,$ we have the following integration by parts identity
$$\int_K vw_i d\mu= -\int_K u \frac{\partial (ve^{-V})}{\partial x_i} dx.$$

Recall that the trace operator is a continuous linear operator 
$$TR: W^{1,2}(K,\mu)\rightarrow L^2(\partial K,\mu)$$
such that for every $u\in C^1(K),$ continuous up to the boundary, we have 
$$TR(u)=u|_{\partial K}.$$
We shall use informal notation $\int_{\partial K} u d\mu$ to mean $\int_{\partial K} TR(u) d\mu$. Similarly, we use notation $\int_{\partial K} \langle \nabla u,n_x\rangle d\mu$ to mean $\int_{\partial K} \langle TR(\nabla u),n_x\rangle d\mu$, where $TR(\nabla u)$ is the vector formed by the trace functions of the weak first partial derivatives of $u.$ We shall also use notation $\nabla,$ $\Delta$ and so on, to denote the appropriate quantities in the sense of weak derivatives.

Define the Sobolev space $W^{1,2}(K,\mu)$ to be the space of $L^2(K,\mu)$ functions whose all weak partial derivatives are in $L^2(K,\mu).$ We shall also consider the space 
$$W_0^{1,2}(K,\mu)= W^{1,2}(K,\mu)\cap\{w:\,TR(w)=0\}.$$
These spaces we consider equipped with the Sobolev norm
$$\|u\|_{W^{1,2}(K,\mu)}=\sqrt{c_1\|u\|^2_{L^2(K,\mu)}+c_2 \sum_{i=1}^n\|\frac{\partial u}{\partial x_i}\|^2_{L^2(K,\mu)}}.$$


\medskip

Next, let us recall the Lax-Milgram Lemma (see e.g. Evans \cite{evans}).

\begin{lemma}[The Lax-Milgram Lemma]\label{LaxM}
Let $H$ be a Hilbert space with norm $\|\cdot\|$. Let $Q$ be a symmetric bilinear form on $H$ and $l$ be a linear functional on $H.$ Suppose that
\begin{itemize}
\item $Q$ is continuous, i.e. $Q(f,g)\leq C_1\|f\|\cdot \|g\|.$
\item $Q$ is coercive, i.e. $Q(f,f)\geq C_2\|f\|^2.$
\item $l$ is continuous, i.e. $|l(f)|\leq C_3\|f\|.$
\end{itemize}
Then there exists a unique $h\in H$ such that for every $f\in H$ we have $Q(f,h)=l(f).$
\end{lemma}

\medskip

Lastly, let us recall a few basic facts about Poincare-type inequalities. For any bounded Lipschitz domain $K$ and any finite absolutely-continuous measure $\mu$, there exists a constant $C_{poin}(K,\mu)$, which we call the Poincare constant, such that for every $F\in W^{1,2}(K,\mu)$ we have
\begin{equation}\label{poin-intro}
Var(F)\leq C^2_{poin}(K,\mu)\E |\nabla F|^2.
\end{equation}
Here and below, whenever $K$ and $\mu$ are fixed, we use notation 
$$Var(F)=\frac{1}{\mu(K)}\int_K F^2 d\mu-\left(\frac{1}{\mu(K)}\int_K F d\mu\right)^2,$$
and 
$$\E F=\frac{1}{\mu(K)}\int_K F d\mu.$$

Similarly, for any bounded Lipschitz domain $K$ and any finite absolutely-continuous measure $\mu$, there exists a constant $C_{D}(K,\mu)$, which we call the Dirichlet-Poincare (or Sobolev) constant, such that for every $F\in W_0^{1,2}(K,\mu)$ we have
\begin{equation}\label{dir-poin-intro}
\E F^2\leq C^2_{D}(K,\mu)\E |\nabla F|^2.
\end{equation}
We refer to (\ref{poin-intro}) as the Poincare inequality, and to (\ref{dir-poin-intro}) as the Dirichlet-Poincare or Sobolev inequality. 


Similarly, one has
\begin{equation}\label{C_D-eq}
C_{D}(K,\gamma)<\infty
\end{equation}
 for any Lipschitz domain $K$, see \cite{CarKer}. 

Moreover, for any log-concave probability measure supported on the entire space, and a convex set $K$ (possibly unbounded), we have
\begin{equation}\label{pointunbndlogc}
C_{poin}(K,\mu)<\infty;	
\end{equation}
this fact is outlined in \cite{Liv-paper2}.

\medskip

\subsection{Trace theory and convexity}

Recall that the diameter of a convex body is the radius of the smallest ball which contains it. Next, the in-radius is the radius of the largest ball that is contained in the body.

Let us recall Gagliardo's trace theorem \cite{Gagliado} (see also Grisvard \cite{Gris} Theorem 1.5.10 or Ding \cite{Ding}); we will rely on this result heavily throughout the paper. Recall that a domain $K\subset \R^n$ is called Lipschitz if its boundary is locally a graph of a Lipschitz function. 

\begin{theorem}[Gagliardo's trace theorem]\label{gagliardo-1}
Let $K$ be a Lipschitz domain. Let $g\in W^{1,2}(K,\lambda)$, where $\lambda$ stands for the Lebesgue measure. Then
$$\int_{\partial K} g^2(x)dH_{n-1}(x)\leq C\int_{K} (g^2(x)+|\nabla g|^2) dx,$$ 
where $C$ depends only on $K$. 

Here, as agreed, we use the notation which omits the trace operator, even though the integration is understood in the sense of the trace operator. 
\end{theorem}

When $K$ is convex, $C$ depends only on the ratio of the in-radius and the diameter of $K$. We shall outline a simple proof of this fact.

\begin{theorem}[explicit trace theorem for convex sets]\label{Leb-Garg}
Let $K$ be a convex domain with an in-radius $r>0$ and the diameter $D$. Let $g\in W^{1,2}(K)$. Then 
$$\int_{\partial K} g^2 dH_{n-1}\leq \frac{1}{r}\int_{K} (n+D)g^2+D|\nabla g|^2 dx.$$ 
\end{theorem}
\begin{proof} Suppose without loss of generality that the center of the ball inscribed into $K$ is at the origin. Then $\langle x,n_x\rangle\geq r,$ and we write, using a trick similar to the one that was used in Hosle, Kolesnikov, L \cite{HKL}, which in turn was inspired by the tricks in Kolesnikov, Milman \cite{KolMilsupernew}:
$$\int_{\partial K} g^2 dH_{n-1}\leq \frac{1}{r}\int_{\partial K} \langle xg^2,n_x\rangle dH_{n-1}=\frac{1}{r}\int_K div(xg^2) dx,$$
where in the last passage we integrated by parts. We write, for any $t>0,$
$$div(xg^2)=ng^2+2g\langle \nabla g, x\rangle\leq (n+t)g^2+\frac{1}{t}\langle \nabla g,x\rangle^2.$$
Combining the above with the fact that the diameter of $K$ is $D$, we have
$$\int_{\partial K} g^2 dH_{n-1}\leq \frac{1}{r}\int_{K} (n+t)g^2+\frac{D^2}{t}|\nabla g|^2 dx.$$ 
Selecting (arbitrarily) $t=D,$ we get the result.
\end{proof}

\begin{remark} In some particular situation one may prefer to select another value of $t,$ but in general, we are not seeking tight trace estimates.
\end{remark}

\begin{remark} More generally, one may obtain estimates for $\int_{\partial K} |g|^p dH_{n-1}$, for other values of $p>0.$
	
\end{remark}

As an immediate corollary, we get

\begin{theorem}[Trace theorem for measures]\label{gagliardo}
Let $K$ be a bounded Lipschitz domain. Let $g\in W^{1,2}(K,\mu)$, for a probability measure $\mu$ whose density is bounded and strictly positive on any compact set. Then
$$\int_{\partial K} g^2 d\mu_{\partial K}\leq C_1\int_{K} (g^2+|\nabla g|^2) d\mu,$$ 
where $C$ depends only on $K$ and $\mu.$ If $K$ is convex, the constant $C$ depends only on the diameter and the in-radius of $K$, the dimension $n,$ and $\mu$. 
\end{theorem}
\begin{proof} Indeed, by Theorem \ref{Leb-Garg}, we may let $C_1=\frac{n+D}{r}\frac{max_{x\in\partial K} F(x)}{min_{x\in K} F(x)},$ where $F$ is the density of $\mu;$ the expression is finite by our assumptions.
\end{proof}

\begin{remark} Of course, one may adapt the proof of Theorem \ref{Leb-Garg} to any measure $\mu$ with density $e^{-V}$, and get that for any convex body $K$ containing the ball $y_0+rB^n_2$, and any $g\in W^{1,2}(K)$,
$$\int_{\partial K} g^2 d\mu_{\partial K}\leq \frac{1}{r}\int_K ng^2+2g\langle \nabla g,x-y_0\rangle-g^2 \langle x-y_0,\nabla V\rangle d\mu\leq C\int_K g^2+|\nabla g|^2 d\mu.$$
However, even in the case of log-concave probability measures, the constant will in general depend on the diameter, the in-radius of $K$, and $\mu$. Assuming, say, that $\mu$ is even, and $K$ is symmetric, we get 
$$\int_{\partial K} g^2 d\mu_{\partial K}\leq \frac{1}{r}\int_K (n+x^2)g^2+|\nabla g|^2 d\mu.$$
But still, even if $\mu$ is isotropic, this is difficult to estimate in general for an unbounded convex set, unless we know some further information about the growth of our function $g$.
	
\end{remark}

In some particular cases, however, the dependance may be nicer; in Theorem \ref{GaussGarg} we shall outline one such important case.

Federer \cite{Fed} introduced a notion of a \emph{reach} of a domain $K$ to be 
$$reach(K)=\inf_{y\in\R^n}\sup\{r>0: rB^n_2+y\subset U_K\},$$
where $U_K$ is the set of points in $\R^n$ such that the distance from them to $K$ is attained at a unique point. A version of Gagliardo's theorem for the Gaussian measure was obtained by Harrington, Raich \cite{HaRa}, in which the boundedness assumption is replaced by the assumption that $\partial K$ is $C^2$ and has positive reach.

Here, we show that in the case when $K$ is convex and contains the origin, and $\gamma$ is the standard Gaussian measure, the constant in the trace theorem does not depend on the diameter, and the estimate is valid for unbounded sets. This fact will be helpful in Sections 5 and 7, and leads to part (2) of Theorem \ref{eqchar-intro}.

\begin{theorem}[Gaussian Trace Theorem for convex sets containing the origin]\label{GaussGarg}
Let $K$ be a convex domain such that $rB^n_2\subset K$, for some $r>0$. Let $g\in W^{1,2}(K,\gamma)$, where $\gamma$ is the Gaussian measure. Then
$$\int_{\partial K} g^2 d\gamma_{\partial K}\leq \frac{1}{r}\int_{K} (ng^2+|\nabla g|^2) d\gamma.$$ 
\end{theorem}
\begin{proof} Again, we use the estimate $\langle x,n_x\rangle\geq r,$ incorporate the trick similar to the ones from \cite{HKL}, \cite{KolMilsupernew}, and integrate by parts, to get
$$\int_{\partial K} g^2 d\gamma_{\partial K}\leq \frac{1}{r}\int_{\partial K} \langle g^2x,n_x\rangle d\gamma_{\partial K}=\frac{1}{r}\int_K (div(g^2 x)-g^2x^2) d\gamma.$$
Note that
$$div(g^2 x)=ng^2+2g\langle \nabla g,x\rangle\leq ng^2+|\nabla g|^2+g^2x^2.$$
Combining the above yields the result. \end{proof}

\medskip

\subsection{Crash course on classical existence theorems with emphasis on our specific needs}

In this subsection, we formulate several existence results. Although they are classical and well known, we sketch the proofs, since for us, the specifics of the assumptions will matter, as we will be dealing with some delicate approximation arguments, and with \emph{unbounded domains}.

We say that $u\in W^{1,2}(K,\mu)$ is a weak solution of the Poisson system with Neumann boundary condition, if
$$
\begin{cases}  Lu=F &\, on \,\,\,K\\ \langle \nabla u,n_x\rangle=f &\, on \,\,\,\partial K,\end{cases}$$
whenever the identity
$$
\int_{K}\langle \nabla v,\nabla u\rangle d\mu= \int_{K} vF \, d\mu +\int_{\partial K} vf d\mu_{\partial K}$$
holds for every $v\in W^{1,2}(K,\mu)$ (here the boundary integral is understood in the sense of the trace operator). 

We say that $u\in W^{1,2}(K,\mu)$ is a weak solution of the Poisson system with Dirichlet boundary condition, if
$$
\begin{cases}  Lu=F &\, on \,\,\,K\\ u=f & on \,\,\,\partial K,\end{cases}$$
whenever the identity
$$
\int_{K}\langle \nabla v,\nabla u\rangle d\mu= \int_{K} vF \, d\mu$$
holds for every $v\in W_0^{1,2}(K,\mu)$, and $TR(u)=f.$

\begin{theorem}[Neumann boundary condition]\label{exist-n}
Let $K$ be a bounded Lipschitz domain and let $\mu$ be an absolutely-continuous probability measure with a locally-Lipschitz density. Let $F\in L^2(K,\mu)\cap Lip(K)$ and $f\in L^2(\mu,\partial K)$, such that
$$\int_K F d\mu=\int_{\partial K} f d\mu_{\partial K}.$$
Then there exists a unique function $u\in W^{1,2}(K,\mu)$ which is a solution of 
$$
\begin{cases}  Lu=F &\, on \,\,\,K\\ \langle \nabla u,n_x\rangle=f &\, on \,\,\,\partial K.\end{cases}$$

Furthermore, \\
a) if $\mu$ is the standard Gaussian measure and $K$ is convex, then the boundedness assumption is not required;\\
b) if $f=0$, $\mu$ is a log-concave measure supported on the whole space, and $K$ is a convex set, then the boundedness assumption is not required.

Moreover, if $\partial K$ is $C^{2}$ and $f\in C^{1}(\partial K)$, then $u\in C^{2}(int(K))\cap C^1(\bar{K}).$
\end{theorem}
\begin{proof} Consider the bilinear form
$$Q(u,v)=\int_K \langle \nabla u,\nabla v\rangle d\mu,$$
and the linear functional
$$l(v)=-\int_K Fv d\mu+\int_{\partial K} f v d\mu_{\partial K}.$$
By Cauchy's inequality, $Q(u,v)\leq C \|u\|_{W^{1,2}(K,\mu)}\cdot \|u\|_{W^{1,2}(K,\mu)},$ therefore $Q$ is continuous. Without loss of generality, we may work in the space $W^{1,2}(K,\mu)\cap \{\int_K u d\mu=0\}$ (if we pick $u+C$ in place of $u$, nothing changes). Then, by Poincare inequality (\ref{poin-intro}), $Q$ is coercive (where the coercivity constant depends on the Poincare constant). Lastly, we note that $l$ is continuous by the Cauchy's inequality and the trace theorem \ref{gagliardo-1}. The existence follows by the Lax-Milgram lemma \ref{LaxM}.

The ``moreover'' part follows by the standard regularity estimates, see e.g. Evans \cite{evans} or Kolesnikov, Milman \cite{KM1}. 

The ``furthermore a)'' part follows by using Theorem \ref{GaussGarg} in place of Theorem \ref{gagliardo-1}. The ``furthermore b)'' part follows since one does not need to use Gagliardo's theorem if $f=0,$ and (\ref{pointunbndlogc}) guarantees that $Q$ is coercive, without needing boundedness of the domain.

In order to show uniqueness, suppose by contradiction that there is two functions $u$ and $v$ which satisfy the system, and consider $h=u-v.$ Then $Lh=0$ on $K$ and $\langle \nabla h,n_x\rangle=0$ on $\partial K$ which implies that $\int_K |\nabla h|^2 d\mu=0$, and thus $h=const$ in the sense of the space $W^{1,2},$ yielding that $u=v,$ up to a constant addition.
\end{proof}

In a similar fashion, one may show the following existence result. 

\begin{theorem}[The $L+Id$ operator, Neumann boundary condition]\label{existsLu=-u}
Let $K$ be a bounded convex domain. Let $\gamma$ be the standard Gaussian measure. Let $F\in L^2(K,\gamma)$ and $f\in L^2(\partial K,\gamma)$. Then for any $s\leq 1,$ there exists a function $u\in W^{1,2}(K, \gamma)$ which is a solution of 
$$
\begin{cases}  Lu+su=F &\, on \,\,\,K\\ \langle \nabla u,n_x\rangle=f &\, on \,\,\,\partial K.\end{cases}$$
In case, additionally, $K$ is symmetric and $F$ and $f$ are even, then the solution exists whenever $s\leq 2.$

Moreover, if $\partial K$ is $C^{2}$ and $f\in C^{1}(\partial K)$, then $u\in C^{2}(int(K))\cap C^1(\bar{K}).$
\end{theorem}
\begin{proof} Consider the bilinear form
$$Q(u,v)=\int_K \langle \nabla u,\nabla v\rangle -suv \,\, d\gamma,$$
and the linear functional
$$l(v)=-\int_K Fv d\gamma +\int_{\partial K} f v d \gamma|_{\partial K}.$$
As before, $l$ is continuous by the Cauchy's inequality and the Gagliardo's theorem \ref{gagliardo-1}. By Cauchy's inequality, $Q$ is continuous. Without loss of generality, we may work in the space $W^{1,2}(K,\mu)\cap \{\int_K u d\mu=0\}$: indeed, by switching $F$ with $F+C_0,$ where $C_0$ is fully determined by $F$ and $f,$ we get that $\int_K u d\gamma=0.$ Namely, for this we select
$$C_0=\int_{\partial K} f d\gamma-\int_K F d\gamma.$$

Whenever $K$ is not a cylinder (and, in particular, whenever $K$ is bounded), $C_{poin}(K,\gamma)<1$, as follows from Theorem \ref{eqchar-intro} (the fact that $C_{poin}(K,\gamma)\leq 1$ we already explained above, but here the strict inequality matters). Therefore, by (\ref{poin-intro}), 
$$Q(u,u)= \int_K |\nabla u|^2 - su^2 \, d\gamma \geq (1-sC^2_{poin}(K,\gamma)) \int_K |\nabla u|^2 \, d\gamma,$$
and thus $Q$ is coercive, whenever $s\leq 1$. The existence follows by the Lax-Milgram lemma \ref{LaxM}.

In the symmetric case, $C^2_{poin}(K,\gamma)<\frac{1}{2},$ whenever $K$ is bounded. Therefore, the symmetric version follows in the same manner for any $s\leq 2.$ The ``moreover'' part follows by the standard regularity estimates, see e.g. Evans \cite{evans} or Kolesnikov, Milman \cite{KM1}.
\end{proof}

Lastly, we state existence and uniqueness results for the Dirichlet boundary condition, which will be relevant for us in Section 5 and onwards.

\begin{theorem}[Dirichlet boundary condition]\label{exist-dir}
Let $K$ be a bounded Lipschitz domain and let $\mu$ be an absolutely-continuous probability measure with a locally-Lipschitz density. Let $F\in L^2(K,\mu)$ and $f\in TR(W^{1,2}(K,\mu)\cap C^{2}(K))$. Then there exists a unique function $u\in W^{1,2}(K,\mu)$ which is a solution of 
$$
\begin{cases}  Lu=F &\, on \,\,\,K\\ u=f &\, on \,\,\,\partial K.\end{cases}$$
Furthermore, in the case $f=0$ and the Gaussian measure $\gamma,$ the domain $K$ may be unbounded.

Moreover, if $\partial K$ is $C^{2}$ and $f\in C^{2}(\partial K)$, then $u\in C^{2}(int(K))\cap C^1(\bar{K}).$
\end{theorem}
\begin{proof} Let $v_0\in TR(W^{1,2}(K,\mu)\cap C^{2}(K))$ be any function such that $TR(v_0)=f$. Then $u$ is the desired solution if $u=w+v_0,$ and $w$ is the solution of the system
$$
\begin{cases}  Lw=G &\, on \,\,\,K\\ w=0 &\, on \,\,\,\partial K,\end{cases}$$
with $G=F-Lv_0.$ Consider the bilinear form
$$Q(w,v)=\int_K \langle \nabla w,\nabla v\rangle d\mu$$
and the linear functional
$$l(v)=-\int_K  Gv d\mu,$$
both acting on the space $W^{1,2}_0(K).$ By the Dirichlet-Poincare inequality (\ref{dir-poin-intro}), $Q$ is coercive. Both $Q$ and $l$ are continuous by the choice of our space. The existence follows by the Lax-Milgram lemma \ref{LaxM}. The ``moreover'' part follows by the standard regularity estimates, see e.g. Evans \cite{evans} or Kolesnikov, Milman \cite{KM1}.

For the ``furthermore'' part, recall (\ref{C_D-eq}), i.e. that $C_{D}(K,\gamma)<\infty$ for any Lipschitz domain $K$ such that $K$ is not the whole space (as shall be proven in Corollary \ref{cor-ehr-nonsym-1}).

In order to show uniqueness, suppose by contradiction that there is two functions $u$ and $v$ which satisfy the system, and consider $h=u-v.$ Then $Lh=0$ on $K$ and $h|_{\partial K}=0$, which implies that $\int_K |\nabla h|^2 d\mu=0$, and thus $h=0$ in the sense of the space $W^{1,2},$ yielding that $u=v.$
\end{proof}

\medskip
\medskip

\subsection{Preliminaries from Asymptotic Analysis}

We outline the following result which is well-known to experts.

\begin{lemma}\label{moments}
Let $\mu$ be any rotation-invariant probability measure with an absolutely continuous density. Then 
\begin{itemize}
\item For any $q>0,$ and any convex body $K$ containing the origin,
$$\int_K |x|^q d\mu(x)\geq \int_{B(K)} |x|^q d\mu(x),$$
\item For any $q<0,$ and any convex body $K$ containing the origin,
$$\int_K |x|^q d\mu(x)\leq \int_{B(K)} |x|^q d\mu(x),$$
\end{itemize}
where $B(K)$ is the Euclidean ball centered at the origin such that $\mu(B(K))=\mu(K).$
\end{lemma}
\begin{proof} Let $d\mu(x)=e^{-\varphi(|x|)}$, for some convex function $\varphi$. Suppose $\mu(K)=\mu(RB_2^n).$ We write, by polar coordinates,
$$\int_K |x|^q d\mu(x)-\int_{RB^n_{2}} |x|^q d\mu(x)=$$$$\int_{\sfe}\left(\int_0^{\|\theta\|^{-1}_K} t^{n+q-1} e^{-\varphi(t)}dt-\int_0^{R} t^{n+q-1} e^{-\varphi(t)}dt\right)d\theta\geq $$
$$R^q\int_{\sfe}\left(\int_0^{\|\theta\|^{-1}_K} t^{n-1} e^{-\varphi(t)}dt-\int_0^{R} t^{n-1} e^{-\varphi(t)}dt\right)d\theta=0,$$
since $\mu(K)=\mu(RB_2^n).$ The second assertion follows in an analogous manner.
\end{proof}

	

\medskip

In the case $p=2$, recall the notation
$$J_k(R)=\int_0^R t^k e^{-\frac{t^2}{2}}dt$$
and
$$g_k(t)=t^k e^{-\frac{t^2}{2}}.$$
We shall need the following

\begin{lemma}\label{1-dim-by-parts}
For every $R>0,$ and any integer $n\geq 1,$ we have
$$J_{n+3}(R)=(n+2)J_{n+1}(R)-g_{n+2}(R).$$
\end{lemma}
\begin{proof}
We integrate by parts
$$J_k(R)=\int_0^R t^k e^{-\frac{t^2}{2}}dt=-\frac{1}{k+1} R^{k+1} e^{-\frac{R^2}{2}}+\frac{1}{k+1}\int^R_0 t^{k+2} e^{-\frac{R^2}{2}},$$
and arrive to the conclusion, letting $k=n+1$.	
\end{proof}

We shall also need

\begin{lemma}\label{lastlemma}
Fix $\theta\in\sfe$. Let $K$ be a convex set in $\R^n$ such that it is either barycentered, or for each $y\in \theta^{\perp}$, the interval $K\cap (y+span(\theta))$ contains $y$. Then
$$\frac{1}{\gamma(K)}\int_K \langle x,\theta\rangle^2 d\gamma(x)\leq 1.$$
Moreover, the equality is attained if and only if $K\in Cyl_n.$
\end{lemma}
\begin{proof} If $K$ is barycentered, the inequality follows from the Brascamp-Leib inequality (\ref{BrLi-convex}).

Suppose for each $y\in \theta^{\perp}$, the interval $K\cap (y+span(\theta))$ contains $y$. Consider the operator $T_{\theta}^t$ of dilation by a factor of $t$ in the direction $\theta.$ Then, by our assumption, for any $t\geq 1$ we have $\gamma(T_{\theta}^t K)\geq \gamma(K),$ and therefore, 
	$$\frac{d}{dt}\gamma(T_{\theta}^t K)|_{t=1}\geq 0.$$
	We claim that this amounts to the statement of the Lemma. Consider without loss of generality $\theta=e_1.$ Since $det T^t_{e_1}=t$, we have
	$$\gamma(T^t_{e_1}(K))=\frac{1}{\sqrt{2\pi}^n}\int_{T^t_{e_1}(K)} e^{-\frac{y^2}{2}} dy= \frac{1}{\sqrt{2\pi}^n}\int_{K} e^{-\frac{(x+(t-1)x_1)^2}{2}} t dx,$$
	and thus 
	$$\frac{d}{dt}\gamma(T_{e_1}^t K)|_{t=1}=\int_K (1-x_1^2) d\gamma\geq 0.$$
	The inequality follows. For the ``moreover'' part, note that the function $\gamma(T_{\theta}^t K)$ is strictly increasing in $t$ unless $K\in Cyl_n.$
\end{proof}

\section{The concavity coefficient, isoperimetry and the related discussion.}

\subsection{The concavity coefficient of a convex set} Fix a convex set $K$ in $\R^n$ with $C^2$ boundary, and consider a twice-differentiable function $\varphi:\sfe\rightarrow\R$. Let $g:\partial K\rightarrow\R$ be given by $g(x)=\varphi(n_x).$ Recall that the support function $h_K:\sfe\rightarrow\R$ of a convex set $K$ is the function 
$$h_K(u)=\sup_{x\in K}\langle x,u\rangle.$$
Consider a collection of convex sets $\{K^g_t\}_{t\in [0,\epsilon]}$ with support functions 
$$h_t(u):=h_{K^g_t}(u)=h_K(u)+t\varphi(u).$$
Note that the system $K^g_t$ is well-defined for any strictly-convex set $K$, for a sufficiently small $\epsilon>0.$ In the case that $K$ is not strictly convex, the system $K^g_t$ may still be well defined for many ``admissible'' functions $g.$ For instance, it is well defined for $\epsilon\in [0,1]$ and any function $g$ of the form $g(x)=\varphi(n_x),$ where
$$\varphi(u)=h_L(u)-h_K(u),$$
for any convex set $L,$ since in this case, $K^g_t=(1-t)K+tL,$ which is a convex set for all $t\in [0,1].$ 

One may show the following Lemma, previously discussed in Colesanti \cite{Col1}, Livshyts, Marsiglietti \cite{CLM}, Kolesnikov \cite{KolLiv}, Milman \cite{KolMil-1}, and other related works.
\begin{lemma}\label{step1-concavity}
Let $\mu$ be a probability measure on $\R^n.$ Fix convex sets $K$ and $L$ with $C^2$ boundary, and denote $K_t=(1-t)K+tL$. The twice-differentiable strictly increasing function $F:[0,1]\rightarrow\R$ satisfies
\begin{equation}\label{eq-111} 
F\left((1-t) \mu(K)+t\mu(L)\right)\geq (1-t) F(\mu(K))+tF(\mu(L))
\end{equation}
for all $t\in [0,1]$, if and only if for any $a\in [0,1]$, letting $t_a$ be such that $\mu(K_{t_a})=a,$
\begin{equation}\label{eq-222}
\frac{F''(a)}{F'(a)}\leq -\frac{M_a''(0)}{M_a'(0)^2},
\end{equation}
where $M_a(s)=\mu(K_{t_a+s})$.
\end{lemma}
\begin{proof} In the case when $K$ and $L$ have $C^2$ boundary, $M_a$ is twice differentiable for all $a\in[0,1].$ The inequality (\ref{eq-111}) amounts to the concavity of the function $F(\mu(K_t))$ on $[0,1]$, which is equivalent to the fact that $F(\mu(K_t))|_{t}''\leq 0.$ The latter amounts to (\ref{eq-222}).
\end{proof}

For a convex set $K,$ we say that a function $g:\partial K\rightarrow\R$ is admissible if there exists an $\epsilon>0$ (possibly depending on $g$) such that $K^g_{\epsilon}$ is a well-defined convex set, and the function $M^g(t)=\mu(K^g_t)$ is continuously twice differentiable at $t=0.$ We use notation $Adm(K)$ for the set of such admissible $g.$

Motivated by Lemma \ref{step1-concavity}, and following the notation of Kolesnikov and Milman \cite{KolMil-1}, we denote, for a $C^2-$smooth convex set $K$ and a measure $\mu,$
$$G_{\mu}(K):=\inf_{g\in C^2(\sfe)\cap Adm(K)} \left( -\frac{\mu(K^g_t)|_{t=0}''}{(\mu(K^g_t)|_{t=0}')^2} \right).$$
One may note that $G_{\mu}(K)$ is continuous on the set of convex sets, and thus for a non $C^2$ set $K,$ we may define $G_{\mu}(K)$ by approximation. We shall call $G_{\mu}(K)$ \emph{a concavity coefficient of $K$.} 

Further, fix a class $\mathcal{F}$ of convex sets which is closed under Minkowski interpolations, and for each $K\in\mathcal{F}$ restrict the admissible class of functions $Adm(K)$ to those which interpolate within the class $\mathcal{F}$. Consider a function $G_{\mu,\mathcal{F}}:[0,1]\rightarrow\R$ given by
$$G_{\mu,\mathcal{F}}(a)=\inf_{K\in \mathcal{F}: \mu(K)=a} G_{\mu}(K).$$

We shall especially focus our attention on the particular case when $\mathcal{F}$ is the class of symmetric convex sets, and $g:\sfe\rightarrow\R$ are even functions. We denote

$$G^s_{\mu}(K):=\inf_{g\in C^2(\sfe)\cap Adm(K):\, g(u)=g(-u)} \left( -\frac{\mu(K^g_t)|_{t=0}''}{(\mu(K^g_t)|_{t=0}')^2} \right),$$
and
$$G^s_{\mu}(a)=\inf_{K\in \mathcal{K}^{sym}: \mu(K)=a} G^s_{\mu}(K),$$
where $\mathcal{K}^{sym}$ stands for the collection of all symmetric convex sets in $\R^n.$

As an immediate corollary of Lemma \ref{step1-concavity} and an approximation argument, we get the following Proposition, which has appeared in various forms in \cite{Col1}, \cite{CLM}, \cite{CL}, \cite{KM1}, \cite{KM2}, \cite{KolMil}, \cite{KolMilsupernew}, and others.

\begin{proposition}\label{prop-equiv} 
Fix a class $\mathcal{F}$ of convex sets $K$ which is closed under Minkowski interpolations. Let $\mu$ be a log-concave measure on $\R^n$. Suppose $F:[0,1]\rightarrow\R$ is a strictly increasing continuously twice differentiable function. Then the inequality
$$F(\mu\left((1-t)K+tL\right))\geq (1-t) F(\mu(K))+tF(\mu(L))$$
holds for all $t\in [0,1]$ and every pair $K, L\in \mathcal{F}$, if and only if for every $a\in [0,1],$
$$\frac{F''(a)}{F'(a)}\leq G_{\mu, \mathcal{F}}(a).$$
\end{proposition}

\begin{remark}\label{remlogconc}
Letting $F(a)=\log a$ in Proposition \ref{prop-equiv} we see that $$G_{\mu,\mathcal{F}}(a)\geq -\frac{1}{a}$$
for any class $\mathcal{F}$ of convex sets, whenever $\mu$ is log-concave.
\end{remark}

In other words, Proposition \ref{prop-equiv} implies that ``the optimal'' function $F$ which satisfies (\ref{ehrhard-general}) for some class $\mathcal{F}$ of convex sets, is given by
$$F(a)=\int_0^a exp\left(\int_{c_0}^t G_{\mu,\mathcal{F}}(s)ds\right)dt.$$
Here, the choice of the constant $c_0\in (0,1)$ is arbitrary; we must select $c_0>0$ so that the corresponding integral converges. Note that the function given above is necessarily twice differentiable and strictly increasing.

\medskip

\subsection{Concavity powers}

Recall that for a log-concave measure $\mu$ and $a\in [0,1]$ we let $p(\mu,a)$ be the largest real number such that for all convex sets $K$ and $L$ with $\mu(K)\geq a$ and $\mu(L)\geq a,$ and every $\lambda\in [0,1]$ one has 
$$\mu(\lambda K+(1-\lambda)L)^{p(\mu,a)}\geq \lambda \mu(K)^{p(\mu,a)}+(1-\lambda)\mu(L)^{p(\mu,a)}.$$
Analogously, we define $p_s(\mu,a)$ as the largest number such that for all convex \emph{symmetric} sets $K$ and $L$ with $\mu(K)\geq a$ and $\mu(L)\geq a,$ and every $\lambda\in [0,1]$ one has 
$$\mu(\lambda K+(1-\lambda)L)^{p_s(\mu,a)}\geq \lambda \mu(K)^{p_s(\mu,a)}+(1-\lambda)\mu(L)^{p_s(\mu,a)}.$$
Further, for a convex set $K$ we define 
$$p_s(K,\mu)=\limsup_{\epsilon\rightarrow 0}\left\{p:\,\forall L\in \mathcal{K}^n_s,\,\, \mu((1-\epsilon) K+\epsilon L)^{p}\geq (1-\epsilon) \mu(K)^{p}+\epsilon\mu(L)^{p}\right\},$$
where $\mathcal{K}^n_s $ stands for the set of symmetric convex sets in $\R^n.$ By log-concavity, we get
\begin{lemma}\label{lemma-min-p_s}
$$p_s(a,\mu)=\inf_{K:\,\mu(K)\geq a}p_s(K,\mu).$$
\end{lemma}
\begin{proof} Note that for any pair of symmetric convex sets $K$ and $L,$
\begin{equation}\label{equiv-p}
\mu(\lambda K+(1-\lambda)L)^p\geq \lambda \mu(K)^p+(1-\lambda) \mu(L)^p,	
\end{equation}
with $p=inf_{\lambda} p_s(K_{\lambda},\mu),$ where we use notation $K_{\lambda}=\lambda K+(1-\lambda)L.$ 

In case $\mu(L)\geq a$ and $\mu(L)\geq a,$ we have 
$$\mu(\lambda K+(1-\lambda)L)\geq \mu(K)^{\lambda}\mu(L)^{1-\lambda}\geq a.$$
Additionally, a convex combination of convex symmetric sets is a convex symmetric set. In other words, the class of convex symmetric sets with measure exceeding $a\in[0,1]$ is closed under Minkowski interpolation. Therefore, by (\ref{equiv-p}), the Lemma follows.
\end{proof}

\begin{remark}
The assumption of log-concavity of $\mu$ in the previous Lemma could be replaced with the assumption of quasi-concavity, that is, 
$$\mu\left((1-t)K+tL\right)\geq \min(\mu(K),\mu(L)).$$
\end{remark}

Lastly, we notice the relation between the concavity powers and the concavity coefficient, which follows immediately from Proposition \ref{prop-equiv}:

\begin{lemma}\label{G-p}
For a convex symmetric set $K$,
$$G^s_{\mu}(K)=\frac{p_s(K,\mu)-1}{\mu(K)}.$$
Consequently, for any $a\in[0,1],$
$$G^s_{\mu}(a)=\frac{p_s(\mu,a)-1}{a}.$$
\end{lemma}


We note also

\begin{proposition}\label{ODE}
Let $$F(a)=\int_0^a \exp\left(\int_{c_0}^s \frac{-1+p_s(\mu,t)}{t} dt\right) ds=\int_0^a \exp\left(\int_{c_0}^s G^s_{\mu}(t) dt\right) ds,$$
for an arbitrary $c_0\in (0,1).$ Then for any pair of symmetric convex sets $K,L$ in $\R^n,$ 	$$F(\mu(\lambda K+(1-\lambda)L))\geq \lambda F(\gamma(K))+(1-\lambda)F(\gamma(L)).$$
\end{proposition}
\begin{proof}
By Proposition \ref{key_prop-general}, the function $F(a)$ which satisfies
$$\frac{F''(a)}{F'(a)}=\frac{p^s(\mu,a)-1}{a}=G^s_{\mu}(a),$$
will also satisfy the conclusion. Solving the elementary ODE, we note that the function in the statement of the proposition does satisfy this equation. 	
\end{proof}

\medskip

\subsection{Analyzing the Conjecture \ref{theconj} and its isoperimetric implications}

We shall introduce some \textbf{Notation.} 

\begin{itemize}
\item $\Phi(t)=\frac{1}{\sqrt{2\pi}} \int_{-\infty}^t e^{-\frac{s^2}{2}}ds,\,\,t\in[0,\infty]; $
\item $g_p(t)=t^p e^{-\frac{t^2}{2}},\,\,t\in[0,\infty];$
\item $J_p(R)=\int_0^R g_p(t)dt,\,\,R\in[0,\infty];$
\item $c_p=J_p(+\infty);$
\item $R_k(a)=J_{k-1}^{-1}(c_{k-1}a),\,\,a\in[0,1];$
\item $\varphi(t)=2\Phi(t)-1=R^{-1}_1(t)=\frac{1}{\sqrt{2\pi}}\int_{-t}^t e^{-\frac{s^2}{2}}ds,\,\,t\in[0,\infty];$
\item $s_k(a)=c^{-1}_{k-1} g_{k-1}\circ R_k(a) ,\,\,a\in[0,1];$
\item $\varphi_k(a)=\frac{c_{k-1}(R_k(a)^2-k+1)}{g_k\circ R_k(a)} ,\,\,a\in[0,1].$	
\end{itemize}

We consider also the convex set 
$$C_k(a)=R_k(a) B^k_2\times \R^{n-k}\subset \R^n,$$
which we call \emph{$k$-round cylinder.} The geometric meaning of the function $R_k(a)$ is the radius of the $k$-round cylinder $C_k(a)$ of measure $a,$ i.e. $\gamma(R_k(a)B^n_2\times \R^{n-k})=a.$ Also, note that $s_k(a)=\gamma^+(\partial C_k(a))$. We notice the following relations:

\begin{equation}\label{relation1}
R'_k(a)=\frac{c_{k-1}}{g_{k-1}\circ R_k(a)}=\frac{1}{s_k(a)},\end{equation}

and 

\begin{equation}\label{relation2}
\varphi_k(a)=\left(\log \frac{1}{s_k(a)}\right)'.
\end{equation}

By (\ref{relation1}) and (\ref{relation2}), and the change of variables $x=R_k(s)$, followed by the change of variables $y=R_k(t)$, the function 
$$
F(a)= \int_0^a \exp\left(\int_{C_0}^{t} \min_{k=1,...,n} \varphi_k(s)ds\right) dt=$$
\begin{equation}\label{defF}
\int_0^a \exp\left(\int_{C_0}^{t} \min_{k=1,...,n} \left(\frac{c_{k-1} (-k+1+J^{-1}_{k-1}(c_{k-1}s)^2)}{g_{k}\circ J^{-1}_{k-1}(c_{k-1}s)}\right)ds\right) dt
\end{equation}
from the Conjecture \ref{theconj}, rewrites as
$$F(a)= \sum_{k=1}^n 1_{a\in E_k}\cdot(\beta_k R_k(a)+\alpha_k),$$
for some disjoint sets $E_k\subset [0,1],$ such that $[0,1]=\cup_{k=1}^n E_k,$ and some coefficients $\alpha_k, \beta_k\in\R$. Namely, $E_k\subset [0,1]$ is chosen to be the set where $\varphi_k(a)=\min_{j=1,...,n} \varphi_j(a).$ 

We shall also need the following 
\begin{lemma}[``the strip wins eventually'']\label{varphi1}
There exists a constant $\alpha_1\in [0,1]$ such that for every $a\geq \alpha_1,$ one has
$$\varphi_1(a)=\min_{k=1,...,n} \varphi_k(a).$$
\end{lemma}
\begin{proof} For $a\in [0,1]$, we shall analyze $\frac{d}{dk} \varphi_k(a),$ where $k\in (0,n].$ Note that nothing is stopping us from considering $k$ in the interval, rather than just taking integer values. The Lemma will follow if we show that there exists a constant $\alpha_1\in [0,1]$ such that for every $a\geq \alpha_1,$ one has $\frac{d}{dk} \varphi_k(a)\geq 0$ for all $k\in (1,n].$ 

Indeed, letting 
$$I_{k-1}(R)=\int_0^{\infty} t^{k-1}e^{-\frac{t^2}{2}}\log t \,dt,$$
we note that
$$ \frac{d}{dk} R_k(a)=\frac{ac'_{k-1}}{I_{k-1}\circ R_k (a)},$$
and hence 
$$\frac{d}{dk} \varphi_k(a)=\left(c_{k-1}\frac{R^2_k(a)-k+1}{g_k\circ R_k(a)}\right)'_k=$$$$c'_{k-1}\frac{R^2_k(a)-k+1}{g_k\circ R_k(a)}+ c_{k-1}\frac{2 R_k(a)}{g_k\circ R_k(a)}\cdot \frac{ac'_{k-1}}{I_{k-1}\circ R_k (a)}-c_{k-1}\frac{1}{g_k\circ R_k(a)}-$$$$c_{k-1}\frac{R^2_k(a)-k+1}{g_k\circ R_k(a)} \cdot \left(-\log R_k(a)+\left(R_k(a)-\frac{k-1}{R_k(a)}\right)\frac{ac'_{k-1}}{I_{k-1}\circ R_k (a)}\right).$$
Note that 
$$c'_{k-1}=\int^{\infty}_0 t^{k-1}e^{-\frac{t^2}{2}}\log t \,dt,$$
and $I_{k-1}(R)\rightarrow_{R\rightarrow\infty} c'_{k-1}$. Therefore, there exists an $\alpha_0\in [0,1]$ such that for any $a\geq \alpha_0,$ and any $k\in [1,n]$, we have $\frac{c'_{k-1}}{I_{k-1}\circ R_k (a)}\geq 0,$ and thus multiplying by this factor does not change the sign of the above expression. Denote $R=R_k(a)$. The inequality $\frac{d}{dk} \varphi_k(a)\geq 0$ holds whenever 
$$\frac{(R^2-k+1)I_{k-1}(R)}{c_{k-1}}-\frac{I_{k-1}(R)}{c'_{k-1}} +2aR-$$
$$\frac{I_{k-1}(R)}{c'_{k-1}}(R^2-k+1)\log R+a R \left(R-\frac{k-1}{R}\right)^2\geq 0.$$
In other words, when $a\geq \alpha_0$ and $k\in (1,n],$ and $R=R_k(a),$ there are $c_i(k)\in (A_n, B_n]$, $i\in \{1,...,5\},$ where the interval $(A_n,B_n]$ depends only on the dimension, such that the inequality $\frac{d}{dk} \varphi_k(a)\geq 0$ is equivalent to 
$$aR^3+c_1(k)R^2\log R+c_2(k)R^2+c_3(k)R+c_4(k)\log R+c_5(k)\geq 0.$$
Let $R_0$ be the smallest value of $R$ which makes the above expression non-negative for all $k\in (1,n]$, provided that $a\geq \alpha_0.$ Let $\alpha_1\geq \alpha_0$ be such a number in $[0,1]$ that $R_1(a)\geq R_0$; then $R_k(a)\geq R_0$ for all $k\geq 1$. We conclude that for all $a\in [\alpha_1,1]$ and all $k\in (1,n]$ one has $\frac{d}{dk} \varphi_k(a)\geq 0$. This implies the Lemma.\end{proof}

Here is a few pictures to illustrate Lemma \ref{varphi1}.

\medskip

The graphs of the functions $\varphi_1, \varphi_2, s_1, s_2$:
\begin{center}
\includegraphics{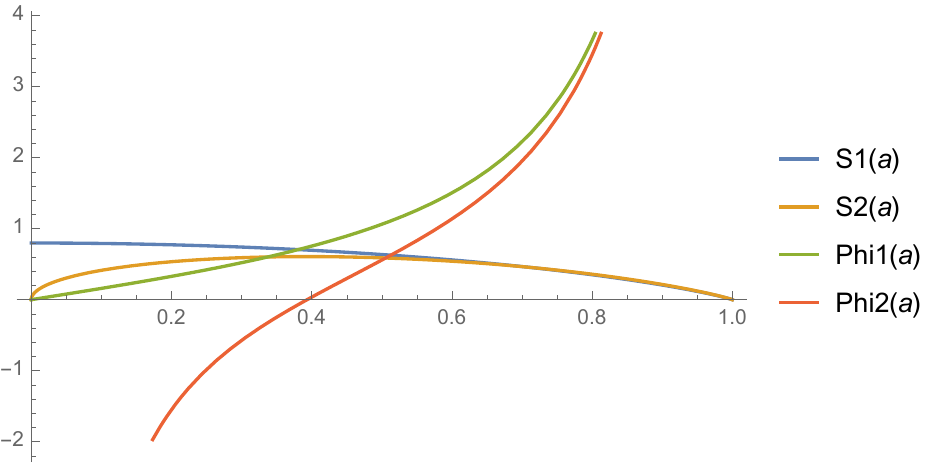}
\end{center}
The situation with $\varphi_1$ and $\varphi_2$ at first:
\begin{center}
\includegraphics{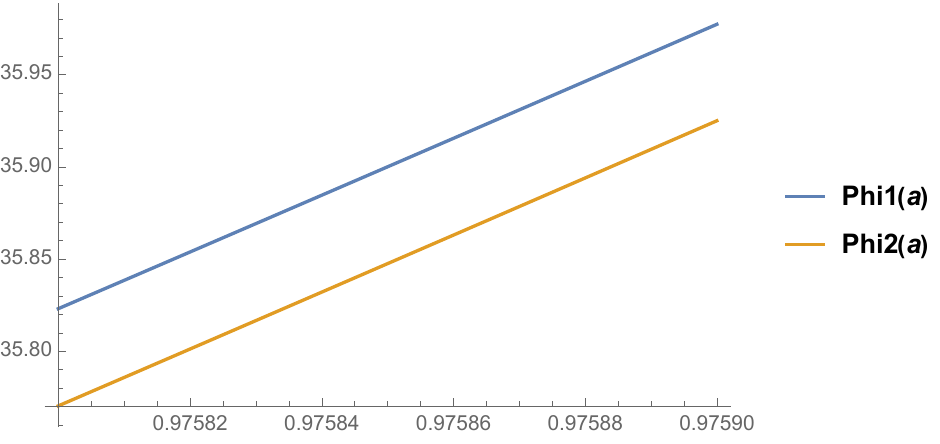}
\end{center}
The situation with $\varphi_1$ and $\varphi_2$ eventually:
\begin{center}
\includegraphics{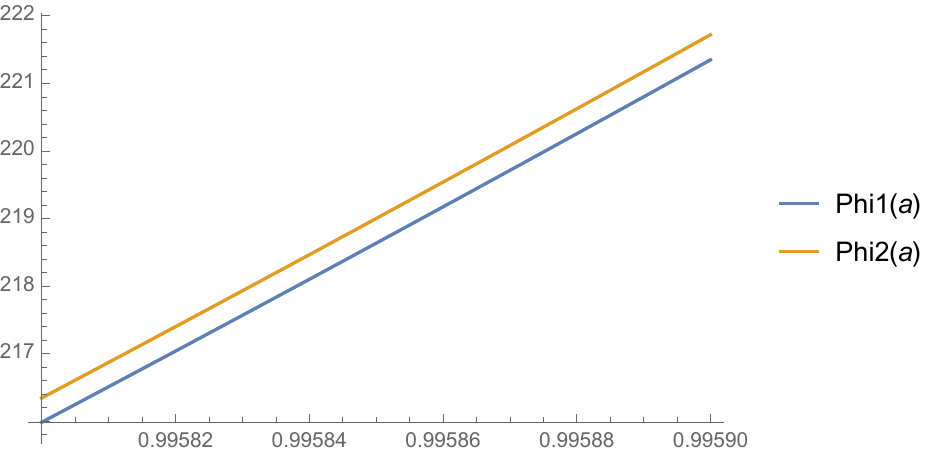}
\end{center}
The graph of $\varphi_1-\varphi_2$:
\begin{center}
\includegraphics{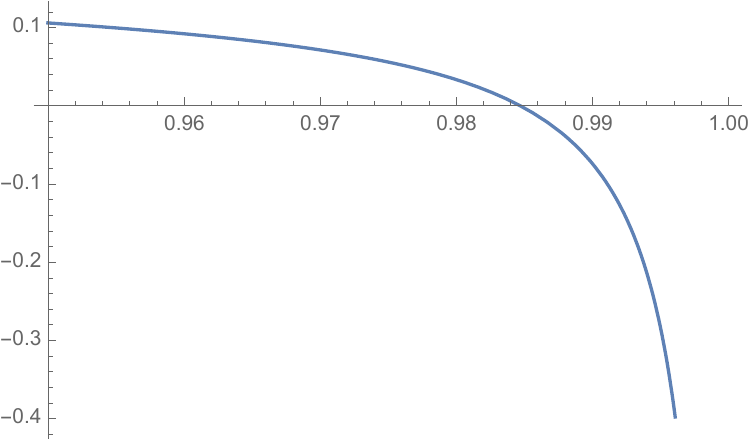}
\end{center}

\medskip

The Ehrhard inequality implies the Gaussian isoperimetric inequality of Borell \cite{Bor-isop} and Sudakov-Tsirelson \cite{ST}, which states that for every $a\in [0,1],$ among all measurable sets in $\R^n$ of Gaussian measure $a$, a half-space has the smallest Gaussian perimeter. Indeed, here is the argument that one may also find in Lata\l{}a \cite{Lat}: for every Borel-measurable set $K,$ and any $\lambda\in [0,1],$
$$\gamma^+(\partial K)=\liminf_{\epsilon\rightarrow 0} \frac{\gamma(K+\epsilon B^n_2)-\gamma(K)}{\epsilon}\geq $$$$\liminf_{\epsilon\rightarrow 0} \frac{\Phi\left((1-\lambda)\Phi^{-1}(\gamma(\frac{K}{1-\lambda}))+\lambda \Phi^{-1}( \gamma(\frac{\epsilon}{\lambda} B^n_2))\right)-\gamma(K)}{\epsilon}=$$
$$\Phi'(\Phi^{-1}(\gamma(K)))\cdot\lim_{t\rightarrow\infty}\frac{\Phi^{-1}(\gamma(tK))}{t}=\frac{1}{\sqrt{2\pi}}e^{-\frac{\Phi^{-1}(\gamma(K))^2}{2}}= \gamma^+(\partial H_K),$$
where $H_K$, as before, is the half-space of Gaussian measure $\gamma(K).$ We used here that 
$$\lim_{t\rightarrow\infty}\frac{\Phi^{-1}(\gamma(tK))}{t}=1,$$
which follows from the facts that
$$\Phi(t)=1-e^{-\frac{(t+o(t))^2}{2}};$$
$$\gamma(t B^n_2)=1-e^{-\frac{(t+o(t))^2}{2}},$$
where $o(t)$ is something that tends to infinity strictly slower than $t$.

In a similar fashion, we show that Conjecture \ref{theconj} implies the sharp isoperimetric statement for symmetric convex sets of sufficiently large Gaussian measure. This fact is, actually, known unconditionally, and was shown by Barchiesi and Julin \cite{italians}; we state the next proposition just for completeness.

\begin{proposition}\label{implic}
Suppose the Conjecture \ref{theconj} holds, at least for some interval $a\in [c',1]$. Then there exists a constant $\alpha\in [0.94,1]$ such that for any $a\geq \alpha,$ and any symmetric convex set $K$ with $\gamma(K)=a,$ one has $\gamma^+(\partial K)\geq \gamma^+(\partial C_1(a)),$ where $C_1(a)$ is a symmetric strip of measure $a.$
\end{proposition}
\begin{proof} Let $\alpha_1\in [0,1]$ be such that for every $a\geq \alpha_1,$ we have $\varphi_1(a)=\min_{k=1,...,n} \varphi_k(a)$ (it exists by Lemma \ref{varphi1}, and the numerical computation shows that it exceeds $0.94$). Let us pick $C_0=\alpha_1$ in the definition of $F$ (\ref{defF}). Let $\alpha=\max(c',\alpha_1).$ By (\ref{relation1}) and (\ref{relation2}), we have, for all $a\geq \alpha,$ that
$$F'(a)=\frac{s_1(C_0)}{s_1(a)}.$$
Further, we have, for a sufficiently large $t\in \R,$
$$F(\gamma(t B^n_2))=\frac{\varphi^{-1}(\gamma(t B^n_2))}{2s_1(C_0)}.$$
Note also that
$$\lim_{t\rightarrow\infty}\frac{\varphi^{-1}(\gamma(tK))}{2t}=1,$$
which follows from the facts that
$$\varphi(t)=1-2e^{-\frac{(t+o(t))^2}{2}};$$
$$\gamma(t B^n_2)=1-e^{-\frac{(t+o(t))^2}{2}},$$
where $o(t)$ is something that tends to infinity strictly slower than $t$. To summarize, we conclude that for $a\geq \alpha,$
$$\frac{1}{F'(a)}\lim_{t\rightarrow\infty} \frac{F(\gamma(t B^n_2))}{t}=s_1(a)=\gamma^+(\partial C_1(a)).$$

\medskip

Suppose that the Conjecture \ref{theconj} holds, at least for some interval $[c',1].$ Then for every convex symmetric set $K$ with $\gamma(K)=a\in [\alpha,1],$ we have
$$\gamma^+(\partial K)=\liminf_{\epsilon\rightarrow 0} \frac{\gamma(K+\epsilon B^n_2)-\gamma(K)}{\epsilon}\geq $$$$\liminf_{\epsilon\rightarrow 0} \frac{F^{-1}\left((1-\lambda)F(\gamma(\frac{K}{1-\lambda}))+\lambda F( \gamma(\frac{\epsilon}{\lambda} B^n_2))\right)-\gamma(K)}{\epsilon}=$$
$$\frac{1}{F'(a)}\cdot\lim_{t\rightarrow\infty}\frac{F(\gamma(tK))}{t}=\gamma^+(\partial C_K(a)),$$
and the proposition follows.
	
\end{proof}

\begin{remark}
One may not adapt the argument of Proposition \ref{implic} to obtain that the Conjecture \ref{theconj} implies the Gaussian isoperimetry for symmetric convex sets on the entire $[0,1]$. The fundamental reason for this is that the equality in the Conjecture \ref{theconj} holds only for pairs of cylinders whose measures belong to certain intervals, and one may not let $t$ tend to infinity, as in the proof of Proposition \ref{implic}, and hope to have any sharp estimate. But one might hope that some local argument should help deduce the Gaussian isoperimetry for symmetric convex sets from the Conjecture \ref{theconj}.	

\end{remark}

\begin{remark}[One-dimensional case]
	We would like to add that Conjecture \ref{theconj} holds in dimension 1, and is, in fact, always an equality. Indeed, when $n=1,$ we simply have $F=\varphi^{-1}.$ For any pair of symmetric intervals $I_a$ and $I_b$ of Gaussian measures $a$ and $b$ respectively, we write $I_a=[-\varphi^{-1}(a), \varphi^{-1}(a)]$ and $I_b=[-\varphi^{-1}(b), \varphi^{-1}(b)]$. Then
	$$F(\gamma(\lambda I_a+(1-\lambda)I_b))=F\circ \varphi (\lambda \varphi^{-1}(a)+(1-\lambda) \varphi^{-1}(b))=$$$$ \lambda \varphi^{-1}(a)+(1-\lambda) \varphi^{-1}(b) = \lambda F(\gamma(I_a))+(1-\lambda) F(\gamma(I_b)).$$
\end{remark}

\medskip

\subsection{The concavity powers of round $k$-cylinders}

In this subsection, we shall prove Proposition \ref{cyl}. It partly follows from \emph{the inequality part} of our main result Corollary \ref{Gauss-main-intro}, whose proof we defer until Section 7. Note that the \emph{equality case characterization} in Corollary \ref{Gauss-main-intro} follows, in part, from Proposition \ref{cyl}, but there is no vicious cycle. First, we outline

\begin{lemma}\label{upperbnd}
$$p_s\left(RB_2^{k}\times \R^{n-k}\right)\leq 1-\frac{J_{k-1}(R)}{g_k(R)}\left(k-1-R^2\right).$$
\end{lemma}
\begin{proof} From the definition of $p_s$ we deduce that $p_s\left(RB_2^{k}\times \R^{n-k}\right)$ is smaller than the smallest non-negative number $p$ such that the function $F(t)=\gamma(t B_2^{k}\times \R^{n-k} )^p$ is concave at $t=R.$ Note that $F(t)=C_{n,k,p} J^p_{k-1}(t)$. Taking the second derivative of this function, we deduce the bound.
\end{proof}

Next, Lemma \ref{1-dim-by-parts} yields immediately 

\begin{lemma}\label{ibp}
When $K=RB_2^{k}\times \R^{n-k}$, for $k=1,...,n,$ we have
$$1-\frac{J_{k-1}(R)}{g_k(R)}\left(k-1-R^2\right)=\frac{\E Y^2-\E Y^4+(\E Y^2)^2+k}{(k-\E Y^2)^2},$$
where $Y=\sum_{i=1}^{k} X_i,$ and the expectation is with respect to the Gaussian measure restricted onto $K.$
\end{lemma}

\medskip

\textbf{Proof of the Proposition \ref{cyl}.} The inequality of Corollary \ref{Gauss-main-intro} (or Theorem \ref{Gauss-main-1}) shows that
$$p_s(K,\gamma)\geq \frac{\E Y^2-\E Y^4+(\E Y^2)^2+k}{(k-\E Y^2)^2},$$ 
where $Y=\sum_{i=1}^{k} X_i$ (as before). By Lemma \ref{ibp}, we get the lower bound
$$p_s(K,\gamma)\geq 1-\frac{J_{k-1}(R)}{g_k(R)}\left(k-1-R^2\right),$$ 
which matches the corresponding upper bound from Lemma \ref{upperbnd}. $\square$ 

\medskip

\subsection{The equivalence of the Conjectures \ref{theconj} and \ref{theconj1}} 

Finally, we outline  

\textbf{Proof of the Proposition \ref{key-connection} (the equivalence of Conjectures \ref{theconj} and \ref{theconj1}):} Suppose first that Conjecture \ref{theconj} holds. Then, by Proposition \ref{prop-equiv}, 
$$G^s_{\gamma}(a)\geq \min_{k=1,...,n} \frac{c_{k-1}(R_k(a)^2-k+1)}{g_k\circ R_k(a)} ,\,\,a\in[0,1],$$ 
which is equivalent, by Lemma \ref{G-p} and Proposition \ref{cyl}, to the fact that
$$p_s(\gamma,a)\geq \min_{k=1,...,n} \left(\frac{ac_{k-1}(R_k(a)^2-k+1)}{g_k\circ R_k(a)}+1\right)= \min_{k=1,...,n} p_s(C_k(a),\gamma).$$
Therefore, Conjecture \ref{theconj1} holds. The converse implication follows in exact same manner. $\square$

We are now finally ready to comment as to why the following function cannot satisfy (\ref{ehrhard-general}) for all symmetric convex sets:
\begin{equation}\label{funcagain}
F(a)=\sum_{k=1}^n 1_{a\in I_k}\cdot( J_{k-1}^{-1}(c_{k-1} a)+a_k),
\end{equation}
where $I_k$ is the sub-interval of $[0,1]$ where the surface area of the $k$-round cylinder $C_k(a)$ is minimal (over all choices of $k$), where $a_k$ are chosen in such a way that $F$ is continuous. Indeed, if this function was to satisfy (\ref{ehrhard-general}) for all symmetric convex sets, then, by Proposition \ref{prop-equiv} and Lemma \ref{G-p}, we would have
$$p_s(\gamma,a)\geq \sum_{k=1}^n 1_{a\in I_k} p_s(C_k(a),\gamma).$$
However, by definition of $p_s(\gamma,a),$
$$p_s(\gamma,a)\leq \min_{k=1,...,n} p_s(C_k(a),\gamma).$$
The two inequalities above are, in fact, incompatible: one may check using Mathematica (see the figures above), that for some $a\in [0,1],$ the strict inequality holds
$$\sum_{k=1}^n 1_{a\in I_k} p_s(C_k(a),\gamma)> \min_{k=1,...,} p_s(C_k(a),\gamma).$$
In other words, unfortunately, the intervals where $p_s(C_k(a),\gamma)$ are minimal, are not the same as the intervals where $s_k(a)$ are minimal.


\medskip

\subsection{Estimates via the L2-method}

In this section we survey the method of obtaining convexity inequalities, previously studied by Kolesnikov and Milman \cite{KM1}, \cite{KM2}, \cite{KM}, \cite{KolMil-1}, \cite{KolMilsupernew}, as well as Livshyts \cite{KolLiv}, Hosle \cite{HKL}, and others. 

Consider a log-concave measure $\mu$ on $\R^n$ with an even twice-differentiable density $e^{-V}.$ Consider also the associated operator
$$Lu=\Delta u-\langle \nabla u,\nabla V\rangle.$$

We recall the notation $H_{\mu}=tr(\rm{II})-\langle \nabla V,n_x\rangle.$ The following Bochner-type identity was obtained by Kolesnikov and Milman. It is a particular case of Theorem 1.1 in \cite{KM1}
(note that $\rm{Ric}_{\mu}  = \nabla^2 \it{V}$ in our case).  This is a generalization of a classical result of R.C.~Reilly.

\begin{proposition}[Kolesnikov-Milman]\label{raileyprop}
Let $\mu$ be a log-concave measure on $\R^n$ with density $e^{-V},$ for some $C^2$ convex function $V.$ Let $u \in C^2(K)$ and $\langle \nabla u, n_x \rangle \in C^1(\partial K)$. Then  
\begin{align}
\label{railey}
\int_{K} (L u)^2 d \mu & =\int_{K} \left(||\nabla^2 u||^2+\langle\nabla^2 V \nabla u, \nabla u\rangle\right) d \mu+
\\&  \nonumber\int_{\partial K}  (H_{\mu} \langle \nabla u, n_x \rangle ^2 -2\langle \nabla_{\partial K} u, \nabla_{\partial K}  \langle \nabla u, n_x \rangle \rangle +\langle \mbox{\rm{II}} \nabla_{\partial K} u, \nabla_{\partial K}  u\rangle )  \,d\mu_{\partial K} (x).
\end{align}
\end{proposition}

Next, we shall use the following identities, building up on the work of Kolesnikov and Milman \cite{KM1}, as well as \cite{CLM}, \cite{KolLiv}.

\begin{proposition}\label{deriv}
Let $K$ be a strictly convex body in $\R^n$ with $C^2-$smooth boundary. Consider a family of convex bodies $K_t,$ with support functions $h_t(y)=h_K(y)+\varphi(y)$, for a fixed 1-homogeneous function $\varphi\in C^2(\R^n)$. Let $u\in C^2(K)\cap W^{1,2}(K)$ be such that $\langle \nabla u,n_x\rangle=\varphi(n_x)$, for all $x\in\partial K.$ Let $M(t)=\mu(K_t).$ Then 	
$$M'(0)=\int_{\partial K} f d\mu_{\partial K}=\int_K Lu d\mu;$$
$$M''(0)= \int_{\partial K} H_{\mu} f^2 -\langle \rm{II}^{-1}\nabla_{\partial K} f,\nabla_{\partial K} f\rangle \, d\mu_{\partial K}=$$$$\int_K (Lu)^2-\|\nabla^2 u\|^2-\langle \nabla^2 V\nabla u,\nabla u\rangle d\mu-Q(u),$$
where
$$Q(u)= \int_{\partial K}  \langle (\mbox{\rm{II}})^{-1} \nabla_{\partial K}  \langle \nabla u, n_x \rangle , \ \nabla_{\partial K}  \langle \nabla u, n_x \rangle \rangle -$$$$2\langle \nabla_{\partial K} u, \nabla_{\partial K}  \langle \nabla u, n_x \rangle \rangle +\langle \mbox{\rm{II}} \nabla_{\partial K} u, \nabla_{\partial K}  u\rangle   \,d\mu_{\partial K} (x)\geq 0.$$
\end{proposition}

Propositions \ref{prop-equiv} and \ref{deriv} yield

\begin{proposition}\label{key_prop-general}
Let $\mathcal{F}$ be a class of convex sets closed under Minkowski interpolation. Let $F:[0,1]\rightarrow\R$ be a strictly increasing twice differentiable function such that for any $a\in[0,1]$, for any $C^2$-smooth $K\in\mathcal{F}$ with $\mu(K)\geq a$, for every $f\in C^1(\partial K)\cap Adm(K)$ there exists a $u\in C^2(K)\cap W^{1,2}(K)$ with $\langle \nabla u,n_x\rangle=f$, and such that,
$$\frac{F''(a)a}{F'(a)}\leq \frac{\int -(Lu)^2+\|\nabla^2 u\|^2+ \langle \nabla^2 V\nabla u,\nabla u\rangle +Q(u)}{\int(Lu)^2}.$$
Then for any pair of sets $K,L\in \mathcal{F},$ one has
$$F(\mu(\lambda K+(1-\lambda)L))\geq \lambda F(\gamma(K))+(1-\lambda)F(\gamma(L)).$$
	
\end{proposition}



\medskip

\subsection{Connections to the Lata\l{}a-Oleszkiewicz S-inequality and a new related question}

Summarizing the previous sub-sections, we point out that 
$$G^s_{\gamma}(C_k(a))=\varphi_k(a),$$
and the Conjecture \ref{theconj} is equivalent to showing that for every symmetric convex set $K$ with $\gamma(K)=a,$
$$G^s_{\gamma}(K)\geq \min_{k=1,...,n} \varphi_k(a).$$
In other words, in view of the Propositions \ref{deriv} and \ref{prop-equiv}, Conjecture \ref{theconj} states that for every strictly-smooth symmetric convex $K$ with $\gamma(K)=a$, and for every even $C^2-$smooth function $f:\partial K\rightarrow \R$, we have
\begin{equation}\label{conj-clear}
\int_{\partial K} H_{\gamma} f^2 -\langle \rm{II}^{-1}\nabla_{\partial K} f,\nabla_{\partial K} f\rangle \, d\gamma_{\partial K} \leq -\min_{k=1,...,n} \varphi_k(a) \left(\int_{\partial K} f d\gamma_{\partial K} \right)^2.
\end{equation}

We obtain

\begin{proposition}\label{f=1}
The Conjecture \ref{theconj} entails that for every strictly-convex symmetric $C^2$-smooth body $K$ with $\gamma(K)=a$,
$$\int_{\partial K} tr(\rm{II}) d\gamma_{\partial K} +\int_K x^2 d\gamma \leq na-\min_{k=1,...,n} \varphi_k(a) \gamma^+(\partial K)^2,$$
and the equality is attained if and only if $K=RB^k_2\times \R^{n-k}$ for some $k=1,...,n$ and $R>0.$
\end{proposition}
\begin{proof} We plug $f=1$ into (\ref{conj-clear}), and recall that $H_{\gamma}=tr(\rm{II})-\langle x,n_x\rangle.$ It remains to integrate by parts
$$\int_{\partial K} \langle x,n_x\rangle d\gamma_{\partial K} =n\gamma(K)-\int_K x^2 d\gamma.$$
\end{proof}

\begin{remark}\label{secondmom} We recall also the geometric meaning
$$\int_{\partial K} \langle x,n_x\rangle d\gamma_{\partial K}=\gamma(tK)'_{t=1},$$
as may be verified directly.	
\end{remark}

Recall that the celebrated Lata\l{}a-Oleszkiewicz S-inequality \cite{sconj} states that for any convex symmetric set $K$ in $\R^n$ and any $t\geq 1,$ $\gamma(tK)\geq \gamma(tS_K)$, where $S_K$ is the symmetric strip such that $\gamma(K)=\gamma(S_K).$ Equivalently, in view of Remark \ref{secondmom}, $\int_K x^2 d\gamma \leq \int_{S_K} x^2 d\gamma.$

\medskip

As a consequence of Propositions \ref{implic} and \ref{f=1}, we get 

\begin{proposition}\label{Sineq}
There exists an $\alpha\in [0,1]$ such that for every $a\in [\alpha,1],$ the validity of the Conjecture \ref{theconj} implies the S-inequality for all symmetric convex sets $K$ with $\gamma(K)=a.$
\end{proposition}
\begin{proof} Suppose Conjecture \ref{theconj} holds. By the Proposition \ref{implic}, there exists $\alpha\in [0,1]$ such that for every $a\in [\alpha,1],$ and every symmetric convex set $K$ with $\gamma(K)=a,$ one has
\begin{equation}\label{implicimplic}
	\gamma^+(\partial K)\geq \gamma^+(\partial S_K).
\end{equation}
Next, note that for all $a\in [\alpha,1]$, we have $\varphi_1(a)=\min_{k=1,...,n} \varphi_k(a),$ and also note that $\varphi_1(a)\geq 0$ on this sub-interval (recall Lemma \ref{varphi1} and the first figure depicting the graphs of the functions $\varphi_1$ and $\varphi_2$). Therefore, combining (\ref{implicimplic}) with the Proposition \ref{f=1}, we have
$$\int_{\partial K} tr(\rm{II}) d\gamma_{\partial K} +\int_K x^2 d\gamma \leq na-\varphi_1(a) \gamma^+(\partial K)^2\leq na-\varphi_1(a) \gamma^+(\partial S_K)^2.$$
It remains to use the convexity of $K$ and to say that $tr(\rm{II})\geq 0,$ to deduce the inequality
$$\int_K x^2 d\gamma \leq na-\varphi_1(a) \gamma^+(\partial S_K)^2.$$
In case $K$ is a symmetric strip, the equality, in fact, occurs in the inequality above, and we get 
$$\int_K x^2 d\gamma \leq \int_{S_K} x^2 d\gamma,$$
concluding the proof.
\end{proof}

\begin{remark}
In a similar fashion, one may notice that the infinitesimal version of the Ehrhard inequality, derived by Kolesnikov and Milman \cite{KolMil}, which states that
$$
\int_{\partial K} H_{\gamma} f^2 -\langle \rm{II}^{-1}\nabla_{\partial K} f,\nabla_{\partial K} f\rangle \, d\gamma_{\partial K}\leq -\eta(a) \left(\int_{\partial K} f d\gamma_{\partial K}\right)^2,
$$
where
$$\eta(a)=\sqrt{2\pi} a\Phi^{-1}(a) e^{\frac{\Phi^{-1}(a)^2}{2}},$$
implies, for all convex sets $K$ with $\gamma(K)\geq 0.5,$ that
$$\int_K x^2 d\gamma \leq \int_{H_K} x^2 d\gamma,$$
where $H_K$ is the half-space such that $\gamma(K)=\gamma(H_K).$ We will discuss the local version of Ehrhard's inequality in Subsection 8.1 a bit more. See Corollary 1.2 in \cite{KolMil} for a similar discussion.
\end{remark}

\medskip

Next, we would like to discuss another implication of Proposition \ref{f=1}.

\begin{proposition}\label{discs} Suppose Conjecture \ref{theconj} holds in $\R^2$, and suppose, further, the symmetric Gaussian isoperimetric conjecture of Morgan-Heilman \cite{Heilman} holds in $\R^2$. Then there exists an interval $[\alpha,\beta]\subset [0,1]$ such that for any $a\in [\alpha,\beta]$, and any symmetric convex set $K$ with $\gamma(K)=a,$ we have
$$
\int_{\partial K} tr(\rm{II}) d\gamma_{\partial K}\leq \int_{\partial B_K} tr(\rm{II}) d\gamma_{\partial B_K},
$$
where $B_K$ is the centered euclidean ball of Gaussian measure $a.$
\end{proposition}
\begin{proof} Let $[\alpha,\beta]$ be such an interval where $0\leq \varphi_2(a)\leq \varphi_1(a)$, and also $s_2(a)\leq s_1(a)$ (see the first figure above to see that $[0.5, 0.9]\subset [\alpha, \beta]$.) From Proposition \ref{f=1} combined with Lemma \ref{moments}, we see that
$$\int_{\partial K} tr(\rm{II}) d\gamma_{\partial K} \leq -\int_{B_K} x^2 d\gamma + na-\varphi_2(a) \gamma^+(\partial K)^2.$$
Since we assumed also that the symmetric Gaussian isoperimetric conjecture of Morgan-Heilman \cite{Heilman} holds in $\R^2$, we note that $\gamma^+(\partial K)\geq \gamma^+(\partial B_K)$ (in view of the assumption that $a$ is such that $s_2(a)\leq s_1(a)$), and thus
$$\int_{\partial K} tr(\rm{II}) d\gamma_{\partial K} \leq -\int_{B_K} x^2 d\gamma + na-\varphi_2(a) \gamma^+(\partial B_K)^2.$$
When $K=B_K,$ the equality is attained, and therefore the Proposition follows. 
\end{proof}

Proposition \ref{discs} motivates us to ask

\begin{question} Let $K$ be a symmetric convex set in $\R^n$. Let $B_K$ be the centered euclidean ball with $\gamma(B_K)=\gamma(K).$ Is it true that	
$$
\int_{\partial K} tr(\rm{II}) d\gamma_{\partial K}\leq \int_{\partial B_K} tr(\rm{II}) d\gamma_{\partial B_K}?
$$
\end{question}

\medskip

\subsection{Inequalities of the type of Minkowski's first}

As an application of our estimates, we show the following analogue of Minkowski's first inequality (see Schneider \cite{book4}). For Lebesgue measure $|\cdot|$ it states that 
$$V_1(K,L)\geq |K|^{\frac{n-1}{n}}|L|^{\frac{1}{n}},$$
where
$$V_1(K,L)=\frac{1}{n}\lim_{\epsilon\rightarrow 0}\frac{|K+\epsilon L|-|K|}{\epsilon}.$$
Recall also the Gaussian definition \cite{LivMink}:
$$\gamma_1(K,L)=\lim_{\epsilon\rightarrow 0}\frac{\gamma(K+\epsilon L)-\gamma(K)}{\epsilon}=\gamma(K+tL)'_{t=1}.$$

We show
\begin{proposition}[Minkowski's first inequality for symmetric sets in Gauss space]\label{isoper}
For any pair of symmetric convex sets $K$ and $L$,
$$\gamma_1(K,L)\geq \left(1-\frac{\E X^2}{n}\right)\gamma(K)^{1-\frac{1}{n-\E X^2}}\gamma(L)^{\frac{1}{n-\E X^2}},$$
where the expected value is taken with respect to the restriction of the Gaussian measure $\gamma$ onto $K.$
\end{proposition}
\begin{proof} For any symmetric convex $L$ we have
$$\gamma_1(K,L)=\lim_{\epsilon\rightarrow 0}\frac{\gamma((1-\epsilon)K+\epsilon L)-\gamma((1-\epsilon)K)}{\epsilon}\geq $$$$\lim_{\epsilon\rightarrow 0}\frac{\left((1-\epsilon)\gamma(K)^p+\epsilon \gamma(L)^p\right)^{\frac{1}{p}}-\gamma((1-\epsilon)K)}{\epsilon},$$
with $p=\frac{1}{n-\E X^2}.$ The above equals
$$\gamma(K)\left(-\frac{1}{p}+\frac{1}{p}\left(\frac{\gamma(L)}{\gamma(K)}\right)^p+\frac{\gamma_1(K,K)}{\gamma(K)}\right).$$
One may note that
$$\gamma_1(K,K)=\int_{\partial K}\langle x,n_x\rangle d\gamma_{\partial K}=\gamma(K)\left(n-\E X^2\right)=\frac{\gamma(K)}{p}.$$
The proposition follows. 
\end{proof}

We outline the partial case of Proposition \ref{isoper} which corresponds to the Gaussian surface area $\gamma^+(\partial K).$ Recall that
$$\gamma^+(\partial K):=\gamma_1(K,B^n_2)= \liminf_{\epsilon\rightarrow 0} \frac{\gamma(K+\epsilon B^n_2)-\gamma(K)}{\epsilon}.$$
By a homogeneity argument, we get the following corollary of Proposition \ref{isoper}.

\begin{cor}
For a symmetric convex set $K$,
$$\gamma^{+}(\partial K)\geq (n-\E X^2)\gamma(K)^{1-\frac{1}{n-\E X^2}}\inf_{R>0} \frac{\gamma(RB^n_2)^{\frac{1}{n-\E X^2}}}{R}\geq \frac{(n-\E X^2)\gamma(K)}{R(K)},$$
where the expected value is taken with respect to the restriction of the Gaussian measure $\gamma$ onto $K,$ and $R(K)$ is the radius of the ball which has the same Gaussian measure as $K.$
\end{cor}

\begin{remark}
Note that when $K=R B_2^n,$ then the second (weaker) inequality is exactly an equality. 
\end{remark}

\section{Equality cases in the Brascamp-Lieb inequality: Proof of Theorem \ref{eqchar-intro}.}

In this subsection, we characterize the equality cases of the Brascamp-Leib inequality restricted to a convex set (\ref{BrLi-convex}), that is 
$$
\mu(K)\int_K f^2 d\mu-\left(\int_K f d\mu\right)^2\leq\mu(K)\int_K \langle (\nabla^2 V)^{-1}\nabla f,\nabla f\rangle d\mu.
$$
This equality case characterization will be used in the proof of our main result, Theorem \ref{Gauss-main-1}, and also is of independent interest.

In the case when $\mu$ is Gaussian, this inequality becomes the Poincare inequality; in the Gaussian case, it has been known since the nineteenth century, and has many proofs, see Chafai and Lehec \cite{ChafLeh} for a survey. The general case, however, is more delicate.

In the case of the Gaussian measure, the question about the equality cases in this inequality was raised in Remark 1.1 by Brandolini, Chiacchio, Henrot, Trombetti \cite{strips-conj}; a partial progress was made by Brandolini, Chiacchio, Krejcirik, Trombetti \cite{ital-2}. Furthermore, related inequalities and their equality cases (under smoothness assumptions) were studied by Cheng and Zhou \cite{ChenZho}, and de Philippis and Figalli \cite{deFF}. Courtade and Fathi \cite{CouFa} obtained a stability version in a related inequality. See Chafai and Lehec \cite{ChafLeh}, where these questions are surveyed, and it is shown in Lemma 3.1 that a measure with density $e^{-\frac{x^2}{2}-W}$, where $W$ is a convex function, only gives equality in (\ref{BrLi-convex}) taken with $V=\frac{x^2}{2}$, in case the density of $\mu$ has a Gaussian factor in some direction. The argument relies on the application of Caffarelli's contraction theorem \cite{Caff}. In particular, this implies part (1) of Theorem \ref{eqchar} in the partial case when $\mu$ is the standard Gaussian measure. Furthermore, this result in the Gaussian case was also obtained by Beck, Jerison \cite{BeJe}, in relation to studying an interesting related inequality.

The crucial question here is \emph{the shape} of the set $K$ for which (\ref{BrLi-convex}) turns into an equality. For the reader's benefit, we start with the equality case characterization in the Brascamp-Lieb inequality for smooth sets; in subsequent statements, the equality cases will be characterized in full, leading to the proof of the Theorem \ref{eqchar-intro}. The approximation argument in the non-smooth case is what presents most of the difficulties, and requires a careful treatment, especially in the non-Gaussian case. In the Gaussian case, we shall establish the quantitative stability estimate, stated in part (2) of Theorem \ref{eqchar-intro}.

\begin{proposition}[Equality cases in the Brascamp-Lieb inequality -- smooth case]\label{eqchar}
Let $\mu$ be a log-concave measure on $\R^n$ with $C^2$ density $e^{-V}$ supported on the whole $\R^n,$ and suppose $\nabla^2 V>0.$ Then for any convex set $K$ with $C^2$ boundary and any function $f\in W^{1,2}(K)\cap C^1(K)$ we have 
$$\mu(K)\int_K f^2 d\mu-\left(\int_K f d\mu\right)^2\leq\mu(K)\int_K \langle (\nabla^2 V)^{-1}\nabla f,\nabla f\rangle d\mu;$$
Moreover, the equality occurs either if one of the two things happen: 
\begin{itemize}
\item $f=C$ for some constant $C\in\R;$ 
\item there exists a rotation $U$ such that\\ 
a) $UK=L\times \R^{n-k}$ for some $k=1,...,n$ and a $k-$dimensional set $L;$\\ 
b) $f\circ U=\langle \nabla V,\theta\rangle+C$, for some vector $\theta\in\R^{n-k}$ and some constant $C\in\R;$\\
\end{itemize} 
\end{proposition}
\begin{proof}
We shall analyze ``H\"ormander's proof with the boundary'' of the Brascamp-Lieb inequality, discovered by Kolesnikov and Milman \cite{KolMil}. See also H\"ormander \cite{Horm}, or Cordero-Erasquin and Klartag \cite{CorKl} for an exposition of the ``boundary-free'' argument, which is simpler, but is not illuminating in terms of the equality cases. 

Without loss of generality we may assume that $\int f=0$ (otherwise we may pass to the function $f-\int f$). Consider the operator $Lu=\Delta u-\langle \nabla u,\nabla V\rangle.$ Consider an arbitrary function $u\in W^{2,2}(K)$ such that $Lu=f.$ Using integration by parts and Proposition \ref{raileyprop}, we write (\ref{eq-BL-key}) with the general $V:$

\begin{align}\label{eq-11}
\nonumber\int_K f^2 d\mu&=2\int_K f Lu d\mu-\int_K (Lu)^2 d\mu=
\\&-2\int_K \langle \nabla f,\nabla u\rangle d\mu+2\int_{\partial K} f\langle \nabla u,n_x\rangle d\mu_{\partial K}	-
\int_{K} \left(||\nabla^2 u||^2+\langle \nabla^2 V\nabla u,\nabla u\rangle\right) d \mu+
\\&  
-\nonumber\int_{\partial K}  (H_{\mu} \langle \nabla u, n_x \rangle ^2 -2\langle \nabla_{\partial K} u, \nabla_{\partial K}  \langle \nabla u, n_x \rangle \rangle +\langle \mbox{\rm{II}} \nabla_{\partial K} u, \nabla_{\partial K}  u\rangle )  \,d\mu_{\partial K} (x).
\end{align}

We assume that $\langle \nabla u,n_x\rangle=0:$ indeed, since $\int f=0,$ one may solve the equation $Lu=f$ with the zero Neumann boundary condition, as per the ``moreover'' and ``furthermore b)'' parts of Theorem \ref{exist-n}. Then (\ref{eq-11}) becomes, after regrouping,
$$\int_K f^2 d\mu=$$
\begin{align}\label{eq-22}
-\int_K \left(2\langle \nabla f,\nabla u\rangle+ \langle \nabla^2 V\nabla u,\nabla u\rangle\right) d\mu-\int_{K} \|\nabla^2 u\|^2 d \mu-\int_{\partial K}  \langle \mbox{\rm{II}} \nabla_{\partial K} u, \nabla_{\partial K}  u\rangle  \,d\mu_{\partial K} (x).
\end{align}
In order to prove (\ref{BrLi-convex}), we combine (\ref{eq-22}) with the inequalities
\begin{equation}\label{ech-1}
\int_K \left(2\langle \nabla f,\nabla u\rangle+ \langle \nabla^2 V\nabla u,\nabla u\rangle\right) d\mu\geq -\int_K \langle (\nabla^2 V)^{-1}\nabla f,\nabla f\rangle d\mu;
\end{equation}

\begin{equation}\label{ech-2}
\int_{K} \|\nabla^2 u\|^2 d \mu\geq 0;
\end{equation}

\begin{equation}\label{ech-3}
\int_{\partial K}  \langle \mbox{\rm{II}} \nabla_{\partial K} u, \nabla_{\partial K}  u\rangle )  \,d\mu_{\partial K} \geq 0.
\end{equation}

For the last inequality, convexity of $K$ was used. 

\medskip

Suppose now that the equality holds. Then the equality must hold also in (\ref{ech-1}), (\ref{ech-2}) and (\ref{ech-3}). Suppose $f$ is not a constant function. The equality in (\ref{ech-2}) holds whenever $u$ is a linear function, i.e., for some vector $\theta\in\R^n$ and $c\in\R$, $u=-\langle x,\theta\rangle+c$ and therefore $f=-L\langle x,\theta\rangle =\langle \nabla V,\theta\rangle.$ Moreover, since $\langle \nabla u,n_x\rangle=0,$ the convex set $K$ is such that all of its normal vectors belong to some $(n-1)-$dimensional subspace (unless $\theta=0,$ in which case $f=0$). This means that $K$ is a cylinder parallel to $\theta$.

We note that $f=\langle \nabla V,\theta\rangle$ implies $\nabla f=\nabla^2 V\theta$, and therefore the equality in (\ref{ech-1}) also holds.


Lastly, note that $\nabla_{\partial K} u =-\theta$, and thus when $K$ is a cylinder parallel to $\theta,$ the second fundamental form at $\theta$ is zero everywhere, and therefore
$$\int_{\partial K}  \langle \mbox{\rm{II}} \nabla_{\partial K} u, \nabla_{\partial K}  u\rangle   \,d\mu_{\partial K}= \int_{\partial K}  \langle \mbox{\rm{II}}  \theta, \theta\rangle   \,d\mu_{\partial K}=0,$$
meaning that the equality holds also in (\ref{ech-3}).  
\end{proof}

\begin{remark}
In fact, the equality in (\ref{ech-3}) holds if and only if $K$ is a cylinder, provided that the boundary of $K$ is $C^3$, and $e^{-V}$ is supported on the entire $\R^n$. Indeed, if for some vector $\theta\in\R^n\setminus\{0\},$ we have that for every $x\in \partial K,$ the vector $\theta_x$ (the projection of $\theta$ onto $T_x$, the tangent space of $K$ at $x$) satisfies $\langle \rm{II}\theta_x,\theta_x\rangle=0$, then this means that 
$$\partial K=\{x\in\partial K:\, \langle n_x,\theta\rangle=0\}\cup\{x\in\partial K:G_K(x)=0\},$$
where $G_K(x)=\det \rm{II} (x)$ is the Gauss curvature of $K$ at $x.$ Note that everywhere on the set $A= \{x\in\partial K:\, \langle n_x,\theta\rangle=0\}$, the Gauss curvature is also zero, since $\rm{II} =dn_x=0.$ Since the boundary of $K$ is $C^3,$ the Gauss curvature is continuous. Thus for all points on $\partial K,$ the Gauss curvature is zero. A boundary of a convex set is a regular complete surface, and therefore, by the result of Pogorelov Hartman-Nirenberg \cite{HartNir} and Massey \cite{massey}, $K$ is a cylinder.
\end{remark}

For $k=(k_1,...,k_n),$ where $k_i\in\mathbb{Z},$ let $B_k=\prod^n_{i=1} [k_i,k_i+1].$ We shall make use of the decomposition $\R^n=\cup_{k\in \mathbb{Z}^n} B_k.$

We shall show

\begin{proposition}[Stability in the Brascamp-Lieb inequality -- smooth case]\label{stability}
Let $\mu$ be a probability log-concave measure on $\R^n$ with $C^2$ density $e^{-V}$ supported on the whole $\R^n,$ and suppose $\nabla^2 V>0.$ Fix $\epsilon>0$. Consider a convex set $K$ with $C^2$ boundary and a function $f\in W^{1,2}(K)\cap C^1(K)$ such that 
$$\mu(K)\int_K f^2 d\mu-\left(\int_K f d\mu\right)^2\geq\mu(K)\int_K \langle (\nabla^2 V)^{-1}\nabla f,\nabla f\rangle d\mu-\epsilon.$$
Then there exists a vector $\theta\in\R^n$ (possibly zero), which depends only on $K$ and $f$, such that
\begin{itemize}
\item If $K$ is bounded, we have
\begin{equation}\label{first}
\int_{\partial K} \langle \theta, n_x\rangle^2 d \mu_{\partial K}\leq C_1(K,V)\epsilon,
\end{equation}
where $C_1(K,V)>0$ depends only on $V$ and on the in-radius and the diameter of $K$; and
\begin{equation}\label{second-1}
\|f-\langle \nabla V,\theta\rangle-\frac{1}{\mu(K)}\int_{K} f d\mu\|_{L^1(K,\mu)}\leq c(\epsilon),
\end{equation}
where 
$$c(\epsilon)\leq \sqrt{\mu(K)}\left(\sqrt{\frac{C^2_{poin}(K,\mu)+1}{2}n\epsilon}+\sqrt[4]{\epsilon \int |\nabla V|^2 d\mu|_{K}}\right).$$ 

Furthermore, for the standard Gaussian measure, we have 
$$c(\epsilon)\leq \sqrt{\gamma(K)}\left(\sqrt{n\epsilon}+\sqrt[4]{n\epsilon}\right)$$ (regardless of $K$); if, additionally to $\mu=\gamma$, we have $rB^n_2\subset K,$ then we have also 
$$C_1(K,V)\leq \frac{n+1}{r}.$$

\item If $K$ is not necessarily bounded, then for each $k\in\mathbb{Z}^n,$ we have
$$
\int_{\partial K\cap B_k} \langle \theta, n_x\rangle^2 d \mu_{\partial K}(x)\leq C_1(k)\epsilon,
$$
for some $C_1(k)$ which only depends on $k$, $\mu$ and $K.$
\item In addition, for each $k\in\mathbb{Z}^{n},$ 
\begin{equation}\label{second}
\|f-\langle \nabla V,\theta\rangle-\frac{1}{\mu(K\cap B_k)}\int_{K\cap B_k} f d\mu\|_{L^1(K,\mu)}\leq c_k(\epsilon),
\end{equation}
where
$$c_k(\epsilon)\leq \sqrt{\mu(K\cap B_k)}\left(\sqrt{\frac{C^2_{poin}(K\cap B_k,\mu)+1}{2}n\epsilon}+\sqrt[4]{\epsilon \int |\nabla V|^2 d\mu|_{K\cap B_k}}\right).$$ 
\end{itemize}

\end{proposition}
\begin{proof} Similarly to Proposition \ref{eqchar}, without loss of generality we may assume that $\int f=0$ and consider the function $u\in W^{2,2}(K)\cap C^2(K)$ such that $Lu=f$ and $\langle \nabla u, n_x\rangle=0$ on $\partial K.$ Recall (\ref{eq-22}):

$$\int_K f^2 d\mu=$$
\begin{align}\label{eq-22-again}
-\int_K \left(2\langle \nabla f,\nabla u\rangle+ \langle \nabla^2 V\nabla u,\nabla u\rangle\right) d\mu-\int_{K} \|\nabla^2 u\|^2 d \mu-\int_{\partial K}  \langle \mbox{\rm{II}} \nabla_{\partial K} u, \nabla_{\partial K}  u\rangle   \,d\mu_{\partial K} (x).
\end{align}
Combining (\ref{eq-22-again}) with 
$$
\int_K \left(2\langle \nabla f,\nabla u\rangle+ \langle \nabla^2 V\nabla u,\nabla u\rangle\right) d\mu\geq -\int_K \langle (\nabla^2 V)^{-1}\nabla f,\nabla f\rangle d\mu,
$$
and an application of convexity of $K$
$$
\int_{\partial K}  \langle \mbox{\rm{II}} \nabla_{\partial K} u, \nabla_{\partial K}  u\rangle   \,d\mu_{\partial K} \geq 0,
$$

we get

\begin{equation}\label{eq-33}
\int_K f^2 d\mu\leq
\int_K \langle (\nabla^2 V)^{-1}\nabla f,\nabla f\rangle d\mu-\int_{K} \|\nabla^2 u\|^2 d \mu.
\end{equation}

Under the assumption of the Proposition, this yields

\begin{equation}\label{ineq-hess-u}
\int_{K} \|\nabla^2 u\|^2 d \mu\leq \epsilon.
\end{equation}

Therefore, there exists a vector $\theta\in\R^n$ such that
$$u=-\langle x,\theta\rangle+v$$
with $\int_K \nabla v d\mu=0$ and $\int_{K} \|\nabla^2 v\|^2 d \mu\leq \epsilon.$ By the Poincare inequality (\ref{poin-intro}), this implies that $\int_K |\nabla v|^2 d\mu\leq C^2_{poin}(K,\mu)\epsilon$. Note also that
$$\langle \nabla v,n_x\rangle=\langle \nabla u,n_x\rangle + \langle \theta,n_x\rangle= \langle \theta,n_x\rangle.$$

First, to show (\ref{second}), we note that $f-\langle\nabla V,\theta\rangle=Lu+L\langle x,\theta\rangle=Lv,$ and write, using $(\Delta u)^2\leq n\|\nabla^2 u\|^2,$ for every $\alpha>0$
$$\|f-\langle \nabla V,\theta\rangle\|_{L^1(K,\mu)}=\int_K |Lv| d\mu= \int_K |\Delta v-\langle \nabla v,\nabla V\rangle| d\mu\leq $$$$\sqrt{\mu(K)}\left(\sqrt{n\int_K \|\nabla^2 v\|^2}d\mu+\frac{\alpha}{2}\sqrt{\int_K |\nabla v|^2 d\mu}+\frac{1}{2\alpha}\sqrt{\int_K |\nabla V|^2 d\mu}\right)\leq$$$$\sqrt{\mu(K)}\left(\sqrt{\frac{C^2_{poin}(K,\mu)+1}{2}n\epsilon}+\sqrt[4]{\epsilon \E |\nabla V|^2}\right),$$
where in the last passage we optimized in $\alpha.$ The conclusion follows if $K$ is bounded. If $K$ is not bounded, we consider $K\cap B_k$ instead, and the corresponding conclusion follows in the same manner.

Recall that in the Gaussian case, $C^2_{poin}(K,\mu)\leq 1,$ and the ``furthermore'' estimate on $c(\epsilon)$ follows.

\medskip
\medskip

Next, to show (\ref{first}) for a bounded set $K$, we apply the Trace Theorem \ref{gagliardo-1} to all the partial derivatives of $v,$ sum it up, and use the facts that $\int_K |\nabla v|^2 d\mu\leq C^2_{poin}(K,\mu)\epsilon$ and $\int_K \|\nabla v\|^2 d\mu\leq \epsilon:$
$$\int_{\partial K} |\langle \nabla v,n_x\rangle|^2 d\mu_{\partial K}\leq \int_{\partial K} |\nabla v|^2 d\mu_{\partial K}\leq $$$$C'\int_K (|\nabla v|^2+\|\nabla v\|^2) d\mu\leq C'\cdot (C^2_{poin}(K,\mu) +1)\cdot\epsilon,$$
where $C'$ only depends on the in-radius and the diameter of $K$, and on $V$. The inequality (\ref{first}) follows if we remember that $\langle \theta,n_x\rangle=\langle \nabla v,n_x\rangle.$

\medskip
\medskip 

Lastly, in the case when $K$ is not bounded, given $k=(k_1,...,k_n),$ let $M_{k}=\partial K\cap B_k,$. Applying the Trace Theorem \ref{gagliardo-1} on $M_k$, we get
$$
\int_{M_k} \langle \theta, n_x\rangle^2 d \mu_{\partial K}= \int_{M_k} \langle \nabla v, n_x\rangle^2 d \mu_{\partial K}\leq$$$$ C_1(k)\int_{K\cap B_k} |\nabla v|^2 + \|\nabla^2 v\|^2 d\gamma\leq C_1(k)(1+C^2_{poin}(K,\mu))\epsilon.
$$

To get the ``furthermore'' estimate on $C_1(K,V)$ in the Gaussian case, apply Theorem \ref{GaussGarg} in place of Theorem \ref{gagliardo-1}.
 \end{proof}

Finally, we are ready to prove Theorem \ref{eqchar-intro}.

\medskip

\textbf{Proof of the Theorem \ref{eqchar-intro}.} We shall focus on showing part (1), since part (2) follows from the Gaussian estimates in Proposition \ref{stability}. Suppose $K$ and $f$ give equality in this inequality. For a large enough $R>0,$ let $K_R=K\cap RB^n_2.$ For an arbitrary $\delta>0,$ consider a convex set $K^R_{\delta}$ with $C^{2}$ boundary, such that $\frac{1}{2}K^R\subset K^R_{\delta}\subset K^R$, such that $\mu(K\setminus K^R_{\delta})\leq \delta$, and such that $d_{TV}(\gamma|_{\partial K^R}, \gamma|_{\partial K_{\delta}^R})\leq \delta$. This is possible to do, for example, by selecting $K^R_{\delta}$ very close to $K^R$ in the Hausdorff distance. Indeed, the relation between the measures follows in a straight-forward manner, while the total variation estimate follows since Lebesgue surface area measures converge weakly when bodies converge in Hausdorff distance (see Schneider \cite{book4}), and hence the same holds for the surface area measures of bounded convex sets with respect to $\mu$, in view of our assumptions.  

Since $\mu(K\setminus K^R_{\delta})\leq \delta$ and $K$ gives an equality in the Brascamp-Lieb inequality, we see that there exists an $\epsilon(\delta)>0,$ possibly depending on $\delta$, $V$ and $f,$ such that
	$$
	\mu(K^R_{\delta})\int_{K_{\delta}^R} f^2 d\mu-\left(\int_{K^R_{\delta}} f d\mu\right)^2\geq\mu(K^R_{\delta})\int_{K^R_{\delta}} \langle  (\nabla^2 V)^{-1}\nabla f,\nabla f\rangle d\mu-\epsilon(\delta),
	$$
	such that $\epsilon(\delta)\rightarrow_{\delta\rightarrow 0} 0.$
	By Proposition \ref{stability}, denoting $C_{\delta,R}=\frac{1}{\mu(K^R_{\delta})}\int_{K^R_{\delta}} f d\mu,$ there exists a $\theta_{\delta,R}\in\R^n$ (possibly zero, which depends only on $K_{\delta}^R$ and $f$), such that, for each $k\in\mathbb{Z}^n,$
\begin{equation}\label{f-epsilon}
\|f-\langle \nabla V,\theta_{\delta,R}\rangle-C_{\delta,R}\|_{L^1(K^R_{\delta}\cap B_k,\mu)}\leq c_1(\delta, k),
\end{equation}
with $c_1(\delta, k)$ depending only on $k,$ $\mu$ and possibly $K,$ but not on $R.$ Fix $k\in \mathbb{Z}^n$; in the following paragraph, all the parameters may depend on it. Note that $C_{\delta, R}\rightarrow_{R\rightarrow\infty} C=\frac{1}{\mu(K\cap B_k)}\int_{K\cap B_k} f d\mu$. By (\ref{f-epsilon}), there exists an $r>0,$ possibly depending on $K$, $f$ and $V,$ such that $\theta_{\delta}\in rB^n_2$ whenever $\epsilon(\delta)\in (0,1).$ Therefore, there exists a vector $\theta,$ such that $\theta_{\delta,R}\rightarrow_{\delta\rightarrow 0, R\rightarrow\infty}\theta$, where the convergence is understood in the sense of sub-sequences. Letting $\delta\rightarrow 0,$ we arrive to the conclusion b) of the theorem.

Next, we recall that there exists a function $v=v_{\delta,R}$ on $K_{\delta}^{R}$ with $\int_{K_{\delta}^{R}} \nabla v d\gamma=0$ and $\int_{K_{\delta}^{R}} \|\nabla^2 v\|^2 d \mu\leq \epsilon(\delta),$ and such that $\langle \nabla v,n_x\rangle= \langle \theta_{\epsilon,R},n_x\rangle$ on $\partial K_{\delta}^{R}.$ For $k=(k_1,...,k_n),$ where $k_i\in\mathbb{Z},$ let $M^{\delta,R}_{k}=\partial K_{\delta}^{R}\cap B_k,$ and $M_{k}=\partial K\cap B_k.$ Note that provided that $\delta$ is chosen small enough, and $R$ is chosen large enough (depending on $k$), the set $K^{\delta}_R\cap B_k$ is close to $K\cap B_k$ in Hausdorff distance by $c(\delta,k)$, and the constant from the trace Theorem \ref{gagliardo-1} for $K^{\delta}_R\cap B_k$ is the same, up to a factor of $2$, as that of $K\cap B_k$, which in turn depends only on $K$, $V$ and $k$, but not on $\delta$ or $R.$ We get, using the Trace Theorem \ref{gagliardo-1},
$$
\int_{M^{\delta,R}_k} \langle \theta_{\delta,R}, n_x\rangle^2 d \mu_{\partial K}= \int_{M^{\delta,R}_k} \langle \nabla v, n_x\rangle^2 d \mu_{\partial K}\leq$$$$ C(k)\int_{K_{\delta}^{R}\cap B_k} |\nabla v|^2 + \|\nabla^2 v\|^2 d\mu\leq C(k)\int_{K_{\delta}^{R}}|\nabla v|^2 + \|\nabla^2 v\|^2 d\mu$$
$$C(k)(1+C^2_{poin}(K,\mu))\epsilon,
$$
where in the last passage we used the fact that $C_{poin}(K_{\delta}^{R},\mu)\leq C_{poin}(K,\mu)$ and $\gamma(K_{\delta}^{R})\geq c_0>0$, for some fixed constant $c_0>0$ that does not depend on $\delta, R,$ or $k.$ Finally, for each fixed $k,$ we let $\delta\rightarrow 0$ and $R\rightarrow\infty,$ thereby insuring that $\epsilon(\delta)\rightarrow 0,$ and we therefore get that $\langle \theta,n_x\rangle=0$, with $\theta$ chosen above. In the last step we used also that the surface measure on $M^{\delta,R}_k$ tends to $M_k$ weakly by construction.

Combining the assertion for all $k\in \mathbb{Z}^n,$ we arrive to the conclusion of the Theorem \ref{eqchar-intro}. $\square$

\begin{remark} Note that in the Gaussian case, we have a stability estimate which works for unbounded sets, regardless of their diameter. Thus in the Gaussian case, the approximation argument is easier, and does not require considering the intersection with $B_k.$
\end{remark}

\section{On the Gaussian torsional rigidity}

In this section, the measure $\gamma$ is fixed to be the standard Gaussian on $\R^n,$ and for a convex set $K$ we fix the notation
$$\int:=\frac{1}{\gamma(K)}\int_K d\gamma.$$

\subsection{Torsional rigidity: general discussion} For a convex domain $K$ in $\R^n$ and a function $F\in L^2(K)\cap C(\bar{K}),$ define the $F-$Gaussian torsional rigidity by
$$T^F_{\gamma}(K):=\sup_{v\in W^{1,2}(K,\gamma):\,v|_{\partial K}=0} \frac{\left(\int Fv\right)^2}{\int |\nabla v|^2},$$
where the derivatives are understood in the weak sense, and the boundary value is understood in the sense of the trace operator, as per our convention.

In the case of Lebesgue measure and $F=1$, this object has been studied heavily, see e.g. Polya, Szeg\"o \cite{PolSz}. A rich theory involving isoperimetric inequalities related to the Lebesgue torsional rigidity has been developed in the past couple of centuries, see, for instance, Kohler-Jobin \cite{KoJo}, Brasco \cite{Brasco}, Colesanti, Fimiani \cite{Colesanti-torrig}. In this section, we will see that the extension of this object to the Gaussian setting also has a variety of nice properties. We shall later require most of them for the proof of our main result Corollary \ref{Gauss-main-intro}, as well as Theorem \ref{Gauss}.

The following classical Lemma holds, in fact, in the most abstract setting, and for an arbitrary measure, but for simplicity, we state it for the Gaussian measure on $\R^n.$

\begin{lemma}\label{inf-sup}
Let $K$ be a convex set and $F\in L^2(K,\gamma)\cap Lip(K)$. Then
$$\sup_{v} \frac{\left(\int Fv\right)^2}{\int |\nabla v|^2}=\inf_{u} \int |\nabla u|^2,$$
where the infimum runs over all $W^{1,2}(K,\gamma)$ functions $u: K\rightarrow\R$ with $Lu=F,$ and the supremum runs over all $W^{1,2}(K,\gamma)$ functions $v: K\rightarrow \R$ which vanish on $\partial K.$
\end{lemma}
\begin{proof} Let $u:K\rightarrow\R$ be such that $Lu=F$. On one hand, note that for any $v\in W_0^{1,2}(K,\gamma)$, and any $t\in\R,$
$$\int |\nabla u-t\nabla v|^2\geq 0,$$
and thus
$$\int |\nabla u|^2\geq -2t\int \langle \nabla u,\nabla v\rangle-t^2\int |\nabla v|^2=2t\int Fv-t^2\int |\nabla v|^2.$$
Optimizing in $t,$ we get
$$\int |\nabla u|^2\geq \frac{\left(\int Fv\right)^2}{\int |\nabla v|^2}.$$
In the case when $Lv=F$, the right hand side above equals $\int |\nabla v|^2,$ and therefore
$$\inf_{u} \int |\nabla u|^2=\int |\nabla v_0|^2=\frac{(\int Fv_0)^2}{\int |\nabla v_0|^2},$$
where $v_0 \in W_0^{1,2}(K,\gamma)$ is the unique function with $Lv_0=F$ and vanishing on the boundary (see Theorem \ref{exist-dir}). The inequality above also yields that 
$$\frac{(\int Fv_0)^2}{\int |\nabla v_0|^2}\geq \frac{(\int Fv)^2}{\int |\nabla v|^2},$$
for any function $v$ vanishing on the boundary. Taking the supremum in the inequality above completes the proof. \end{proof}

We shall need a stability estimate for the Lemma \ref{inf-sup}. Below we shall use notation
$$Var_{\gamma|_{\partial K}}(w)=\frac{1}{\gamma^+(\partial K)} \int_{\partial K} w^2 d\gamma|_{\partial K}-\left(\frac{1}{\gamma^+(\partial K)}\int_{\partial K} w d\gamma|_{\partial K}\right)^2.$$
If the boundary variance $Var_{\gamma|_{\partial K}}(w)$ of a function $w$ is small, then it is somewhat close to taking a constant value on the boundary of $K.$ With this in mind, the following statement is a natural stability version for the Lemma \ref{inf-sup}:

\begin{proposition}\label{tor-rig-stab}
Let $K$ be a $C^2$ convex set in $\R^n$ of Gaussian measure $a\in (0,1)$, such that $rB^n_2\subset K$. Fix $F\in W^{1,2}(K,\gamma)\cap Lip(K)$, and let $u\in W^{1,2}(K,\gamma)\cap C^2(K)$ be such that $Lu=F.$ Suppose that 
$$\frac{1}{\gamma(K)}\int_K |\nabla u|^2 d\gamma\leq T_{\gamma}^F(K) +\epsilon.$$
Then
$$Var_{\gamma|_{\partial K}}(u)\leq C(K)\epsilon,$$
where $C(K)=\sqrt{2\pi}\frac{n+1}{r} ae^{\frac{\Phi^{-1}(a)^2}{2}}.$
\end{proposition}
\begin{proof} Let $u_0$ be the unique function such that $Lu_0=F$ and $u_0|_{\partial K}=0$ (see Theorem \ref{exist-dir}). Consider a function $v=u-u_0.$ Then $Lv=0$, by our assumptions. Thus, integrating by parts (strictly speaking, by the definition of the weak solution), we get
$$\int \langle \nabla v,\nabla u_0\rangle =-\int u_0Lv +\frac{1}{\gamma(K)}\int_{\partial K} u_0 \langle \nabla v, n_x\rangle d\gamma_{\partial K}=0.$$
Therefore,
$$\int |\nabla u|^2 =\int |\nabla u_0-\nabla v|^2= \int |\nabla u_0|^2 +\int |\nabla v|^2 =T^F_{\gamma}(K)+\int |\nabla v|^2.$$
By our assumptions, we therefore have
\begin{equation}\label{vgradbnd}
\int |\nabla v|^2 \leq \epsilon.
\end{equation}
Next, by the Gaussian trace theorem \ref{GaussGarg}, for any constant $c_0\in\R,$ we have  
$$\int_{\partial K} (v-c_0)^2 d\gamma|_{\partial K}\leq \frac{1}{r}\int_K (n(v-c_0)^2 +|\nabla v|^2)d\gamma(x).$$
Letting $c_0=\frac{1}{\gamma(K)}\int_K v d\gamma,$ applying the Poincare inequality (\ref{poinc-sect4}) to $v-c_0$, and using (\ref{vgradbnd}), we see that the above is bounded from above by $\gamma(K)\frac{(n+1)\epsilon}{r}.$

Note also that for any $c_0\in\R,$
$$Var_{\gamma|_{\partial K}}(u)\leq \frac{1}{\gamma^+(\partial K)} \int_{\partial K} (v-c_0)^2 d\gamma|_{\partial K}.$$
We conclude that
\begin{equation}\label{conclude-v}
Var_{\gamma|_{\partial K}}(u)\leq \frac{n+1}{r}\frac{\gamma(K)}{\gamma^+(\partial K)} \epsilon.
\end{equation}
It remains to note that by the Gaussian isoperimetric inequality \cite{Bor-isop}, \cite{ST}, 
$$\frac{\gamma(K)}{\gamma^+(\partial K)}\leq \sqrt{2\pi} \gamma(K)e^{\frac{\Phi^{-1}(\gamma(K))^2}{2}}.$$
\end{proof}

\medskip

Next, we show the continuity property of the torsional rigidity.

\begin{proposition}\label{tr-cont}
Suppose $K$ and $L$ are bounded convex domains in $\R^n$, such that $\gamma(K\triangle L)\leq \epsilon$. Let $F\in L^2(K\cup L,\gamma)\cap C(\overline{K\cup L})$. Then 
$$|T^F(K)-T^F(L)|\leq c(\epsilon)\rightarrow_{\epsilon\rightarrow 0} 0,$$
where $C(\epsilon) $ depends only on $K$ and $L$, and $F.$
\end{proposition}
\begin{proof} Note first that it is enough to show that for any convex domains $A$ and $B$ such that $A\subset B$ and $\gamma(B\setminus A)\leq \epsilon,$ we have $|T^F(A)-T^F(B)|\leq c(\epsilon).$ 

Indeed, applying this fact with $A=K$ and $B=K\cap L$, we would get $|T^F(K)-T^F(K\cap L)|\leq c(\epsilon),$ and then, similarly, we could get $|T^F(L)-T^F(K\cap L)|\leq c(\epsilon),$ which would together yield $|T^F(K)-T^F(L)|\leq 2 c(\epsilon).$

Next, by the definition of torsional rigidity, if $A\subset B,$ we have
$$T^F(A)\leq \frac{\gamma(B)}{\gamma(A)}T^F(B)=(1+C^1\epsilon) T^F(B).$$

Therefore, it suffices to focus on the case when $A\subset B,$ $\gamma(B\setminus A)\leq \epsilon,$ and to show that $T^F(B)-T^F(A)\leq c(\epsilon).$ Let $u$ be the function on $A$ with $u|_{\partial A}=0$ and $Lu=F,$ and $v$ be the function on $B$ with $v|_{\partial B}=0$ and $Lv=F.$ Our aim is to show that
$$\frac{1}{\gamma(B)}\int_B |\nabla v|^2 d\gamma-\frac{1}{\gamma(A)}\int_A |\nabla u|^2 d\gamma \leq c(\epsilon).$$
In view of our assumption that $\gamma(B\setminus A)\leq \epsilon,$ it suffices to show
\begin{equation}\label{suffices}	
\int_B |\nabla v|^2 d\gamma-\int_A |\nabla u|^2 d\gamma=\int_{B\setminus A} |\nabla v|^2 d\gamma+\int_A (|\nabla v|^2-|\nabla u|^2)d\gamma \leq 
c(\epsilon).
\end{equation}

Firstly, $\nabla v\in L^2(B)$, and therefore, by the Dominated Convergence Theorem,
$$
\int_{B\setminus A} |\nabla v|^2 d\gamma\leq 
c'(\epsilon)\rightarrow_{\epsilon\rightarrow 0} 0.
$$
Next, integrating by parts, and using $u|_{\partial A}=0$, as well as $Lu=Lv=F,$ we see
\begin{equation}\label{v-u}
\int_A (|\nabla v|^2-|\nabla u|^2)d\gamma =\int_{\partial A} v\langle \nabla (v-u),n_x\rangle d \gamma|_{\partial K}. 
\end{equation}

By the boundary regularity estimate Theorem 14.2 from Gilbarg, Trudinger \cite{GilTru}, since $A$ is convex, we have that $|\nabla u|$ is bounded on $\bar{A}$, and also $|\nabla v|$ is bounded on $\bar{B},$ and the bounds depend only on $n$ and $\sup_A |u|$ (as well as $sup_B |v|$). Thus
\begin{equation}\label{v-u-next}
\int_{\partial A} v\langle \nabla v,n_x\rangle d \gamma|_{\partial K}\leq C_2 \int_{\partial A} v^2 d \gamma|_{\partial K}\leq c(\epsilon).
\end{equation}

By the trace theorem \ref{gagliardo-1}, applied with $K=B\setminus A,$ using compactness, we see that
$$\int_{\partial A} v^2 d \gamma|_{\partial K}\leq  C'\int_{B\setminus A} (v^2+|\nabla v|^2)d\gamma\leq  c_1(\epsilon),$$
where in the last passage we used that $\gamma(B\setminus A)\leq \epsilon$, and by the dominated convergence theorem. Thus (\ref{suffices}) follows.
\end{proof}

\begin{cor}\label{continuityTRapp}
For any convex body $K$, any continuous $F: K\rightarrow\R,$ and any $\epsilon>0$ there exists a convex body $K_{\epsilon}\subset K$ with $C^2$ boundary, such that $K_{\epsilon}$ is $\epsilon-$close to $K$ in Hausdorff distance, and $|T^F(K)-T^F(K_{\epsilon})|\leq c(\epsilon),$ for some $c(\epsilon)\rightarrow_{\epsilon\rightarrow 0} 0,$ which depends only on $K.$
\end{cor}
\begin{proof} We select $K_{\epsilon}$ to be a $C^2$ $\delta$-approximation of $K$ from inside in Hausdorff distance, where $\delta<\epsilon$ is chosen sufficiently small so that $|T^F(K)-T^F(L)|\leq \epsilon,$ which is possible to do by the previous proposition. The crucial observation here is that the constant $C$ from Proposition \ref{tr-cont} is uniformly bounded once $\delta$ is chosen small enough. 

Indeed, it only depends on the in-radius of $K_{\epsilon}$ (which is the same as that of $K$, up to a factor of $2$), and on the supremum of the solution $u_{\epsilon}$ of the equation $Lu_{\epsilon}=F$ on $K_{\epsilon}$ with the Dirichlet boundary condition (as per Theorem \ref{GaussGarg} and Theorem 14.2 from \cite{GilTru}). Letting $v=u-u_{\epsilon},$ where $u$ is the solution of $Lu=F$ on $K$ with the Dirichlet boundary condition, we see that $Lv=0$ and $|v|_{\partial K_{\epsilon}}|\leq C(\epsilon)\rightarrow 0,$ by continuity of $v.$ This yields that $|v_{\epsilon}|<C(\epsilon)$ on $K_{\epsilon},$ by the maximum principle. Thus indeed $|v_{\epsilon}|<C'$ on $K_{\epsilon}$, with $C'$ that only depends on $K$ (and not $K_\epsilon$). We conclude that the constant $C$ from Proposition \ref{tr-cont} is uniformly bounded once $\delta$ is chosen small enough, and the corollary follows. \end{proof}

Lastly, we shall need

\begin{lemma}\label{infiniteapproximation}
For any convex set $K$ with non-empty interior, any continuous $F:K\rightarrow\R,$ and any $\epsilon>0,$ there exists an $R>0$ such that 
$$|T^F(K)-T^F(K\cap R B^n_2)|\leq \epsilon.$$
\end{lemma}
\begin{proof} Firstly, we note that
$$T^F(K\cap R B^n_2)\leq \frac{\gamma(K)}{\gamma(K\cap R B^n_2)}T^F(K),$$
and select $R_1$ to be large enough so that the above yields 
$$T^F(K\cap R_1 B^n_2)-T^F(K)\leq \epsilon.$$

Next, let $v:K\rightarrow \R$ be the solution of $Lv=F$ and $v|_{\partial K}=0$ and $v_R:K\rightarrow \R$ be the solution of $Lv_R=F$ and $v|_{\partial K_R}=0$. Suppose $R>R_2$ for some $R_2>0.$ Then
$$\gamma(K)\left(T^F(K)-T^F(K\cap R B^n_2)\right)\leq\int_{K\setminus R_2B^n_2} |\nabla v|^2 d\gamma+\int_{K\cap R_2 B^n_2} |\nabla v|^2-|\nabla v_R|^2 d\gamma.$$
Select $R_2$ to be large enough so that $\int_{K\setminus R_2B^n_2} |\nabla v|^2 d\gamma\leq \epsilon.$

Finally, integrating by parts, we see
$$\int_{K\cap R_2 B^n_2} |\nabla v|^2-|\nabla v_R|^2 d\gamma=$$$$\int_{\partial K\cap R_2 B^n_2} \langle \nabla v-\nabla v_R,n_x\rangle (v+v_R) d\gamma|_{\partial K}+ \int_{K\cap R_2 \sfe} \langle \nabla v-\nabla v_R,n_x\rangle (v+v_R) d\gamma|_{R\sfe}.$$
Since $v$ and $v_R$ take value $0$ on $\partial K\cap R B^n_2,$ and in view of the boundary regularity estimate Theorem 14.2 from Gilbarg, Trudinger \cite{GilTru} (which yields that $\langle \nabla v-\nabla v_R,n_x\rangle$ is bounded almost everywhere on $\partial K\cap R_2 B^n_2$), we get that
$$\int_{\partial K\cap R_2 B^n_2} \langle \nabla v-\nabla v_R,n_x\rangle (v+v_R) d\gamma|_{\partial K}=0.$$
Lastly, we write 
$$\int_{K\cap R_2 \sfe} \langle \nabla v-\nabla v_R,n_x\rangle (v+v_R) d\gamma|_{R\sfe}\leq $$$$ 10\sqrt{\int_{K\cap R_2 \sfe} |\nabla v|^2 + |\nabla v_R|^2}\cdot \sqrt{\int_{K\cap R_2 \sfe} v^2 +v_R^2 },$$
which can be made arbitrarily small in view of the fact that $v, v_R, |\nabla v|, |\nabla v_R|$ are all square integrable on $K,$ provided that $R_2$ is selected appropriately. Letting $R=\max(R_1,R_2)$ finishes the proof.
\end{proof}

\medskip
\medskip


\medskip

\subsection{Isoperimetric inequalities for the Gaussian torsional rigidity}

First, let us consider the case $F=1$ and use notation $T_{\gamma}(K)=T^1_{\gamma}(K)$. We notice an isoperimetric inequality for $T_{\gamma},$ analogous to the Saint-Venant theorem (see, e.g. Polya-Szeg\"o \cite{PolSz} for the Lebesgue version):

\begin{proposition}[Gaussian version of the Saint-Venant inequality]\label{GSV}
Fix $a\in [0,1]$ and let $K$ be a convex set with $\gamma(K)=a.$ Then 
$$T_{\gamma}(K)\leq T_{\gamma}(H_a),$$
where $H_a$ is the half-space of Gaussian measure $a.$
\end{proposition}
\begin{proof} Let $v_0$ vanish on the boundary of $K$, $Lv_0=1,$ and let $v_0^*$ be its Ehrhard symmetrization $v_0^*: H_a\rightarrow\R$. By the definition of the Ehrhard symmetrization, $v_0^{*}$ also vanishes on the boundary of $H_a$, and $\int v_0^*=\int v_0$. By Lemma \ref{EhrPr} (the Ehrhard principle),
$$\int_{H_a} |\nabla v_0^*|^2 d\gamma\leq \int_{K} |\nabla v_0|^2 d\gamma.$$ 
Thus, in view of Lemma \ref{inf-sup},
$$T_{\gamma}(K)=\frac{\left(\int_K v_0 d\gamma\right)^2}{\int_K |\nabla v_0|^2 d\gamma}\leq \frac{\left(\int_{H_a} v^*_0 d\gamma\right)^2}{\int_{H_a} |\nabla v^*_0|^2 d\gamma}\leq T_{\gamma}(H_a).$$
\end{proof}

In what follows, we shall discuss a generalization of the Proposition \ref{GSV}.

\medskip


Let us first discuss a Gaussian analogue of Talenti's inequality. For some background on the usual (Lebesgue) Talenti's inequality \cite{Talenti}, see also books by Vazquez \cite{Vasquez} or Kesavan \cite{Kesavan}. The following proposition was shown by Betta, Brock, Mercaldo, Posteraro \cite{gady}, as was discovered by the author after this manuscript was written. We shall outline the proof for the reader's convenience.

\begin{proposition}[Gaussian analogue of Talenti's inequality]\label{talenti}
Let $F\in L^2(K,\gamma)$, for some measurable set $K$ in $\R^n,$ and suppose $F\geq 0.$ Let $u:K\rightarrow\R$ be the weak solution of the equation $Lu=-F$ on $K$ with the Dirichlet boundary condition $u|_{\partial K}=0$. Let $H_K=\{x\in\R^n:\,x_1\leq \Phi^{-1}(\gamma(K))\}$ (the half-space of the same Gaussian measure as $K$), and let $v$ on $H_K$ be given by $Lv=-F^*,$ with $v|_{\partial H_K}=0.$ Then $v\geq u^*$ everywhere on $H_K.$
\end{proposition}

First, we formulate two lemmas. 

\begin{lemma}\label{maxprinc}
Let $u:K\rightarrow\R$ be the weak solution of the equation $Lu=-F$ on $K$ with the Dirichlet boundary condition $u|_{\partial K}=0$, with $F\geq 0.$ Then $u\geq 0.$
\end{lemma}
\begin{proof} This classical fact is a straight-forward consequence of the maximum principle, and we outline this simple argument for the reader's convenience, in the case when $u$ is the strong solution of class $C^2$ (and the general case follows by approximation). Suppose $u(x)<0$ at some point $x\in int(K).$ As $u\in C^2(K),$ and since $u|_{\partial K}=0,$ there exists a point $x_0\in int(K)$ which is a local minima for $u.$ Therefore, $\nabla^2 u(x_0)>0$; additionally, $\nabla u(x_0)=0$, and thus $Lu|_{x_0}=\Delta u|_{x_0}$. However, $Lu|_{x_0}\leq 0$ by our assumption that $Lu=-F$ for $F\geq 0$. In other words, $tr(\nabla^2 u)|_{x_0}\leq 0,$ leading to the contradiction with the fact that $\nabla^2 u(x_0)>0.$	
\end{proof}

Next, we formulate the Gaussian analogue of the Hardy-Littlewood inequality. The reader may find the Lebesgue version, with an analogous proof, e.g. in Burchard \cite{Almut}.

\begin{lemma}[Gaussian Hardy-Littlewood inequality]\label{HLrinc}
Let $K$ be a measurable set and suppose $f,g: K\rightarrow\R$ are non-negative measurable functions in $L^2(\gamma, K)$. Then
$$\int_K fg d\gamma(x)\leq \int_{H_K} f^* g^* d\gamma(x).$$
\end{lemma}
\begin{proof} We write
$$\int_K fg d\gamma(x)=\int_K \int_0^{\infty} \int_0^{\infty} 1_{\{f(x)>t\}} 1_{\{g(x)>t\}}dtds d\gamma(x)=$$$$\int_0^{\infty} \int_0^{\infty}\gamma\left(\{f>s\}\cap\{g>t\}\right)ds dt,$$
and similarly,
$$\int_{H_K} f^*g^* d\gamma(x)=\int_0^{\infty} \int_0^{\infty}\gamma\left(\{f^*>s\}\cap\{g^*>t\}\right)ds dt.$$
Thus it suffices to show that for any pair of measurable sets $A$ and $B,$
$$\gamma(A\cap B)\leq \gamma(H_A\cap H_B).$$
Indeed, since $H_A$ and $H_B$ are hyperplanes, and one of them is contained in the other, we have 
$$\gamma(H_A\cap H_B)=\min(\gamma(H_A),\gamma(H_B))= \min(\gamma(A),\gamma(B))\geq \gamma(A\cap B),$$
and the lemma follows.	
\end{proof}

\textbf{Proof of the Proposition \ref{talenti}.} For every $t\geq 0,$ integrating by parts, and using the assumption $Lu=-F,$ we get 
$$\int_{\{u>t\}} |\nabla u|^2 d\gamma = -\int_{\{u>t\}} uLu d\gamma+\int_{\partial \{u>t\}} u\langle \nabla u,n_x\rangle d\gamma_{\partial \{u>t\}}=$$$$ \int_{\{u>t\}} uF d\gamma+t\int_{\partial \{u>t\}} \langle \nabla u,n_x\rangle d\gamma_{\partial \{u>t\}}= \int_{\{u>t\}} uF d\gamma-t \int_{\{u>t\}} F d\gamma.$$
Therefore, 
$$\frac{d}{dt} \int_{\{u>t\}} |\nabla u|^2 d\gamma =\frac{d}{dt}\left(\int_{\{u>t\}} uF d\gamma-t \int_{\{u>t\}} F d\gamma\right)=$$$$t \int_{\partial \{u>t\}} F d\gamma_{\partial \{u>t\}} - \int_{\{u>t\}} F d\gamma -t \int_{\partial \{u>t\}} F d\gamma_{\partial \{u>t\}},$$
and we conclude that
\begin{equation}\label{tal1}
\frac{d}{dt} \int_{\{u>t\}} |\nabla u|^2 d\gamma = - \int_{\{u>t\}} F d\gamma. 
\end{equation}
Next, by Cauchy-Schwartz inequality, for every $h>0$,
$$\left(\frac{1}{h}\int_{u\in [t,t+h]} |\nabla u| d\gamma\right)^2\leq \frac{1}{h}\int_{u\in [t,t+h]} d\gamma\cdot\frac{1}{h}\int_{u\in [t,t+h]} |\nabla u|^2 d\gamma,$$
and thus
\begin{equation}\label{tal2}
\left(\frac{d}{dt} \int_{\{u>t\}} |\nabla u| d\gamma\right)^2 \leq -\phi'_u(t)\cdot\frac{d}{dt} \int_{\{u>t\}} |\nabla u|^2 d\gamma,
\end{equation}
where 
$$\phi_u(t)=P_{\gamma}(u<t)= P_{\gamma}(u^*<t).$$
Here $P_{\gamma}$ stands for the probability with respect to the Gaussian measure, and the last equality follows from the properties of the Ehrhard rearrangement. Combining (\ref{tal1}) and (\ref{tal2}), we get
\begin{equation}\label{tal3}
\left(\frac{d}{dt} \int_{\{u>t\}} |\nabla u| d\gamma\right)^2 \leq \phi'_u(t)\cdot\int_{\{u>t\}} F d\gamma.
\end{equation}
Next, by the co-area formula,
$$\int_{\{u>t\}} |\nabla u| d\gamma = \int_t^{\infty} \gamma^+(\partial \{u>s\}) ds,$$
and therefore,
\begin{equation}\label{tal4}
\frac{d}{dt} \int_{\{u>t\}} |\nabla u| d\gamma=-\gamma^+ (\partial \{u>t\}). 
\end{equation}
Combining (\ref{tal3}) and (\ref{tal4}), we get
\begin{equation}\label{tal5}
\gamma^+ (\partial \{u>t\})^2 \leq \phi'_u(t)\cdot\int_{\{u>t\}} F d\gamma.
\end{equation}
Next, by the Gaussian isoperimetric inequality \cite{ST}, \cite{Bor-isop}, 
\begin{equation}\label{tal6}
\gamma^+(\partial \{u>t\}) \geq \gamma^+(\partial \{u^*>t\})=\frac{1}{\sqrt{2\pi}} e^{-\frac{\Phi^{-1}\circ \phi_u(t)^2}{2}}.
\end{equation}
Combining (\ref{tal5}) and (\ref{tal6}), we get
\begin{equation}\label{tal7}
e^{\Phi^{-1}\circ \phi_u(t)^2}\phi'_u(t)\geq \left(2\pi\int_{\{u>t\}} F d\gamma\right)^{-1}.
\end{equation}
Next, using Lemma \ref{HLrinc} (the Gaussian Hardy-Littlewood inequality), with $f=F$ and $g=1_{\{u>t\}}$, we get 
\begin{equation}\label{tal8}
\int_{\{u>t\}} F d\gamma\leq \int_{\{u^*>t\}} F^* d\gamma.
\end{equation}
Combining (\ref{tal7}) and (\ref{tal8}), and using the fact that $\phi_{u^*}=\phi_u,$ we conclude
\begin{equation}\label{tal9}
e^{\Phi^{-1}\circ \phi_{u^*}(t)^2}\phi'_{u^*}(t)\geq \left(2\pi\int_{\{u^*>t\}} F^* d\gamma\right)^{-1}.
\end{equation}
Note that
$$\phi_{u^*}(t)=\Phi\circ (u^*)^{-1}(t),$$
and thus
$$\phi'_{u^*}(t)=\frac{1}{\sqrt{2\pi}}e^{\frac{(u^*)^{-1}(t)^2}{2}} ((u^*)^{-1})'_t.$$
Let us also denote 
$$g(x)=\int_{-\infty}^x F^*(se_1) d\gamma(s).$$
Then (\ref{tal9}) rewrites as 
\begin{equation}\label{tal10}
e^{\frac{(u^*)^{-1}(t)^2}{2}}((u^*)^{-1})'_t \cdot g\circ (u^*)^{-1}(t)\geq \frac{1}{\sqrt{2\pi}}.
\end{equation}
Now, if $K$ is a half-space and $F=F^*$, then we have an equality in (\ref{tal2}) (since in that case, $u=u^*$ and $|\nabla u|=const$ on $\{u=t\}$); we have an equality in the isoperimetric inequality (\ref{tal6}); lastly, we have an equality in (\ref{tal9}). Therefore, we have an equality in (\ref{tal10}) in this case. In other words, recalling our assumptions on $v$, we have
\begin{equation}\label{tal11}
e^{\frac{v^{-1}(t)^2}{2}}(v^{-1})'_t \cdot g\circ v^{-1}(t)=\frac{1}{\sqrt{2\pi}}.
\end{equation}
Let us denote by $h$ such a function that
$$h'(x)=e^{\frac{x^2}{2}}g(x).$$
Combining (\ref{tal10}) and (\ref{tal11}), we see that for all $t\geq 0,$
\begin{equation}\label{tal12}
(h\circ(u^*)^{-1})'_t \geq (h\circ v^{-1})'_t.
\end{equation}
Note that by our Dirichlet boundary assumptions, $(u^*)^{-1}(0)=v^{-1}(0)=\Phi^{-1}(\gamma(K))$. Combining this with (\ref{tal12}), we get, for all $t\geq 0,$
\begin{equation}\label{tal13}
h\circ(u^*)^{-1}\geq h\circ v^{-1}.
\end{equation}
Since $h'\geq 0,$ we see that $h$ is a non-decreasing function, and therefore, (\ref{tal13}) implies that for all $t\geq 0,$
\begin{equation}\label{tal14}
(u^*)^{-1}(t)\geq v^{-1}(t),
\end{equation}
and therefore, $u^*\leq v. \square$

\medskip
\medskip

From Proposition \ref{talenti}, we deduce the following isoperimetric fact, which is a generalization of Proposition \ref{GSV}.

\begin{cor}
For any measurable $K$, and any square-integrable non-sign-changing $F:K\rightarrow \R,$ we have $T^F(K)\leq T^{F^*}(H_K).$
\end{cor}
\begin{proof} Note that $T^F(K)=T^{-F}(K),$ and without loss of generality, suppose that $F\geq 0$. By Lemma \ref{inf-sup}, 
$$T^F(K)=\int_K |\nabla u|^2 d\gamma=-\int_K uF d\gamma;$$$$ T^{F^*}(H_K)=\int_{H_K} |\nabla v|^2 d\gamma=-\int_{H_K} vF^* d\gamma, $$
where $Lu=F$, $u|_{\partial K}=0,$ and $Lv=F^*$, $v|_{\partial H_K}=0.$ By Lemma \ref{maxprinc} (applied with $-u$ rather than $u$), we get $u\leq 0$ on $K$, and $v\leq 0$ on $H_K.$ By Lemma \ref{HLrinc} (the Gaussian Hardy-Littlewood inequality), applied with $f=-u$ and $g=F,$
$$-\int_K uF d\gamma\leq -\int_{H_K} u^*F^* d\gamma.$$
By Theorem \ref{talenti}, $-u^*\leq -v$ point-wise, and thus, in view of all of the above, $T^F(K)\leq T^{F^*}(H_K).$
\end{proof}

\medskip

\subsection{Some lower bounds for torsional rigidity}

In this subsection, we present several lower bounds for Gaussian torsional rigidity, and several of them will be important for us in Section 6. Recall the notation:
$$\varphi(x)=\frac{1}{\sqrt{2\pi}}\int_{-x}^x e^{-\frac{s^2}{2}} ds$$
and
$$J_{n-1}(R)=\int_0^R t^{n-1} e^{-\frac{t^2}{2}} dt,$$
and also $c_{n-1}=J_{n-1}(\infty).$ We start by noticing

\begin{lemma}\label{ball}
Let $K$ be a convex symmetric set in $\R^n$ and suppose $\gamma(K)\geq a.$ Then the inradius of $K$, denoted $r(K),$ satisfies
$$r(K)\geq 0.5\varphi^{-1}(a).$$
\end{lemma}
\begin{proof} Indeed, suppose not. Consider the smallest symmetric strip $S$ containing $K$. Then $S$ also does not contain $0.5\varphi^{-1}(a)B_2^n$, and thus has width smaller than $\varphi^{-1}(a)$. But then the Gaussian measure of $S$ is smaller then $a,$ arriving to a contradiction.
\end{proof}

As an immediate consequence, we get

\begin{lemma} Let $K$ be convex symmetric set with $\gamma(K)\geq a,$ for $a\in [0,1]$. Then 
$$T_{\gamma}(K)\geq \frac{J_{n-1}(0.4\varphi^{-1}(a))}{ac_{n-1}}T_{\gamma}(0.4\varphi^{-1}(a)B_2^n).$$
\end{lemma}
\begin{proof} By Lemma \ref{ball}, $K$ contains a centered ball $0.5\varphi^{-1}(a)B_2^n$. The lemma thus follows by considering the function $v$ on $K$ with $Lv=1$ on $0.4\varphi^{-1}(a)B_2^n$, and vanishing on $K\setminus 0.5\varphi^{-1}(a)B^n_2$ (which exists by Theorem \ref{exist-dir}), and the definition of torsional rigidity.
\end{proof}

We shall make use of the following bound:

\begin{lemma}\label{bound2}
Let $K$ be convex symmetric body with $\gamma(K)=a,$ for $a\in [0,1]$. Then 
$$T_{\gamma}(K)\geq \frac{\varphi^{-1}(a)^2}{4e^2n^2}.$$
Furthermore,
$$T_{\gamma}(K)\geq \frac{\varphi^{-1}(a)^2}{4a^2}\sup_{t\in [0,1]} \gamma(t B_a)^2(1-t)^2,$$ 
where $B_a$ is the centered euclidean ball of measure $a.$
\end{lemma}
\begin{proof} Let $v:K\rightarrow\R$ be given by $v(x)=1-\|x\|_K$. It vanishes on the boundary of $K,$ and thus
$$T_{\gamma}(K)\geq \frac{(\int v)^2}{\int |\nabla v|^2}=\frac{(1-\E\|X\|_K)^2}{\int |\nabla \|x\|_K|^2}.$$
Note that $\nabla \|x\|_K\in\partial K^o,$ and $K^o\subset \frac{1}{r(K)}B^n_2$. Thus 
$$|\nabla \|x\|_K|\leq \frac{1}{r(K)}\leq \frac{2}{\varphi^{-1}(a)},$$
where in the last passage we used Lemma \ref{ball}. In addition, note that for any $t\in [0,1],$
$$1-\E\|X\|_K\geq \frac{\gamma(tK)(1-t)}{\gamma(K)}\geq t^n(1-t),$$
where in the last passage we used a rough lower bound of $t^n$ for $\frac{\gamma(tK)}{\gamma(K)}$, which follows as $\gamma$ is ray-decreasing. Plugging $t=\frac{n-1}{n},$ the statement follows from the estimates above.

To get the second bound, we use Lemma \ref{moments} in place of the bound $\gamma(tK)\geq t^n\gamma(K)$, which implies that for any $t\in [0,1],$ one has $\gamma(tK)\geq \gamma(tB_a)$, where $B_a$ is the centered ball of measure $a.$
\end{proof}

\medskip

In analogy with Lemma \ref{bound2}, we get the following fact, which will be crucial in the next section:

\begin{lemma}\label{last-touch}
For any $F\in L^2(K)$, 
$$T^F_{\gamma}(K)\geq \frac{r(K)^2\left(\E \left[ F(1-\|X\|^2_K)\right]\right)^2}{4\E \|X\|^2_K},$$
where $r(K)$ is the in-radius of $K.$ Furthermore, if $K=R B^k_2\times \R^{n-k}$ is a round $k$-cylinder, and $F(x)=k-\sum_{i=1}^k x^2_i$, then we have an equality in the estimate above.
\end{lemma}
\begin{proof}
We use Lemma \ref{inf-sup} to argue that for any $v\in W^{1,2}(K)$ with $v|_{\partial K}=0,$ one has
$$T^F_{\gamma}(K)\geq \frac{\left(\int Fv\right)^2}{\int |\nabla v|^2}.$$
We plug $v=1-\|x\|^2_K$ to get the bound
\begin{equation}\label{eqholdscyl}
T^F_{\gamma}(K)\geq \frac{\left(\int F(1-\|x\|^2_K)\right)^2}{4\int \|x\|_K^2|\nabla \|x\|_K|^2}.
\end{equation}
It remains to recall that $\nabla \|x\|_K \in \partial K^o,$ and $K^o\subset \frac{1}{r(K)}B^n_2$, and therefore 
$$| \nabla \|x\|_K |\leq \frac{1}{r(K)}.$$

When $K=R B^k_2\times \R^{n-k}$ and $F(x)=k-\sum_{i=1}^k x^2_i$, we have $c\cdot F=L(1-\|x\|^2_K)$, and thus the equality holds in (\ref{eqholdscyl}) by Lemma \ref{inf-sup}. In addition, in this case we have $|\nabla \|x\|_K|= \frac{1}{r(K)},$ and thus the equality must hold in the estimate of the Lemma. 
\end{proof}

\medskip




We conclude with some natural questions about the Gaussian torsional rigidity.

\begin{question}
\begin{enumerate}
	\item Over all convex sets of measure $a\in [0,1],$ which set minimizes $T_{\gamma}(K)$?
	\item Over all symmetric convex sets of measure $a\in [0,1],$ which set minimizes $T_{\gamma}(K)$?
	\item Over all symmetric convex sets of measure $a\in [0,1],$ which set maximizes $T_{\gamma}(K)$?
	\item Over all symmetric convex sets of measure $a\in [0,1],$ which set maximizes (minimizes) $T^{n-x^2}_{\gamma}(K)$?
\end{enumerate}
\end{question}

\section{The Gaussian case, Theorems \ref{Gauss-main-intro} and \ref{Gauss}.}

In this section we discuss estimates in the case of the standard Gaussian measure $\gamma$ on $\R^n$, with density $\sqrt{2\pi}^{-n} e^{-x^2/2}$. We shall use the notation
$$\int:=\frac{1}{\gamma(K)}\int_K d\gamma(K),$$
where $K$ is a convex set in $\R^n$ (in each instance it will be clear from the context which convex set is considered.)

Most of this section is dedicated to proving Theorem \ref{Gauss-main-1}. Let us first deduce Corollary \ref{Gauss-main-intro} from it; first, we formulate the precise version of Corollary \ref{Gauss-main-intro}:

\begin{cor}\label{Gauss-main}
For any symmetric convex set $K$ in $\R^n,$
$$p_s(K,\gamma)\geq$$$$ \sup_{\alpha\in\R}\frac{0.5\left(\E \|X\|^2_K\right)^{-1}r(K)^2\left(\E (1-\|X\|^2_K)\cdot (\alpha-\E \|X\|^2_K) \right)^2-\E \|X\|^4_K +\left(\E \|X\|^2_K\right)^2}{(\alpha-\E \|X\|^2_K)^2}$$$$+\frac{1}{n-\E X^2}.$$

Moreover, the equality holds if and only if $K=RB^k_2\times \R^{n-k}$ for some $R>0$ and $k=1,...,n.$
\end{cor}

\begin{remark}
A direct computation reveals that the optimal choice of $\alpha$ in Theorem \ref{Gauss-main} is
$$\alpha=\frac{1+4r(K)^2 \E \|X\|^2_K\left(\E \|X\|^2_K-\E \|X\|^4_K\right)}{2r(K)^2 \E \|X\|^2_K(1-\E \|X\|^2_K) }.$$
\end{remark}

\medskip

\textbf{Proof of the Corollary \ref{Gauss-main}.} Combining Theorem \ref{Gauss-main-1} and Lemma \ref{last-touch} applied with $F=\alpha-\|x\|^2_K$, we arrive to the inequality in the Corollary\ref{Gauss-main}. The equality cases of Theorem \ref{Gauss-main-1} show that only round $k$-cylinders can be the equality cases in Theorem \ref{Gauss-main}. Proposition \ref{cyl} gives an expression for $p^s_{\gamma}$ for round $k$-cylinders, and a direct computation shows that it coincides with the one given by Corollary \ref{Gauss-main}. Alternatively, one may use the fact that round $k$-cylinders are the equality cases of Theorem \ref{Gauss-main-1}, and also are equality cases of Lemma \ref{last-touch} (in view of the fact that $F=k-\sum_{i=1}^k x^2_i$ in this case), to conclude that they are, indeed, also the equality cases of Corollary \ref{Gauss-main}. $\square$

\medskip
\medskip

\subsection{The key estimate.}

We outline the following proposition, which substantially improves upon the estimate of Eskenazis and Moschidis \cite{EM}. For the reader's convenience, the following proposition is stated together with equality cases, even though so far, we assume smoothness. We shall get rid of the smoothness assumption later on. Recall the notation $\int=\frac{1}{\gamma(K)}\int_K d\gamma.$

\begin{proposition}\label{propGauss}
For any symmetric convex set $K$ in $\R^n$ with $C^2$-smooth boundary and any even $u\in W^{2,2}(K)\cap C^2(int \,K)$,
$$\int \|\nabla^2 u\|^2\geq \int |\nabla u|^2 +\frac{(\int Lu)^2}{n-\E X^2}.$$
Moreover, the equality holds if and only if either 
\begin{itemize}
\item $u=C_1 x^2+C_2,$ for some $C_1,C_2\in\R$ or
\item there exists a rotation $R$ such that $RK=L\times \R^{n-k}$, for some $k-$dimensional symmetric convex set $L\subset \R^k,$ and $u(Rx)=\sum_{i=1}^n \alpha_i x^2_i$, for some real numbers $\alpha_1,...,\alpha_n,$ such that $\alpha_1=...=\alpha_{n-k}.$
\end{itemize}
\end{proposition}
\begin{proof} Let $u=v+t\frac{x^2}{2}$, for some $t\in\R.$ Then
\begin{equation}\label{111}
\|\nabla^2 u\|^2=\|\nabla^2 v\|^2+2t\Delta v+t^2n=\|\nabla^2 v\|^2+2tLv+2t\langle \nabla v,x\rangle+t^2n.
\end{equation}
Note also that
\begin{equation}\label{222}
Lv=Lu-tL\frac{x^2}{2}=Lu-t(n-x^2).
\end{equation}
Using that $K$ is symmetric and $u$ is even, and consequently $\int \nabla v=0$, and applying the Gaussian Poincare inequality (\ref{poinc-sect4}) to each of the $\frac{\partial v}{\partial x_i},$ $i=1,...,n$ we get
\begin{equation}\label{333}
\int\|\nabla^2 v\|^2\geq \int |\nabla v|^2.
\end{equation}
Combining (\ref{111}), (\ref{222}), (\ref{333}), and the fact that $\nabla u=\nabla v+tx,$ we get
\begin{equation}\label{1111}
\int \|\nabla^2 u\|^2\geq \int |\nabla v|^2+2t\langle \nabla v,x\rangle+t^2n+2t Lu -2t^2(n-x^2)=\int |\nabla u|^2+2t Lu +t^2(x^2-n)
\end{equation}
Plugging the optimal $t=\frac{\int Lu}{n-\int x^2}$ yields
$$\int \|\nabla^2 u\|^2\geq \int |\nabla u|^2+\frac{(\int Lu)^2}{n-\E X^2},$$
arriving to the inequality of the Proposition.

In order to characterize the equality cases, suppose that the equality occurs. Then the equality occurs in (\ref{333}), which means that the equality occurs in the Brascamp-Lieb inequality (\ref{BrLi}) for every $\frac{\partial v}{\partial x_i}$. By Proposition \ref{eqchar}, either $K$ is arbitrary and $v$ is an affine function (which means, in view of symmetry, that $v$ is a constant function), or $K$ is a cylinder, and there exists a collection of vectors $\theta_1,...,\theta_n\in\R^n$ (some of which may be zero), and constants $c_1,...,c_n\in\R,$ such that $\frac{\partial v}{\partial x_i}=\langle x,\theta_i\rangle+c_i$, and $\{t\theta_i:t\in\R\}\subset K$. Let us discuss this case by case.

\textbf{Case 1.} $v=C,$ for some constant $C\in\R.$ Then $u=C_1 x^2+C_2,$ for some $C_1,C_2\in\R,$ while $K$ is arbitrary.

\textbf{Case 2.} Suppose $v=C+\langle x,\theta\rangle,$ for some non-zero vector $\theta.$ Then $u=C+\langle x,\theta\rangle+C_1 x^2,$ which can only be an even function if $\theta=0,$ bringing us back to the first case.

\textbf{Case 3.} Suppose $K$ is a cylinder, and $\frac{\partial v}{\partial x_i}=\langle x,\theta_i\rangle+c_i$, and $\{t\theta_i:t\in\R\}\subset K$, for all $i=1,...,n$. Thus, since $v$ is even, we have $v=\langle Ax,x\rangle+C,$ where $A$ is an $n\times n$ matrix such that $A^Te_i=\theta_i$, and therefore $span(A^Te_i)\subset K.$ Therefore, there exists a rotation $R$ such that $RK=L\times \R^{n-k}$, for some $k-$dimensional symmetric convex set $L\subset \R^k,$ and $v(Rx)=\sum_{i=n-k+1}^n \beta_i x^2_i$, for some real numbers $\beta_i$. Recalling the relation between $u$ and $v$, and since $x^2=|Rx|^2$, we have $u(Rx)=\sum_{i=1}^n \alpha_i x^2_i$, for some real numbers $\alpha_1,...,\alpha_n,$ such that $\alpha_1=...=\alpha_{n-k}.$
\end{proof}

\medskip
\medskip

\subsection{Proof of the inequality of the Theorem \ref{Gauss-main-1}}\label{ineq-proof} For the inequality part, we may assume, without loss of generality, that $K$ is $C^2$-smooth. By the Proposition \ref{key_prop-general}, it is enough to show that for any $f\in C^1(\partial K)$ there exists a $u\in C^2(K)\cap W^{1,2}(K)$ with $\langle \nabla u,n_x\rangle=f(x)$ on $x\in\partial K$, and such that
$$\frac{1}{\mu(K)}\int \|\nabla^2 u\|^2+|\nabla u|^2\geq p_s(K,\gamma)\left(\int Lu\right)^2+Var(Lu),$$
with $p_s(K,\gamma)$ satisfying the estimate of the Theorem. By Theorem \ref{exist-n}, one may choose $u$ to be the solution of $Lu=F$ on $K,$ for some $F\in L^2(K,\gamma)\cap Lip(K),$ and $\langle \nabla u,n_x\rangle=f(x)$ on $x\in\partial K$, where $\int F=\frac{1}{\gamma(K)}\int_{\partial K} f d\gamma|_{\partial K}$. In the case when $\int_{\partial K} f d\gamma|_{\partial K}=0,$ we simply take $F=0$, and thus the desired conclusion holds for an arbitrary value of $p_s(K,\gamma).$ Therefore, by homogeneity, we have
$$p_s(K,\gamma)\geq \sup_{F\in L^2(K,\gamma)\cap Lip(K):\,\int F\neq 0}\inf_{u: Lu=F} \frac{ \int \|\nabla^2 u\|^2+|\nabla u|^2-Var(F)}{(\int F)^2}.$$
Thus, as a corollary of the Proposition \ref{propGauss}, we get that for a  $C^2-$smooth convex set $K,$ and for any $F\in L^2(K,\gamma)\cap Lip(K)$ with $\int F\neq 0,$
$$ p_s(K,\gamma)\geq  \inf_{u\in C^2(int K)\cap W^{2,2}(K,
\gamma): Lu=F} \frac{ 2\int |\nabla u|^2-Var(F)}{(\int F)^2}+\frac{1}{n-\E X^2}.$$
By Lemma \ref{inf-sup}, for any $u$ with $Lu=F,$ we have
$$\int |\nabla u|^2\geq T^F_{\gamma}(K).$$
Without loss of generality, by homogeneity, we may restrict ourselves to the case $\int F=1$. Thus the inequality of the theorem follows. $\square$ 

\medskip
\medskip

In what follows, we shall characterize the equality cases for Theorem \ref{Gauss-main-1}. Before we proceed with the formal argument, we outline a sketch of the equality case characterization in the smooth case, for the reader's convenience.

\begin{remark}[equality case characterization in the smooth case]

In order to characterize the equality cases among smooth convex sets, we notice that, firstly, $u$ (as defined in subsection \ref{ineq-proof}) must satisfy
\begin{equation}\label{u-min}
\int |\nabla u|^2=T^F_{\gamma}(K),	
\end{equation}
and, secondly, the equality must occur in the Proposition  \ref{propGauss}. (We note also that assuming sufficient smoothness about $L$ as well as $K$ allows us to use elliptic regularity results, see e.g. Evans \cite{evans}, which would ensure sufficient smoothness of $u$.) By Lemma \ref{inf-sup}, and in view of Theorem \ref{exist-dir}, the equation (\ref{u-min}) is true when $u$ is the unique function satisfying $Lu=F$ and taking a constant value on the boundary of $K$. From the equality case characterization in Proposition \ref{propGauss}, we get the following cases.

\textbf{Case 1.} $u=C_1\frac{x^2}{2}+C_2.$ Then $F=C_1(n-x^2),$ and therefore, $K$ is a euclidean ball: indeed, the only situation in which the function $u=C_1\frac{x^2}{2}+C_2$ takes a constant value on the boundary of $K$ is when $K$ a euclidean ball.

\textbf{Case 2.} There exists a rotation $R$ such that $RK=L\times \R^{n-k}$, for some $k-$dimensional symmetric convex set $L\subset \R^k,$ and $u(Rx)=\sum_{i=1}^n \alpha_i x^2_i$, for some real numbers $\alpha_1,...,\alpha_n,$ such that $\alpha_1=...=\alpha_{n-k}.$ Then the function $\alpha\sum_{i=1}^k x_i^2+\sum_{i=n-k+1}^n \alpha_i x^2_i$ takes constant values on the boundary of $RK=L\times \R^{n-k}$, and this implies that the set $L$ must be a $k-$dimensional euclidean ball.

This shows that the only possible equality cases for Theorem \ref{Gauss-main} are $k$-round cylinders. The fact that they are, in fact, equality cases, is the statement of Proposition \ref{cyl}, which we already shown.
\end{remark}

\medskip

\subsection{Stability in the key estimate}

We deduce immediately from the stability result for the Gaussian case of the Brascamp-Lieb inequality Proposition \ref{stability}:

\begin{proposition}\label{stabil-keyprop}
Fix $\epsilon>0$. Suppose that for a symmetric convex set $K$ with $C^2$-smooth boundary and an even $u\in W^{2,2}(K)\cap C^2(K)$, we have
$$\int \|\nabla^2 u\|^2\leq \int |\nabla u|^2 +\frac{(\int Lu)^2}{n-\E X^2}+\epsilon.$$
Then either
\begin{itemize}
\item $\|u-C_1 x^2-C_2\|_{L_1(K,\gamma)}\leq \sqrt{\gamma(K)}\left(\sqrt{n\epsilon}+\sqrt[4]{n\epsilon}\right),$ for some $C_1,C_2\in\R$ or
\item there exists a rotation $R$ such that 
$$\|u(Rx)-\sum_{i=1}^n \alpha_i x^2_i-C_0\|_{L_1(K,\gamma)}\leq \sqrt{\gamma(K)}\left(\sqrt{n\epsilon}+\sqrt[4]{n\epsilon}\right),$$ for some real numbers $\alpha_1,...,\alpha_n,$ such that $\alpha_1=...=\alpha_{n-k},$ and simultaneously, there exists a vector $\theta\in \R^{n-k}$ such that 
$$\int_{\partial K} \langle \theta,n_x\rangle^2 d\gamma_{\partial K}\leq \frac{(n+1)\epsilon}{r},$$
where $r$ is the in-radius of $K.$
\end{itemize}
\end{proposition}
\begin{proof} Indeed, by the assumption, we deduce that a near-equality must occur in (\ref{333}), and the conclusion follows immediately from Proposition \ref{stability} (the stability in the Brascamp-Lieb inequality).	
\end{proof}

\medskip

In what follows, we outline an approximation argument, to get the equality cases in the class of arbitrary convex sets.

\subsection{Proof of the equality cases in Theorem \ref{Gauss-main-1}}

Suppose the equality occurs in Theorem \ref{Gauss-main-1} for some convex set $K$. Namely, suppose that for some $F\in L^2(K,\gamma)\cap Lip(K)$ with $\int F=1$,
$$p_s(K,\gamma)=2T_{\gamma}^F(K)-Var(F)+\frac{1}{n-\E X^2}.$$

As before, for a large $R>0,$ let $K_R=K\cap RB^n_2.$ For an arbitrary $\delta>0,$ consider a convex set $K^R_{\delta}$ with $C^{2}$ boundary, such that $\frac{1}{2}K^R\subset K^R_{\delta}\subset K^R$, such that $K^R_{\delta}$ is $\delta$-close to $K^R$ in Hausdorff distance, and such that $\gamma(K\setminus K^R_{\delta})\leq \delta$. Note that our assumption implies that that $d_{TV}(\gamma|_{\partial K^R}, \gamma|_{\partial K_{\delta}^R})\leq \delta$: indeed, this follows from the corresponding comparison in the Lebesgue case, which is outlined in Schneider \cite{book4}. 

Then, using Proposition \ref{continuityTRapp} and Lemma \ref{infiniteapproximation}, one may select the sequence of $R$ and $\delta$ in such a way that the torsional rigidity $T_{\gamma}^F(K^R_{\delta})$ converges to the torsional rigidity $T_{\gamma}^F(K).$ Using also the continuity of $p_s(K,\gamma)$ (which follows from the definition), and the continuity of $Var(F)$ and $\E X^2,$ we get
\begin{equation}\label{refernow}
p_s(K_{\delta}^R,\gamma)\leq 2T_{\gamma}^F(K)-Var(F)+\frac{1}{n-\E X^2}+\epsilon,
\end{equation}
for some $\epsilon>0$ that depends only on $\delta$ and $R$. Recall that by Proposition \ref{key_prop-general}, there exists a $u_{\epsilon}\in W^{2,2}(K^R_{\delta})\cap C^2(int K_{\delta}^R)$ with $Lu_{\epsilon}=F$ on $K^R_{\delta}$ such that
\begin{equation}\label{eqcaseseq1}
p_s(K_{\delta}^R,\gamma)\geq \int \|\nabla^2 u_{\epsilon}\|^2+|\nabla u_{\epsilon}|^2-Var(F).
\end{equation}

Using Lemma \ref{inf-sup}, we see that
$$\int \|\nabla^2 u_{\epsilon}\|^2+|\nabla u_{\epsilon} |^2\leq 2\int |\nabla u_{\epsilon} |^2 +\frac{1}{n-\E X^2}+\epsilon.$$
By Proposition \ref{stabil-keyprop}, either 
$$\| u_{\epsilon}-C'_{\epsilon} x^2-C_{\epsilon}\|_{L_1(K_{\delta}^R,\gamma)}\leq c(\epsilon),$$
for $c(\epsilon)\rightarrow 0$ such that $c(\epsilon)$ only depends on $\epsilon$ and $n,$ or there exists a rotation $U^{\epsilon}$ (which may depend on $\epsilon$), such that 
$$\| u_{\epsilon}(U^{\epsilon}x)-\sum_{i=1}^n \alpha_i(\epsilon) x^2_i-C_0(\epsilon)\|_{L_1(K_{\delta}^R,\gamma)}\leq c(\epsilon).$$ Here $\alpha_i(\epsilon)$ stand for some real numbers $\alpha_1(\epsilon),...,\alpha_n(\epsilon),$ where $\alpha_1(\epsilon)=...=\alpha_{n-k}(\epsilon);$ simultaneously, there exists a vector $\theta_{\epsilon}\in \R^{n-k}$ such that 
$$\int_{\partial K_{\delta}^R} \langle \theta_{\epsilon},n_x\rangle^2 d \gamma_{\partial K_{\delta}^R}\leq c''(\epsilon),$$
with $c''(\epsilon)\rightarrow_{\epsilon\rightarrow 0} 0,$ depending only on $n$ and the in-radius of $K_{\delta}^R,$ which in turn is bounded from below regardless of the values of $\delta$ and $R$ by our construction.

Recall that $d_{TV}(\gamma|_{\partial K^R}, \gamma|_{\partial K_{\delta}^R})\leq \delta$. We let $\epsilon\rightarrow 0$ (while considering appropriate subsequences), to conclude that $K^R_{\delta}$ converges to $K,$ $\theta_{\epsilon}\rightarrow \theta$, and therefore, $u_{\epsilon}$ converges weakly to a function $u\in W^{2,2}(K,\gamma)$, such that either $u(x)=\frac{\alpha x^2}{2}+C,$ or $u(x)=\langle Ax,x\rangle+C,$ for some positive-definite matrix $A$. Furthermore, whenever $\theta\in Ker(A)$, we have $\langle \theta,n_x\rangle=0$ for almost all $x\in\partial K.$ Therefore, $K$ has to be a cylinder. 

It remains to conclude that $K$ is a ``round'' $k-$cylinder. Indeed, by Proposition \ref{propGauss},
$$p_s(K_{\delta}^R,\gamma)\geq 2\int |\nabla u_{\epsilon}|^2-Var(F)+\frac{1}{n-\E X^2}.$$
Thus by (\ref{refernow}),
\begin{equation}\label{u-min-stab}
\int_{K_{\epsilon}} |\nabla u_{\epsilon} |^2\leq T^F_{\gamma}(K_{\epsilon})+\frac{\epsilon}{2},	
\end{equation}
By Proposition \ref{tor-rig-stab}, (\ref{u-min-stab}) yields that 
$$Var_{\gamma|_{\partial K}}(u_{\epsilon}) \leq c(\epsilon),$$
where $c(\epsilon)\rightarrow_{\epsilon\rightarrow 0}$ and depends only on $n$ and $r$ (which are fixed). Since 
$$d_{TV}(\gamma|_{\partial K^R}, \gamma|_{\partial K_{\delta}^R})\leq \delta,$$ 
we see, after letting $\delta\rightarrow 0$ and $R\rightarrow\infty,$ that there exists a $C\in\R$ such that $u|_{\partial K}=C$. Recall that either $u(x)=\frac{\alpha x^2}{2}+C,$ or $u(x)=\langle Ax,x\rangle+C,$ for some positive-definite matrix $A$; the fact that $u|_{\partial K}=C$ thus implies that $K$ is a round $k-$cylinder.

Finally, the fact that all round $k-$cylinders do appear as equality cases for Theorem \ref{Gauss-main-1} was shown in Proposition \ref{cyl}, and the proof is complete. $\square$

\medskip

\subsection{Some more corollaries of Theorem \ref{Gauss-main-1} and proof of the Theorem \ref{Gauss}.}

As another corollary of Theorem \ref{Gauss-main-1}, we have

\begin{cor}\label{Gauss-main-cor}
For any symmetric convex set $K$ in $\R^n,$
$$p_s(K,\gamma)\geq \sup_{J\subset\{1,...,n\}}\frac{2T_{\gamma}^{\#J-x_J^2}(K)-\E X_J^4+\left(\E X_J^2\right)^2}{(\#J-\E X^2_J)^2}+\frac{1}{n-\E X^2},$$
where $X_J=Proj_{span(e_j)_{j\in J}}X.$

Moreover, the equality holds if and only if $K=RB^k_2\times \R^{n-k}$ for some $R>0$ and $k=1,...,n.$	
\end{cor}
\begin{proof} We let $F=\frac{\#J-x^2_J}{\#J-\int x^2_J}$ (as $\int F\neq 0$ by Lemma \ref{lastlemma}), and use Theorem \ref{Gauss-main-1}.
\end{proof}

As another immediate corollary of Theorem \ref{Gauss-main-1}, obtained from it by letting $F=1$, we get

\begin{cor}\label{corT1}
For any symmetric convex set $K$ in $\R^n,$
$$p_s(K,\gamma)\geq 2T_{\gamma}(K)+\frac{1}{n-\E X^2}.$$	
\end{cor}

\begin{remark}
Theorem \ref{existsLu=-u} allows us to define, for a strictly-convex smooth set $K,$ the \emph{poincare deficit}
$$PD(K)=\inf_{v: Lv=-v}\frac{\E |\nabla v|^2-\E v^2}{(\E v)^2}+1,$$
where the expectation is with respect to the restriction of the Gaussian measure to $K$. This quantity can be viewed as an alternative to torsional rigidity. Analogously, we shall consider the ``symmetric'' Poincare deficit: for a symmetric strictly convex set $K$ in $\R^n,$ consider
$$PD_2(K)=\inf_{v: Lv=-2v}\frac{\E |\nabla v|^2-2\E v^2}{(\E v)^2}+2.$$
Our proof also yields, in fact,
$$p_s(K,\gamma)\geq 2PD_2(K)+\frac{1}{n-\E X^2}.$$
\end{remark}

\medskip

\textbf{Proof of the Theorem \ref{Gauss}.} Finally, we proceed with the proof of Theorem \ref{Gauss}. By Lemma \ref{lemma-min-p_s} and Corollary \ref{corT1}, it suffices to show that
$$
\inf_{K:\mu(K)\geq a}\left(2T_{\gamma}(K)+\frac{1}{n-\E X^2}\right)\geq \frac{\varphi^{-1}(a)^2}{4e^2n^2}+\frac{1}{n-\frac{1}{c_{n-1}a}J_{n+1}\circ J^{-1}_{n-1}(c_{n-1}a)}
$$
By Lemma \ref{moments},
$$\inf_{K:\,\gamma(K)=a}\frac{1}{\gamma(K)}\int_K x^2d\gamma(x)=\frac{1}{a} \int_{B_a} x^2 d\gamma(x),$$
where $B_a$ is the euclidean ball of Gaussian measure $a.$ Integration by polar coordinates shows that
$$\frac{1}{a} \int_{B_a} x^2 d\gamma(x)= \frac{1}{c_{n-1}a}J_{n+1}\circ J^{-1}_{n-1}(c_{n-1}a).$$
Combining this with Lemma \ref{bound2} (which provides a lower bound for $T_{\gamma}(K)$), we arrive to the conclusion $\square.$

\medskip

\subsection{A remark about a more flexible estimate}

As before, use the notation 
$$\int:=\frac{1}{\gamma(K)}\int_K d\gamma(x),$$
where $K$ is a symmetric convex body in $\R^n$ and $\gamma$ is the standard Gaussian measure on $\R^n.$ Let $C_{poin}$ be the Poincare constant of the restriction of $\gamma$ on $K.$

\begin{lemma}
For any $C^2$ even function $u:K\rightarrow\R$, and any $\lambda\in [0,\frac{1}{C^2_{poin}(\gamma,K)}]$ we have
$$\int \|\nabla^2 u\|^2+|\nabla u|^2\geq (\lambda+1)T^{F_{\lambda}}_{\gamma}(K)+\frac{(\int Lu)^2}{n-\frac{3\lambda-1}{\lambda+1}\E X^2},$$
where the expectation is taken with respect to the restriction of the Gaussian measure on $K$, and
$$F_{\lambda}=Lu-(n-x^2)\frac{\lambda-1}{\lambda+1}\frac{\int Lu}{n-\frac{3\lambda-1}{\lambda+1}\E X^2}.$$
\end{lemma}
\begin{proof} Let
$$u=v+t\frac{x^2}{2}.$$
Then, since $v$ is even,
$$\int \|\nabla^2 u\|^2+|\nabla u|^2\geq \int \|\nabla^2 v\|^2+2t\Delta v+t^2n+|\nabla v+tx|^2\geq \lambda\int \|\nabla v\|^2+2t\Delta v+t^2n+|\nabla v+tx|^2,$$
where in the last passage, the inequality is true for any $\lambda\in [0, C^{-2}_{poin}(\gamma,K)]$, because of the Poincare inequality. Next, we write, as before,
$$\Delta v=Lu-L\frac{x^2}{2}+\langle x,\nabla v\rangle,$$
and estimate the above by
$$\int (\lambda+1)|\nabla v+\frac{2\lambda t}{\lambda+1}x|^2+2t Lu+t^2(-n+\frac{3\lambda-1}{\lambda+1}x^2).$$
Lastly, we plug the optimal
$$t_0=-\frac{\int Lu}{n-\frac{3\lambda-1}{\lambda+1}\E X^2},$$
and get that 
$$\int \|\nabla^2 u\|^2+|\nabla u|^2\geq (\lambda+1)\int |\nabla u+t_0\frac{\lambda-1}{\lambda+1}x|^2+\frac{(\int Lu)^2}{n-\frac{3\lambda-1}{\lambda+1}\E X^2}.$$
It remains to note, by Lemma \ref{inf-sup},
$$\int |\nabla u+t_0\frac{\lambda-1}{\lambda+1}x|^2\geq T^{F_{\lambda}}_{\gamma}(K).$$
\end{proof}

\section{On some upgrades of the Gaussian Brascamp-Lieb inequality.}

In this section, the measure is taken to be Gaussian, a convex set $K$ is fixed, and $\E$ stands for the expected value with respect to the restriction of the Gaussian measure onto $K,$ and similarly, $Var$ and $Cov$ stand for the variance and covariance.

\subsection{Non-symmetric case}

We start by noticing the following ``upgrade'' on the Brascamp-Lieb inequality:

\begin{lemma}\label{Ehr-nonsym}
For any function $f\in W^{1,2}(K,\gamma)\cap Lip(K),$ and any $\theta\in\sfe,$
$$\E f^2\leq \E|\nabla f|^2+\left(1-\frac{\left(r_f(K)+\E\langle X,\theta\rangle\right)^2}{1-Var(\langle X,\theta\rangle)}\right)  (\E f)^2,$$
where 
	$$r_f(K)=\frac{\int_{\partial K} f\langle \theta,n_x\rangle d\gamma|_{\partial K}}{\gamma(K)\E f}.$$
\end{lemma}
\begin{proof} Apply the Brascamp-Lieb inequality to $f+t\langle\theta,x\rangle:$
$$Var(f+t\langle\theta,x\rangle)\leq \E|\nabla f+t\theta|^2,$$
and plug the optimal 
$$t=-\frac{Cov(f, \langle X,\theta\rangle)-\E\langle \nabla f,\theta\rangle}{1-Var(\langle X,\theta\rangle)},$$ 
to get	
$$Var(f)\leq \E|\nabla f|^2-\frac{\left(Cov(f, \langle X,\theta\rangle)-\E\langle \nabla f,\theta\rangle\right)^2}{1-Var(\langle X,\theta\rangle)}.$$	
Integrating by parts, and using the fact that $L\langle x,\theta\rangle=-\langle x,\theta\rangle$, we see that 
$$\frac{1}{\gamma(K)}\int_K f\langle x,\theta\rangle d\gamma=-\frac{1}{\gamma(K)}\int_K fL\langle x,\theta\rangle d\gamma=-r_f(K)\E f+ \E\langle \nabla f,\theta\rangle,$$
which finishes the proof. \end{proof}

Define 
$$\eta_{\theta}(K)=\frac{1-\E\langle X,\theta\rangle^2}{(\E\langle X,\theta\rangle)^2}.$$
Note that $\eta_{\theta}(K)\geq -1$ by the Brascamp-Lieb (\ref{BrLi-convex}) inequality applied to $\langle x,\theta\rangle;$ also, for a convex set $K$ such that for each $y\in \theta^{\perp}$, the interval $K\cap (y+span(\theta))$ contains $y$, we have +$\eta_{\theta}(K)\geq 0$, as follows from Lemma \ref{lastlemma}.

We shall notice some immediate corollaries of Lemma \ref{Ehr-nonsym}.

\begin{cor}\label{cor-ehr-nonsym} For any function $f\in W^{1,2}(K,\gamma)\cap Lip(K)$, and any $\theta\in\sfe,$

	\begin{itemize}
	\item If $r_f(K)=r_{\langle x,\theta\rangle}(K)$, then
	$$\E f^2+\eta_{\theta}(K)(\E f)^2\leq \E |\nabla f|^2.$$
	
	\item If $\int_{\partial K} f\langle n_x,\theta\rangle d\gamma|_{\partial K}=0,$ then
	$$\E f^2-\frac{\eta_{\theta}(K)}{1+\eta_{\theta}(K)}(\E f)^2\leq \E |\nabla f|^2.$$
	
	\item In particular, if $f|_{\partial K}=0,$ then
	$$\E f^2\leq \max(1,1+\eta_{\theta}(K))\E |\nabla f|^2.$$
		
	\end{itemize}
\end{cor}
\begin{proof} The first two assertions follow immediately, while the third assertion follows by the Cauchy-Schwartz inequality $(\E f)^2\leq \E f^2$, applied in the event that $\eta_{\theta}(K)\geq 0$ (we use also that $1+\eta_{\theta}(K)\geq 0.$)
\end{proof}

Next, define
$$\eta(a)=\sqrt{2\pi} a\Phi^{-1}(a) e^{\frac{\Phi^{-1}(a)^2}{2}}.$$
Note that $\eta_{\theta}(H_{\theta,a})=\eta(a),$ where $H_{\theta,a}=\{\langle x,\theta\rangle\leq \Phi^{-1}(a)\}.$ Note also that
$$\eta(a)\geq -1.$$
First, we formulate an immediate corollary of Corollary \ref{cor-ehr-nonsym}.

\begin{cor}\label{cor-ehr-nonsym-1} For any Borel-measurable set $K$ and  any function $f\in W^{1,2}(K,\gamma)\cap Lip(K)$, such that $f|_{\partial K}=0,$ we have
\begin{equation}\label{drpoin}
\E f^2-\frac{\eta(a)}{1+\eta(a)}(\E f)^2\leq \E |\nabla f|^2.
\end{equation}
\end{cor}
\begin{proof} Note that Ehrhard's symmetrization preserves Dirichlet boundary condition. By Ehrhard's principle, the worst constant $C(a)$ in the inequality (\ref{drpoin}), among the sets $K$ of Gaussian measure $a,$ is attained when $K$ is a half-space 
$$H_{a,\theta}=\{x:\langle x,\theta\rangle\leq \Phi^{-1}(a)\},$$ 
and $f$ is constant on all hyperplanes parallel to its boundary. By the second assertion of Corollary \ref{cor-ehr-nonsym-1}, we have
$$C(a)\geq\frac{\eta_{\theta}(H_{a,\theta})}{1+\eta_{\theta}(H_{a,\theta})}=\frac{\eta(a)}{1+\eta(a)},$$
where the last passage is a direct computation. \end{proof}

We can also get another curious general estimate, which shall follow from the following Lemma:

\begin{lemma}\label{favorite}
For any convex set $K$ of Gaussian measure $a,$ and any $\theta\in\sfe,$ we have $\eta_{\theta}(K)\geq \eta(a)$. In other words,
$$\E \langle X,\theta\rangle^2+\eta(a)(\E\langle X,\theta\rangle)^2\leq 1.$$	
\end{lemma}

\begin{remark} Note that Lemma \ref{favorite} is stronger that the Poincare inequality, since $\eta(a)\geq -1.$
\end{remark}

Combining Lemma \ref{favorite} with the first assertion of Corollary \ref{cor-ehr-nonsym}, we get
\begin{cor}\label{cor-ehr-nonsym-11} For any convex set $K$ and  any function $f\in W^{1,2}(K,\gamma)\cap Lip(K) $, such that for some $\theta\in\sfe,$ we have $\int_{\partial K} f\langle n_x,\theta\rangle d\gamma|_{\partial K}=\frac{1-\E\langle X,\theta\rangle^2}{\gamma(K)\E \langle X,\theta\rangle}\E f,$ we have
$$
\E f^2+\eta(a)(\E f)^2\leq \E |\nabla f|^2.
$$
\end{cor}

\medskip

The rest of this subsection is devoted to the discussion of Lemma \ref{favorite}. First, we present a quick proof of it.

\textbf{Proof of Lemma \ref{favorite} via Ehrhard's inequality.} Kolesnikov and Milman \cite{KolMil} showed that Ehrhard's inequality implies, for any function $f\in C^2(\partial K)\cap L^2(\partial K),$
\begin{equation}\label{ehrimpl}
\int_{\partial K} H_{\gamma} f^2 -\langle \rm{II}^{-1}\nabla_{\partial K} f,\nabla_{\partial K} f\rangle \, d\gamma|_{\partial K}\leq -\eta(a) \left(\int_{\partial K} f d\gamma|_{\partial K}\right)^2;
\end{equation}
See Proposition \ref{key_prop-general} in the case when $F=\Phi^{-1}$, which explains the implication. We plug $f=\langle\theta,n_x\rangle,$ and use the fact that $\nabla_{\partial K} \langle\theta,n_x\rangle =\rm{II}\theta$. We  integrate by parts (retracing back the steps done in Kolesnikov-Milman \cite{KM1}, subsection 2.2, where the generalized Reilly identity was derived from an earlier version of the Bochner-Lichnerowicz-Weitzenbock formula \cite{Lich}), and this yields Lemma \ref{favorite}. $\square$

\medskip

Let us now explain a proof of Lemma \ref{favorite} which does not rely on Ehrhard's inequality. We start by noticing 

\begin{lemma}\label{incr-psi}
Suppose a continuously differentiable function $F:\R\rightarrow\R$ is increasing. Let $G$ be given by $G'(t)=F(t)e^{-\frac{t^2}{2}}.$ Then $G\circ \Phi^{-1}$ is convex on $[0,1].$
\end{lemma}
\begin{proof} We note that 
$$(G\circ \Phi^{-1})|''_t=2\pi (G''(s)+G'(s)s)e^{s^2},$$
where $s=\Phi^{-1}(t).$ In case that $G'(t)=F(t)e^{-\frac{t^2}{2}},$ and $F$ is increasing, we have 
$$G''(s)+G'(s)s= e^{-\frac{s^2}{2}}(F(s)s-sF(s)+F'(s))\geq 0,$$
finishing the proof.	\end{proof}

We deduce

\begin{proposition}\label{prop-gen}
Let $F:\R\rightarrow\R$ be an increasing function. Then for any convex set $K,$
$$\int_K F(x_1)d\gamma\geq \int_{H_K} F(x_1)d\gamma,$$
where $H_K$ is the ``left'' half-space of the same Gaussian measure as $K,$ orthogonal to $e_1$.
\end{proposition}
\begin{proof} For $x\in e_1^{\perp},$ denote by $I_x=K\cap \{x+te_1\}.$ Parametrize $\gamma_1(I_x)=a_x$ and
$$I_x=[\Phi^{-1}(\alpha_x),\Phi^{-1}(\alpha_x+a_x)].$$
Then we write
$$\int_K F(x_1) d\gamma(x)=\int_{\R^{n-1}} \frac{1}{\sqrt{2\pi}}\int_{I_x} F(t)e^{-\frac{t^2}{2}}dt\,d\gamma_{n-1}(x)=$$$$\int_{\R^{n-1}} \left(G\circ \Phi^{-1}(\alpha_x+a_x)-G\circ\Phi^{-1}(\alpha_x) \right) d\gamma_{n-1}(x),$$
where $G$ is the antiderivative of $\frac{1}{\sqrt{2\pi}}Fe^{-\frac{t^2}{2}}$. By Lemma \ref{incr-psi}, the function $G\circ \Phi^{-1}$ is convex, and therefore
\begin{equation}\label{1}
G\circ \Phi^{-1}(\alpha_x+a_x)-G\circ\Phi^{-1}(\alpha_x)\geq G\circ \Phi^{-1}(a_x)-G\circ\Phi^{-1}(0).
\end{equation}
Again by convexity of $G\circ \Phi^{-1}$ together with Jensen's inequality, we see that

\begin{equation}\label{2}
\int_{\R^{n-1}} \left(G\circ \Phi^{-1}(a_x)-G\circ\Phi^{-1}(0)\right)d\gamma_{n-1}(x)\geq G\circ \Phi^{-1}(a)-G\circ\Phi^{-1}(0),
\end{equation}
where 
$$a=\int_{\R^{n-1}} a_x d\gamma_{n-1}(x)=\gamma(K).$$

The combination of the inequalities (\ref{1}) and (\ref{2}) finishes the proof, together with the fact that
$$G\circ \Phi^{-1}(a)-G\circ\Phi^{-1}(0)= \int_{H(K)} F(x_1)d\gamma.$$
\end{proof}

As a consequence, we get

\begin{lemma}\label{fav1}
For any convex set $K,$
$$\int_K \langle x,\theta\rangle d\gamma\geq \int_{H_{\theta}(K)} \langle x,\theta\rangle d\gamma,$$
where $H_{\theta}(K)=\{\langle x,\theta\rangle\leq \Phi^{-1}(\gamma(K))\}$. Therefore,
$$\left(\int_K \langle x,\theta\rangle d\gamma\right)^2\leq \left(\int_{H_{\theta}(K)} \langle x,\theta\rangle d\gamma\right)^2.$$
\end{lemma}
\begin{proof} The Lemma follows from the Proposition \ref{prop-gen}, applied with $F=t.$ For the second assertion, note that for any $\theta,$ we have $\int_{H_{\theta}(K)} \langle x,\theta\rangle d\gamma\leq 0$. Thus if $\int_K \langle x,\theta\rangle d\gamma\leq 0,$ the conclusion follows, and otherwise, just switch to $-\theta$ in place of $\theta.$	
\end{proof}

\begin{remark} Also, Lemma \ref{fav1} follows from the fact that the norm of the barycenter $b_K$ of a set $K$ is maximized on a half-space -- see Bobkov \cite{Bobkov-bary}, Neemann \cite{Neeman-notes}. Indeed,
$$(\E\langle X,\theta\rangle)^2=\langle b_K,\theta\rangle^2\leq |b_K|^2\leq |b_{H_{\theta}(K)}|^2,$$
where choosing $\theta$ collinear to the barycenter finishes the proof of Lemma \ref{fav1}. 
\end{remark}

Next, we outline the second ingredient for Lemma \ref{favorite}:

\begin{lemma}\label{fav2}
For any convex set $K,$ using the same notation as before,
$$\int_K \langle x,\theta\rangle^2 d\gamma\leq \int_{H_{\theta}(K)} \langle x,\theta\rangle^2 d\gamma.$$
\end{lemma}

In order to show Lemma \ref{fav2}, we start with the following variant of the Hardy-Littlewood Lemma (which we already explored in Lemma \ref{HLrinc}):

\begin{lemma}[another twist on a one-dimensional Hardy-Littlewood principle]\label{HLrinc-1}
Let $K\subset\R$ be a connected set and suppose $f: K\rightarrow\R$ is a non-negative measurable function in $L^2(\gamma, K)$ with connected level sets. Then
$$\int_K fx^2 d\gamma_1(x)\leq \int_{H_K} f^* x^2 d\gamma_1(x).$$
\end{lemma}
\begin{proof} We write
$$\int_K fx^2 d\gamma(x)=\int_K \int_0^{\infty} \int_0^{\infty} 1_{\{f(x)>t\}} 1_{\{x^2>t\}}dtds d\gamma_1(x)=$$$$\int_0^{\infty} \int_0^{\infty}\gamma_1\left(\{f>s\}\cap\{x^2>t\}\right)ds dt,$$
and similarly,
$$\int_{H_K} f^*x^2 d\gamma(x)=\int_0^{\infty} \int_0^{\infty}\gamma\left(\{f^*>s\}\cap\{x^2>t\}\right)ds dt.$$
Thus it suffices to notice that for any connected set $A\subset \R$ and any symmetric interval $S=\{|x|\leq \alpha\}$,
$$\gamma(A\cap S)\geq \gamma(H_A\cap S),$$
where we use notation $H_A=(-\infty, \Phi^{-1}(\gamma_1(A))].$ \end{proof}

As a consequence, we get

\textbf{Proof of Lemma \ref{fav2}.} Without loss of generality let $\theta=e_1.$ Let
$$f(t)=\gamma_{n-1}(K\cap \{x_1=t\}).$$
Then, by Fubbini's theorem, $\int_{\R} f(s)d\gamma_1(s)=\gamma(K).$ Since $K$ is convex, the function $f$ is log-concave, and, in particular, it has connected level sets. Thus, by Lemma \ref{HLrinc-1},
$$\int_{\R} f(x)x^2 d\gamma_1(x)\leq \int_{\R} f^*(x) x^2 d\gamma_1(x),$$
meaning that
$$\int_K x_1^2 d\gamma\leq \int_{H_K} x_1^2 d\gamma,$$
where, as before $H_K$ is a left half-space of measure $\gamma(K).$ $\square$

Finally, we outline

\textbf{Proof of Lemma \ref{favorite}.} The Lemma follows immediately from the Lemmas \ref{fav1} and \ref{fav2}. $\square$ 

\begin{remark} It would be interesting to find the proof of Lemma \ref{favorite} which does not rely on symmetrizations. Since Lemma \ref{fav1} already does not rely on them, it would suffice to find a proof of Lemma \ref{fav2} which does not use the symmetrization technique. In the case when $K$ is fully contained in a half-space whose boundary contains the origin, one may use Proposition \ref{prop-gen} with $F=-t^2$ (which is increasing in this case), to deduce Lemma \ref{fav2}. Also, if $K$ is such that $\gamma_1(K\cap \{y+te_1\})\leq \frac{1}{2}$ for all $y\in\R^{n-1},$ an elementary argument similar to Lemma \ref{HLrinc-1} gives Lemma \ref{fav2}. But it is not clear to the author how to deduce Lemma \ref{fav2} without using Ehrhard's symmetrizations.
\end{remark}

\medskip

\subsection{Symmetric case}

For a symmetric convex set, we notice

\begin{proposition}\label{prop-sym-str}
	For a symmetric convex set $K$ and a function $f\in W^{1,2}(K,\gamma)\cap Lip(K)$, we have 
$$\int_K f^2 d\gamma-\frac{C(K)}{\gamma(K)}\left(\int_K f d\gamma\right)^2\leq\frac{1}{2}\int_K |\nabla f|^2 d\gamma,$$
where
$$C(K)=1-\frac{(s_f(K)-1)^2(n-\E X^2)^2}{2\E X^2-Var(X^2)}\leq 1.$$
Here
$$s_f(K)= \frac{\int_{\partial K} f\langle x,n_x\rangle d\gamma_{\partial K}}{\gamma(K)(n-\E X^2)\E f}.$$
Moreover, the equality holds if and only if either 
\begin{itemize}
\item $f=C_1 x^2+C_2,$ for some $C_1,C_2\in\R$ or
\item there exists a rotation $R$ such that $RK=L\times \R^{n-k}$, for some $k-$dimensional symmetric convex set $L\subset \R^k,$ and $f(Rx)=\sum_{i=1}^n \alpha_i x^2_i$, for some real numbers $\alpha_1,...,\alpha_n,$ such that $\alpha_1=...=\alpha_{n-k}.$
\end{itemize}
\end{proposition}

We shall state an immediate corollary. 

\begin{cor}\label{cor!}
For a symmetric convex set $K$ and a function $f\in W^{1,2}(K)\cap C^2(K)$ such that $\int_{\partial K} f\langle x,n_x\rangle d\gamma_{\partial K}=0$ (in particular, for any $f$ such that $f|_{\partial K}=0$), we have 
$$\int_K f^2 d\mu+\frac{\beta(K)}{\mu(K)}\left(\int_K f d\mu\right)^2\leq\frac{1}{2}\int_K |\nabla f|^2 d\mu,$$
where
$$\beta(K)=\frac{n^2-2(n+1)\E X^2+\E X^4}{2\E X^2-Var(X^2)}.$$

Moreover, the equality holds if and only if either 
\begin{itemize}
\item $f=C_1 x^2+C_2,$ for some $C_1,C_2\in\R$ or
\item there exists a rotation $R$ such that $RK=L\times \R^{n-k}$, for some $k-$dimensional symmetric convex set $L\subset \R^k,$ and $f(Rx)=\sum_{i=1}^n \alpha_i x^2_i$, for some real numbers $\alpha_1,...,\alpha_n,$ such that $\alpha_1=...=\alpha_{n-k}.$
\end{itemize}
\end{cor}

\begin{remark} Recall that for all symmetric convex sets,
$$\E X^4\leq (\E X^2)^2+2\E X^2,$$
as was shown by Cordero-Erasquin, Fradelizi and Maurey \cite{CFM}. Therefore, $\beta(K)\geq -1.$
\end{remark}

We note also, by virtue of the Cauchy inequality $(\E f)^2\leq \E f^2,$ and by Corollary \ref{cor!}, and since $\beta(K)\geq -1$:

\begin{cor} Let $K$ be a symmetric convex set in $\R^n.$ Then for every any even function $f\in W^{1,2}(K,\gamma)\cap C^1(K)$ such that $f|_{\partial K}=0,$
$$\E f^2\leq \frac{1}{2(1+\beta(K))}\E |\nabla f|^2.$$

\end{cor}

This corollary gives another specific quantitative bound in the Dirichlet-Poincare inequality in the particular case of symmetric convex sets.

\medskip
\medskip

Proving the inequality part of Proposition \ref{prop-sym-str} is quite easy: note that the Brascamp-Lieb inequality (\ref{poinc-sect4-sym}) holds for a function $f+t\frac{x^2}{2},$ optimize in $t$ to plug
$$t=-\frac{(s_f(K)-1)(n-\E X^2)}{2\E X^2-Var(X^2)},$$
and obtain the desired inequality. However, in order to characterize the equality cases, we outline a slightly different argument, and will make use of the work we have already done in the previous section.

Define
\begin{equation}\label{alpha-def}
\alpha(K):=\frac{n(n-1)-(2n+1)\E X^2+\E X^4}{(n-\E X^2)^2}.
\end{equation}

For example, note, by Lemma \ref{1-dim-by-parts}, that 
$$\alpha(RB^n_2)=\frac{J_{n-1}(R)(n-1-R^2)}{g_n(R)}=-\varphi_n(R)\gamma(R B^n_2).$$
In particular, $\alpha(RB^n_2)\geq 0$ whenever $R\leq \sqrt{n-1}.$ Also, $\alpha(R B^n_2)\rightarrow_{R\rightarrow\infty} -\infty.$ One may also notice that $\alpha(H)=-\Phi^{-1}(a)^3e^{-\frac{\Phi^{-1}(a)^2}{2}}$ whenever $H$ is a half-space of measure $a,$ and thus $\alpha(H)\geq 0$ whenever $a\leq \frac{1}{2}.$ 

More generally, note that
$$\frac{1}{\gamma(R B^k_2\times \R^{n-1})}\alpha(R B^k_2\times \R^{n-1})=-G_{\gamma}(R B^k_2\times \R^{n-1})= $$$$\frac{1-p^s_{\gamma}(R B^k_2\times \R^{n-1})}{\gamma(R B^k_2\times \R^{n-1})}=-\varphi_k(R),$$
in the notation of Section 3.

\begin{remark}
In fact, our Conjecture \ref{theconj} yields that for any $a\in [0,1]$ there is a $k\in \{1,...,n\}$ such that for every convex symmetric set $K$ of Gaussian measure $a,$ one has $\alpha(K)\leq \alpha(R B^k_2\times \R^{n-1})$, where $\gamma(R B^k_2\times \R^{n-1})=\gamma(K)=a.$ Recall that we use notation $C_k(a)=R B^k_2\times \R^{n-1}$ when $\gamma(R B^k_2\times \R^{n-1})=a.$

Indeed, our conjecture would yield that for any symmetric convex set $K$ of measure $a$, and for any $f\in W^{1,2}(\partial K)\cap C^1(\partial K),$ 
$$\int_{\partial K} \rm{II} f^2-\langle \rm{II}^{-1}\nabla_{\partial K} f, \nabla_{\partial K} f\rangle\leq C(a) \left(\int_{\partial K} f\right)^2,$$
where $C(a)=\max_{k=1,...,n} \alpha(C_k(a)).$ It remains to observe that for $f=\langle x,n_x\rangle,$ we have
$$\alpha(K)=\frac{\int_{\partial K} \rm{II} f^2-\langle \rm{II}^{-1}\nabla_{\partial K} f, \nabla_{\partial K} f\rangle}{(\int_{\partial K} f)^2}.$$
\end{remark}

We notice a weaker upper-bound on $\alpha:$

\begin{proposition}\label{prop-alpha}
For any symmetric convex set $K$, $\alpha(K)\leq 1.$
\end{proposition}
\begin{proof} With some algebra, our aim rewrites as:
\begin{equation}\label{alpha-almost}
\E X^4-(\E X^2)^2-\E X^2-n\leq 0.
\end{equation}
Recall the result of Cordero-Erasquin, Fradelizi and Maurey \cite{CFM}, cited earlier as (\ref{poinc-sect4-sym}), which states that for all symmetric convex sets,
$$\E X^4\leq (\E X^2)^2+2\E X^2.$$
In addition, by Lemma 5.1 from \cite{KolLiv}, for every convex set containing the origin (in particular, a symmetric one), we have
$$\E X^2\leq n.$$ 
The combination of the above inequalities confirms (\ref{alpha-almost}), and therefore the proposition is proven. \end{proof}

Next, recall the boundary term from the Reilly-type identity of Kolesnikov and Milman (Proposition \ref{raileyprop}):
$$R(u):= \int_{\partial K}  (H_{\gamma} \langle \nabla u, n_x \rangle ^2 -2\langle \nabla_{\partial K} u, \nabla_{\partial K}  \langle \nabla u, n_x \rangle \rangle +\langle \mbox{\rm{II}} \nabla_{\partial K} u, \nabla_{\partial K}  u\rangle )  \,d\gamma_{\partial K} (x).$$

We shall outline

\begin{lemma}\label{R-term}
For any $C^2$ convex set $K,$	 and any $u:K\rightarrow\R$ such that $\langle \nabla u,n_x\rangle=\langle x,n_x\rangle,$ we have
$$\frac{1}{\gamma(K)}R\left(u\right)\geq \alpha(K)(n-\E X^2)^2.$$
\end{lemma}
\begin{proof} First, we somewhat retrace back the steps done in Kolesnikov-Milman \cite{KM1} (subsection 2.2), where the aforementioned generalized Reilly identity was derived from an earlier version of the Bochner-Lichnerowicz-Weitzenbock formula \cite{Lich}, and outline that
$$\int_{\partial K}  H_{\gamma} \langle \nabla u, n_x \rangle ^2 d\gamma_{\partial K} (x) = -\int_{\partial K} \langle \nabla u,x\rangle \langle \nabla u, n_x \rangle d\gamma_{\partial K} (x) +$$
\begin{equation}\label{eqreferRterm} 
\int_{\partial K}  \langle \nabla\langle \nabla u, n_x \rangle,\nabla u\rangle d\gamma_{\partial K} (x) + \int_{\partial K} \Delta_{\partial K} (u)  \langle \nabla u, n_x \rangle d\gamma_{\partial K} (x),
\end{equation}

where $\Delta_{\partial K}$ is the boundary Laplacian. 

In the case when $u=\frac{x^2}{2}$, we have $\Delta_{\partial K} (u)=n-1$ and 
\begin{equation}\label{bnd-grad}
\nabla\langle \nabla u, n_x \rangle =\mbox{\rm{II}} x,
\end{equation}
since $\mbox{\rm{II}} =dn_x$ (with a ``plus'' because $n_x$ is the outer unit normal). Thus
$$\int_{\partial K}  H_{\gamma} \langle \nabla x, n_x \rangle ^2= -\int_{\partial K} x^2\langle x, n_x \rangle d\gamma_{\partial K} (x) +$$$$ \int_{\partial K}  \langle \mbox{\rm{II}} x, x\rangle d\gamma_{\partial K} (x) + (n-1)\int_{\partial K}  \langle x, n_x \rangle d\gamma_{\partial K} (x),$$
and hence, using (\ref{bnd-grad}) again, we get
$$R\left(\frac{x^2}{2}\right)= -\int_{\partial K} x^2\langle x, n_x \rangle d\gamma_{\partial K} (x) + (n-1)\int_{\partial K}  \langle x, n_x \rangle d\gamma_{\partial K} (x).$$
Integration by parts shows that the above equals
$$\gamma(K)\left(n(n-1)-(2n+1)\E X^2+\E X^4 \right)=\gamma(K)\alpha(K)(n-\E X^2)^2.$$

Next, suppose $u:K\rightarrow\R$ is such that $\langle \nabla u,n_x\rangle=\langle x,n_x\rangle.$ Then $u=\frac{x^2}{2}+v,$ for some $v:K\rightarrow\R$ with $\langle \nabla v,n_x\rangle=0.$ It remains to note that by (\ref{eqreferRterm}) and (\ref{bnd-grad}), we have 
$$R\left(\frac{x^2}{2}+v\right)= R\left(\frac{x^2}{2}\right)+\int_{\partial K} \langle \rm{II}\nabla v,\nabla v\rangle d\gamma|_{\partial K}\geq R\left(\frac{x^2}{2}\right),$$
where in the last passage, convexity was used.
\end{proof}

Next, we formulate

\begin{proposition}\label{strong}
For any bounded $C^2$-smooth symmetric convex body $K$ and any even function $f\in W^{1,2}(K,\gamma)\cap C^1(K)$ we have 
$$\E f^2\leq \frac{1}{2}\E |\nabla f|^2+\frac{2\E f \int_{\partial K} f\langle x,n_x\rangle d\gamma_{\partial K}}{\gamma(K)(n-\E X^2)}-\left(\alpha(K)+\frac{1}{n-\E X^2}\right)(\E f)^2.$$
Moreover, if
$$\E f^2\geq \frac{1}{2}\E |\nabla f|^2+\frac{2\E f \int_{\partial K} f\langle x,n_x\rangle d\gamma_{\partial K}}{\gamma(K)(n-\E X^2)}-\left(\alpha(K)+\frac{1}{n-\E X^2}\right)(\E f)^2-\epsilon,$$
then either
\begin{itemize}
\item $\|f-C_1 x^2-C_2\|_{L_1(K,\gamma)}\leq C(\epsilon),$ for some $C_1,C_2\in\R$ or
\item there exists a rotation $R$ such that 
$$\|f(Rx)-\sum_{i=1}^n \alpha_i x^2_i-C_0\|_{L_1(K,\gamma)}\leq C'(\epsilon),$$ for some real numbers $\alpha_1,...,\alpha_n,$ such that $\alpha_1=...=\alpha_{n-k},$ and simultaneously, there exists a vector $\theta\in \R^{n-k}$ such that 
$$\int_{\partial K} \langle \theta,n_x\rangle^2 d\gamma_{\partial K}\leq C''(\epsilon),$$
where $C(\epsilon), C'(\epsilon), C''(\epsilon) \rightarrow_{\epsilon\rightarrow 0} 0.$
\end{itemize}

\end{proposition}
\begin{proof} Without loss of generality, we may assume that $K$ and $f$ are smooth enough to satisfy the assumptions of Theorem \ref{exist-n}, and therefore, there exists a function $u\in W^{2,2}(K,\gamma)\cap C^2(K)$ such that $Lu=f$ and 
$$\langle \nabla u,n_x\rangle=t\langle x,n_x\rangle,$$
where $t=\frac{\E f}{n-\E X^2}.$ Here we used the fact that 
$$\frac{1}{\gamma(K)}\int_{\partial K}\langle x,n_x\rangle d\gamma_{\partial K}(x)=n-\E X^2.$$
Then we integrate by parts and use Proposition \ref{raileyprop}, to get 

\begin{align}\label{eq-BL-key}
\nonumber\int_K f^2 d\gamma&=2\int_K f Lu d\mu-\int_K (Lu)^2 d\gamma=
\\&-2\int_K \langle \nabla f,\nabla u\rangle d\gamma+2\int_{\partial K} f\langle \nabla u,n_x\rangle d\gamma_{\partial K}	-
\int_{K} \left(||\nabla^2 u||^2+|\nabla u|^2\right) d \gamma-R(u).
\end{align}

Using Lemma \ref{R-term}, as well as
\begin{equation}\label{eqnov11}
\int_K 2\langle \nabla f,\nabla u\rangle+2|\nabla u|^2 d\gamma\geq -\frac{1}{2}\int_K  |\nabla f|^2 d\gamma,
\end{equation}
and Proposition \ref{key_prop-general}, which asserts that
$$\int_K \|\nabla^2 u\|^2 d\gamma\geq \int_K |\nabla u|^2 d\gamma+\frac{(\int Lu)^2}{n-\E X^2},$$
we conclude
\begin{equation}\label{finalimpr}
\int_K f^2 d\gamma\leq \frac{1}{2}\int_K |\nabla f|^2 d\gamma+ +2t\int_{\partial K} f\langle x,n_x\rangle d\gamma_{\partial K}-t^2 \alpha(K)\gamma(K)(n-\E X^2)^2.
\end{equation}
It remains to recall that $t=\frac{\E f}{n-\E X^2}$, to get that
$$\E f^2\leq \frac{1}{2}\E |\nabla f|^2+\frac{2\E f \int_{\partial K} f\langle x,n_x\rangle d\gamma_{\partial K}}{\gamma(K)(n-\E X^2)}-\left(\alpha(K)+\frac{1}{n-\E X^2}\right)(\E f)^2.$$
The stability estimate follows directly from Proposition \ref{stabil-keyprop}, in view of our set up.
\end{proof}

Next, we shall obtain a strengthening of Proposition \ref{strong}. 

\begin{proposition}[Strong form of Proposition \ref{strong}]\label{lastprop}
For any bounded $C^2$-smooth symmetric convex body $K$ and any even function $f\in W^{1,2}(K,\gamma)\cap C^1(K)$ we have 
$$\E f^2\leq \frac{1}{2}\E |\nabla f|^2 -(\E f)^2\left(\frac{(s_f(K)-1)^2}{1-\alpha(K)-\frac{1}{n-\E X^2}}-1\right).$$
Moreover, if
$$\E f^2\geq \frac{1}{2}\E |\nabla f|^2 -(\E f)^2\left(\frac{(s_f(K)-1)^2}{1-\alpha(K)-\frac{1}{n-\E X^2}}-1\right)-\epsilon,$$
then either
\begin{itemize}
\item $\|f-C_1 x^2-C_2\|_{L_1(K,\gamma)}\leq C(\epsilon),$ for some $C_1,C_2\in\R$ or
\item there exists a rotation $R$ such that 
$$\|f(Rx)-\sum_{i=1}^n \alpha_i x^2_i-C_0\|_{L_1(K,\gamma)}\leq C'(\epsilon),$$ for some real numbers $\alpha_1,...,\alpha_n,$ such that $\alpha_1=...=\alpha_{n-k},$ and simultaneously, there exists a vector $\theta\in \R^{n-k}$ such that 
$$\int_{\partial K} \langle \theta,n_x\rangle^2 d\gamma_{\partial K}(x)\leq C''(\epsilon),$$
where $C(\epsilon), C'(\epsilon), C''(\epsilon)\rightarrow_{\epsilon\rightarrow 0} 0.$
\end{itemize}
\end{proposition}
\begin{proof} By Proposition \ref{strong}, for an arbitrary $t\in \R,$ we have, for any convex body $K$ containing the origin and any function $f\in W^{1,2}(K,\gamma)\cap C^1(K)$,
$$\E (f+t)^2\leq \E |\nabla f|^2+\frac{2(\E f+t) \int_{\partial K} (f+t)\langle x,n_x\rangle d\gamma_{\partial K}}{\gamma(K)(n-\E X^2)}-\left(\alpha(K)+\frac{1}{n-\E X^2}\right)(\E f+t)^2.$$	
Plugging the optimal 
$$t=\frac{\E f(\alpha(K)+\frac{1}{n-\E X^2}-s_f(K))}{1-\left(\alpha(K)+\frac{1}{n-\E X^2}\right)(K)},$$
we get the first part of the statement. The stability follows from Proposition \ref{strong}, possibly with changing constants.
	
\end{proof}

\medskip

\textbf{Proof of Proposition \ref{prop-sym-str}.} The proposition follows directly from the Proposition \ref{lastprop} using the approximation argument, exactly as in Section 6. $\square$

\medskip

\begin{remark}
In a similar fashion to Lemma \ref{R-term}, one may notice that
$$\frac{1}{\gamma(K)} R(\langle x,\theta\rangle)=\E \langle X,\theta\rangle^2-|\theta|^2,$$
for any vector $\theta.$ Indeed, we note that for every $\theta\in\R^n,$
\begin{equation}\label{bnd-grad-theta}
\nabla\langle \nabla u, n_x \rangle =\mbox{\rm{II}} \theta,
\end{equation}
and $\Delta_{\partial K} \langle x,\theta\rangle=0.$ Hence
$$R(\langle x,\theta\rangle)= -\int_{\partial K} \langle x,\theta\rangle\langle \theta, n_x \rangle d\gamma_{\partial K} (x) =-|\theta|^2\gamma(K)-\int_{K} \langle x,\theta\rangle L \langle x,\theta\rangle,$$
and the claim follows from integration by parts.

Therefore, one may also derive Lemma \ref{Ehr-nonsym} in a similar fashion as its symmetric version.
\end{remark}

\end{document}